\documentclass[12pt,leqno]{amsart}
\usepackage{amsmath,amsfonts,latexsym,graphicx,amssymb,url, color}
\usepackage[pagebackref=true]{hyperref}
\usepackage{txfonts} 

\usepackage[hmargin=2.5 cm,vmargin=2.0 cm]{geometry}

\usepackage{graphicx}
\usepackage{color} 

\newcommand{\dist}{\text{dist}} %distance

\newcommand{\R}{{\mathbb R}} %%reals

\newcommand{\N}{{\mathbb N}}

\newcommand{\e}{\varepsilon} 

\newcommand{\A}{\mathbf{A}}

\newcommand{\bnu}{\vec{\nu}}
\newcommand\norm[1]{\left\| #1\right\|}

\newcommand{\M}{{\mathcal M}}
\newcommand{\calW}{{\mathcal W}}
\newcommand{\calE}{{\mathcal E}}
\newcommand{\wei}[1]{\langle #1 \rangle}
\newcommand{\sgn}{\text{sgn}}

\def\longequals{\mathbin{=\kern-2pt=}}
\def\eqdef{\mathbin{\buildrel \rm def \over \longequals}}

\newtheorem{theorem}{Theorem}[section]

\newtheorem{definition}[theorem]{Definition}
\newtheorem{remark}[theorem]{Remark}
\newtheorem{lemma}[theorem]{Lemma}

\numberwithin{equation}{section}

\newcommand{\beq}{\begin{equation}}
\newcommand{\eeq}{\end{equation}}

\definecolor{darkred}{rgb}{.70,.12,.20}

\definecolor{darkgreen}{rgb}{.20,.52,.14}

\begin{document}
\title[Self-Diffusion, Cross-Diffusion, Global Smooth Solution]{Self-Diffusion and Cross-Diffusion Equations: $W^{1,p}$-Estimates and Global Existence of Smooth Solutions}
%Global time existence  of  smooth solutions to a cross-diffusion system}

\author[Luan Hoang]{Luan T. Hoang$^{\dag}$}
\address{$^\dag$ Department of Mathematics and Statistics, Texas Tech University, Box 41042, Lubbock, TX 79409--1042, U.S.A.}
\email{luan.hoang@ttu.edu}

\author[Truyen Nguyen]{Truyen V. Nguyen$^\ddag$} 
\thanks{Truyen Nguyen gratefully acknowledges the support provided by  NSF grant DMS-0901449.}
\address{$^\ddag$Department of Mathematics, University of Akron, 302 Buchtel Common, Akron, OH 44325--4002, U.S.A}
\email{tnguyen@uakron.edu}

\author[Tuoc Phan]{Tuoc V. Phan$^{\dag\dag,*}$}
\address{$^{\dag\dag}$ Department of Mathematics, University of Tennessee, Knoxville, 227 Ayress Hall, 1403 Circle Drive, Knoxville, TN 37996, U.S.A. }
\email{phan@math.utk.edu}
%\thanks{$^\dag$Corresponding author}

\address{$^*$Corresponding author}

\date{\today}

\begin{abstract} 
We investigate the global time existence of smooth solutions for the Shigesada-Kawasaki-Teramoto system of cross-diffusion equations of two competing species in population dynamics. If there are self-diffusion in one species and no cross-diffusion in the other, we show that the system has a unique smooth solution for all time in bounded domains of any dimension.
We obtain this result by deriving global $W^{1,p}$-estimates of Calder\'{o}n-Zygmund type  for a class of nonlinear reaction-diffusion equations with self-diffusion. These estimates are achieved  
by employing Caffarelli-Peral perturbation technique
together with a new two-parameter scaling argument.

% To establish the result, we 
%derive global Calder\'{o}n-Zygmund type estimates %for a class of nonlinear reaction-diffusion %equations with self-diffusion by combining %Caffarelli-Peral perturbation technique with a %particular two-parameter scaling argument.
% We investigate the global time existence of smooth solutions for the Shigesada~-~Kawasaki~-~Teramoto 
% system of cross-diffusion equations of two competing species in population dynamics. 
% Our result shows that  if there is self-diffusion in one species and no cross-diffusion in the other, then the system has a unique 
% non-negative smooth solution  globally in time for smooth open bounded domains of any dimensions. To establish the result, we 
% derive Calder\'{o}n~-~Zygmund type estimates for a class of nonlinear reaction-diffusion equations with self-diffusion using  
% Caffarelli~-~Peral perturbation technique.
% 
% NEW: We investigate the global time existence of smooth solutions for the Shigesada-Kawasaki-Teramoto system of cross-diffusion equations of two competing species in population dynamics. 
% If there are self-diffusion in one species and no cross-diffusion in the other, then the system,  in a smooth bounded domain of any dimension greater than $1$, is proved to have a unique non-negative smooth solution  for all time. To establish the result, we 
% derive Calder\'{o}n-Zygmund type estimates for a class of nonlinear reaction-diffusion equations with self-diffusion by combining Caffarelli-Peral perturbation technique with a particular two-parameter scaling argument.
\end{abstract}

\maketitle

\tableofcontents

%\setcounter{equation}{0}

%\myclearpage
%\begin{document}

\section{Introduction and Main Results}\label{Intro}
Let $\Omega$ be a bounded open set in $\mathbb R^n$ with $n \geq 2$. 
We consider the 
following popular system of reaction-diffusion equations: %with self-diffusion and %cross-diffusion of 
%two competing species in %$\Omega$:
 \begin{equation} \label{KST}
 \left \{ \  
\begin{array}{lcll}
 u_t  & = & \Delta[(d_1 + a_{11} u + a_{12} v)u] + u(a_1- b_1u-c_1v) & 
\mbox{in}\quad \Omega \times (0,\infty), \\ 
v_t & = & \Delta[(d_2 + a_{21}u + a_{22}v)v] + v(a_2-b_2u-c_2v) & \mbox{in}\quad \Omega \times (0, \infty),\\ 
 \begin{displaystyle}
\frac{\partial u}{\partial \bnu} 
\end{displaystyle}
& = & 
\begin{displaystyle}
\frac{\partial v}{\partial \bnu} 
\end{displaystyle}
 = 0 &  \mbox{on}\quad \partial \Omega \times (0, \infty), \\ 
u(x,0) & = & u_0(x) \geq 0, \quad v(x,0) = v_0(x) \geq 0 \quad & \mbox{in} \quad \Omega,
\end{array}  \right . 
\end{equation}
where the coefficients $a_k, b_k ,c_k, d_k$ are positive constants, while $a_{ik}$ are non-negative constants, for $i, k= 1,2$. Hereafter, $\bnu(\cdot)$ denotes the unit outward normal vector field on the boundary $\partial\Omega$.

The system \eqref{KST} was proposed by Shigesada-Kawasaki-Teramoto in \cite{SKT} to model the spatial segregation of two competing species in the region $\Omega$. It is usually referred to as the SKT system of cross-diffusion equations.  In \eqref{KST}, $u$ and $v$ are the population 
densities of the two species. 
The terms $d_1 \Delta u, d_2\Delta v$ are the diffusion ones due to the random movements of individual species with positive diffusion rates $d_1, d_2$. Meanwhile, $\Delta[(a_{11}u ~+~a_{12} v)u]$ and $\Delta [(a_{21}u~+~a_{22}v)v]$ come from the directed movements of the individuals toward  favorable environments.
The considered species hereby move away from the high population density to avoid the population pressure, hence $a_{ik}$ are non-negative. The constants $a_{11}$, $a_{22}$ are 
called self-diffusion coefficients, while $a_{12}$ and $a_{21}$ are cross-diffusion coefficients.  The homogeneous Neumann boundary conditions mean that there are no movements across the boundary. 
We note that the zero order nonlinearities in \eqref{KST} are reaction terms of the standard 
Lotka-Volterra competition type or Fisher-Kolmogorov-Petrovskii-Piskunov reaction type. Also the system \eqref{KST} reduces to the well-known Lotka-Volterra system of 
predator-prey equations  when  $a_{ik} =0$ for all $i, k =1,2$.  

%Interested readers may see %\cite{W-Ni, KL, SKT, Yagib} 
%for more details on the %model's biological %background. 

The system \eqref{KST} has attracted interests of many mathematicians. 
%(see \cite{H2, H3, H, CLY, Deuring, JUKim,  LZ, LeD,  LN,  LNN,  LNW, W-Ni, KL, Poz, Shim, T, TVP,  Wiegner, Yagi, Yagib}).  
We particularly refer the interested readers to the survey paper \cite{Yamada2} and the books \cite{W-Ni, KL, Yagib}. 
The local existence of non-negative solutions is established by H. Amann 
%from his abstract theory %constructed 
in the seminal papers \cite{H2, H3}. 
This result is summarized in the following theorem.

\begin{theorem}[\cite{H2, H3, H}] \label{local-existence} 
Suppose $n \geq 2$ and $\partial \Omega$ is smooth.
Let $p_0 \in (n,\infty)$ and $u_0, v_0$ be non-negative functions in 
$W^{1,p_0}(\Omega)$. Then there exists a maximal time $t_{\textup{max}} \in (0, \infty]$  such that the 
system \eqref{KST} has a unique non-negative solution in $\Omega \times (0, t_{\textup{max}})$ with 
\[ u,~v\ \in \
C([0,t_{\textup{max}}), W^{1,p_0}(\Omega))~\cap~C^\infty(\overline{\Omega} \times (0,t_{\textup{max}})). \]
Moreover, if $t_{\textup{max}} < \infty$ then
\begin{equation} \label{global-cond}
\lim_{t \rightarrow t_{\textup{max}}^{-}} \Big [\norm{u(\cdot, t)}_{W^{1,p_0}(\Omega)} + 
\norm{v(\cdot, t)}_{W^{1,p_0}(\Omega)} \Big]=\infty.
\end{equation}
%for all $T \in (0, \infty)$, $T \leq t_{\textup{max}}$, then $t_{\textup{max}} = \infty$.
\end{theorem}
Many efforts have been made to investigate the existence globally in time of solutions for \eqref{KST}.
In some special  cases with very strong restrictions on the spatial dimension $n$ and the coefficients $d_k, a_{ik}$, $i, k =1,2$, the solutions are proved to exist globally in time 
(see \cite{CLY1, CLY, Deuring, JUKim, LeD-Mis, LeD, LN, LNN, LNW,  Shim, T, TVP, Yagi}). 
% (see \cite{CLY1, CLY, Deuring, JUKim, LeD-Mis, LeD, LN, LNN, LZ, LNW,  Shim, T, TVP,  Yagi}). 
Despite these achievements, whether this full system possesses global time solutions or 
finite time blow up solutions remains challenging and vastly open, even for $n=2$.  

In this  paper, we study the system \eqref{KST} when there are self-diffusion in one species and no cross-diffusion in the other. Specifically, we 
investigate \eqref{KST}  when $a_{11} >0$ and $a_{21} =0$: 
\begin{equation} \label{KST-r}
  \left \{ \  
\begin{array}{lcll}
 u_t  & = & \Delta[(d_1 + a_{11} u + a_{12} v)u] + u(a_1- b_1u-c_1v) & 
\mbox{in}\quad \Omega \times (0,\infty), \\ 
v_t & = & \Delta[(d_2 +  a_{22}v)v] \quad \quad \quad + v(a_2-b_2u-c_2v) & \mbox{in}\quad \Omega \times (0, \infty), \\ 
 \begin{displaystyle}
\frac{\partial u}{\partial \bnu} 
\end{displaystyle}
& = & 
\begin{displaystyle}
\frac{\partial v}{\partial \bnu} 
\end{displaystyle}
 = 0 &  \mbox{on}\quad \partial \Omega \times (0, \infty), \\ 
u(x,0) & = & u_0(x) \geq 0, \quad v(x,0) = v_0(x) \geq 0 \quad & \mbox{in} \quad \Omega. 
\end{array}  \right . 
\end{equation}
The system \eqref{KST-r} was studied in \cite{CLY1, CLY, LeD-Mis, LeD, LN, LNN, LNW, TVP, Tuoc} where the global time existence is established either with some restrictive conditions on the coefficients or for small $n$.  
For the latter, the result is proved by Lou-Ni-Wu \cite{LNW} for $n=2$, by Le-Nguyen-Nguyen \cite{LNN} and Choi-Lui-Yamada \cite{CLY} for $n\leq 5$, and by Phan \cite{Tuoc} for $n\leq 9$.
%the current best result requires $n<10$, see \cite{Tuoc}. 
%(More known results are %surveyed in \cite{Yamada2}.) 
However,   whether the solution of the system \eqref{KST-r} exists globally in time for every dimension $n$ is still a well-known open problem.
This question is on the list of open problems made by Y. Yamada in \cite{Yamada2}. One  main purpose  of the current paper  is to give it an affirmative answer. Precisely, we prove the following result:
%which is the first main result of the paper.
%========
\begin{theorem} \label{global-existence} 
Suppose $n \geq 2$ and $\partial \Omega$ is smooth.
Let $a_{11} >0$ and $u_0, v_0$ be non-negative functions in $W^{1,p_0}(\Omega)$ for some $p_0 > n$.
Then the system \eqref{KST-r} possesses  a unique, non-negative global solution $(u,v)$ with 
 \[
 u, v \in  C([0,\infty), W^{1,p_0}(\Omega)) \cap C^\infty(\overline{\Omega} \times (0,\infty)).
 \] 
\end{theorem}

Let us discuss the main difficulties and our strategy of proving Theorem \ref{global-existence}. Thanks to Theorem~\ref{local-existence}, it is sufficient  to show that 
 condition \eqref{global-cond} for finite time blowup does not happen. It is known that this task could be  achieved if one can  
obtain  $L^\infty$-estimates for the solutions $u$ and $v$  in finite time intervals. As there exists the maximum principle for the second equation in \eqref{KST-r}, the central issue
is to establish the boundedness for $u$. For this, the maximum principle is naturally of our first consideration.
Unfortunately, such maximum principle is not available for the system and this presents a serious obstacle. 
%Even so, one may find a transformation to convert the equations into some that a certain form of maximum principle can be applied. 
%This approach is used in \cite{LeD, LN, TVP}. Certainly, it is not universal and requires some restrictive conditions on the coefficients.  

One possible approach to get around the lack of the maximum principle for the system is to exploit the first equation in \eqref{KST-r} 
to get $L^p$-estimates for $u$ for sufficiently large $p$. Since the Laplacian  term in this equation can be expressed as 
$\nabla\cdot  [ (d_1 + 2 a_{11} u + a_{12} v)\nabla u] + a_{12} \nabla\cdot  [u\nabla v]$, the approach is  only plausible if one is able to show that
 $\nabla v\in L^p$ for large $p$. However, this type of  gradient estimates for $v$ is essentially not known. We would like to stress that
the classical Sobolev regularity theory 
\cite{GT, La, Lib} as well as its very recent developments \cite{B1, B2, BW1,  CP, PS} cannot be applied to get  
 $W^{1,p}$-estimates for $v$ due to the nonlinear structure in the second equation in \eqref{KST-r}.
In previous studies, many authors tried to avoid dealing  with this key issue  by using  De Giorgi-Nash-Moser techniques to establish $C^\alpha$-regularity for $v$ first.
However for nonlinear equations of reaction-diffusion type in  \eqref{KST-r}, this also requires the establishment of a priori $L^p$-estimate for $u$ for some  $p>(n+2)/2$.
In general, obtaining such $L^p$-estimate  for $u$ is not known and challenging unless one assume that  $n\leq 9$. This is the main reason that limits the known works such as \cite{CLY, LeD-Mis, LNW, LNN, TVP, Tuoc,  Yagi} to small dimension $n$ only. 

Our purpose  is to tackle directly the problem of obtaining $L^p$-estimates for  $\nabla v$ in terms of  $L^p$ norms of $u$. We establish new global $W^{1,p}$-estimates of Calder\'{o}n-Zygmund type that are suitable for the scalar nonlinear diffusion equation appearing in \eqref{KST}.  
This is our second goal of the paper which is also a topic of independent interest in view of recent developments in \cite{
B1, B2, BW1, CP, GT, LeD-ho, La, Lib, PS}.
 Not only does it help to prove Theorem~\ref{global-existence}, we believe that our result on $W^{1,p}$-estimates  also gives some insight into the structure of equations in \eqref{KST-r} that is not known before. For the scaling and transformation invariant reason that will be explained below, we study equations in more general form than the one in \eqref{KST-r}.

For any fixed $T>0$, we consider the following class of nonlinear parabolic 
equations:
\begin{equation}\label{vsys}
\left \{
\begin{array}{lcll}
u_t  &=&  \nabla\cdot[(1+\alpha \lambda u)\A \nabla u] + \theta^2 u(1-\lambda u) - \lambda \theta c u  \quad &\text{in}\quad \Omega_T:=\Omega\times(0,T], \\
\begin{displaystyle}
\frac{\partial u}{\partial \bnu} 
\end{displaystyle}
 &=& 0 \quad &\text{on}\quad \partial\Omega\times (0,T), \\
u(\cdot,0) & =& u_0(\cdot)  \quad &\text{in}\quad \Omega,
\end{array}\right.
\end{equation}
where  $\alpha \geq 0$, $\theta,\lambda>0$ are constants, and $c(x,t)$ is a non-negative measurable function.
We also assume that
\begin{equation}\label{elipticity}
\left\{
\begin{split}
 & \A=(a_{ij}): \Omega_T \to  \M^{n\times n}\ \text{is symmetric, measurable, and there exists} \ \Lambda>0 \mbox{ such that:} \\
& \Lambda^{-1} |\xi|^2 \leq \xi^T \A(x,t) \xi  \leq \Lambda |\xi|^2\quad \mbox{for almost every\ $(x,t) \in\Omega_T$ and all }\xi\in\R^n.
\end{split} \right.
\end{equation}
Here $\mathcal M^{n\times n}$ is the linear space of $n\times n$ matrices of real numbers.
Our goal is to derive global $W^{1,p}$-estimates for weak solution $u$ of \eqref{vsys} for a general class of $\A$ and a general domain $\Omega$. To state the result, we need the following definitions.
%====================
\begin{definition} Given  $ R>0$. Let $\A$ be a function from $\Omega_T$ to $\mathcal M^{n\times n}$. We define
% say that $\A$ belongs to the class  $\calA_{ %\delta,R}(\Omega)$ if 
\begin{equation*}
[\A]_{BMO(R, \Omega_T)}=\sup_{0<\rho\leq R}\sup_{(y,s)\in \overline{\Omega_T}} \frac{1}{|K_\rho(y, s)|}\int_{K_\rho(y, s)\cap\Omega_T}{|\A(x, t) - \bar{\A}_{B_\rho(y)\cap\Omega}(t)|^2\, dx dt},
\end{equation*}
where $K_\rho(y,s) = B_\rho(y)\times (s-\rho^2, s]$ is a parabolic cube and $\bar{\A}_{U}(t) = \fint_{U}{\A(x,t )\, dx}$.
\end{definition}
%====================
%====================
\begin{definition} For $\delta, R>0$, we say that $\Omega$ is $(\delta,R)$-Lipschitz  if for every $x_0\in \partial \Omega$ 
%and every $r\in (0,R]$, 
there exists a Lipschitz continuous
function $\gamma:\R^{n-1}\to \R$ such that - upon relabeling and reorienting the coordinate axes - we have
\[\Omega\cap B_R(x_0) =\big\{(x', x_n)\in B_R(x_0):\, x_n > \gamma(x')\big\}
\]
and
\[
\text{Lip}(\gamma) \eqdef\sup\Big\{ \frac{|\gamma(x') -\gamma(y')|}{|x'-y'|} : (x',\gamma(x')),\, (y',\gamma(y'))\in B_R(x_0),  x'\neq y'\Big \}\leq \delta.
\]
\end{definition}
%====================
Our main result on the regularity of solutions to \eqref{vsys} is the following theorem:
%====================
\begin{theorem}\label{vreg} 
Let  $n \geq 2$,  $T>0$,  $ \alpha \geq 0, \lambda>0$, $0<\theta\leq 1$, and $\A$ satisfy \eqref{elipticity}. Assume that $c$ is a non-negative function in $L^p(\Omega_T)$ for some $p > 2$.  There exists a number $\delta =\delta(p, R,\Lambda,\alpha, n)>0$ such that  if $\Omega$ is $(\delta,R)$-Lipschitz and $[\A]_{BMO(R, \Omega_T)} \leq \delta$, then any weak solution $u$ of the problem \eqref{vsys} with $0 \leq u \leq \lambda^{-1} $ in $\Omega_T$ satisfies
\begin{equation}\label{global-estimate-Lipschitz}
\int_{\Omega\times [\bar t, T]}|\nabla u|^p \, dx dt\leq C\left\{  \Big(\frac{\theta}{\lambda} \vee  \|u\|_{L^2(\Omega_T)}\Big)^p + \int_{\Omega_T}|c|^p \, dx dt\right\}
\end{equation}
for every $\bar t\in(0,T)$.  Here $C>0$ is a constant depending only on  $\Omega$, $\bar t$, $p$, $R$, $\Lambda$,  $n$ and $\alpha$, but independent of $\theta, \lambda$.
\end{theorem}

We remark that the condition $0 \leq u \leq \lambda^{-1}$ is natural to ensure that the equation is uniformly parabolic, and 
%comes from the comparison principle 
%(Lemma~\ref{CP} below), and
 is not restrictive for applications (see Lemma~\ref{Lp-v}). It is also worth mentioning that $W^{1,p}$-estimates 
for {\it linear} parabolic equations are obtained in \cite{B1, B2, BW1}.
%To see what we mean by weak solutions, we refer to Definition \ref{weak-def}.

The proof of Theorem~\ref{vreg} is given in Section~\ref{regpart}. We employ the perturbation technique  introduced by  Caffarelli-Peral   \cite{CP} for equations in divergent form. Similar approach is also used in  \cite{B1, B2, BW1,PS}. This technique is a variation of the method developed by Caffarelli \cite{C} for fully nonlinear uniformly elliptic equations (see also \cite{CC}).
We note that the second equation in \eqref{KST-r} is not invariant with respect to 
the scalings $u(x,t)\rightarrow s^{-1}u(sx, s^2t)$ and $u(x,t)\rightarrow r^{-1} u(x,t)$ for 
$s, r>0$. It is also not invariant with respect to the transformation that flattens the boundary of $\Omega$. This presents a serious problem in establishing 
global $W^{1,p}$-estimates without assuming any smallness condition on the relevant functions. We handle this by introducing  the pair of constants $\lambda, \theta$ and the coefficient matrix $\A$ into  \eqref{vsys} to ensure that this class of equations is invariant under the mentioned scalings and transformation. The parameters $\lambda$ and $\theta$   play a key role in our approach.  On the other hand, this creates technical difficulties in obtaining approximation estimates that are uniformly in both $\lambda$ and $\theta$  (Lemmas \ref{lm:compare-solution} and  \ref{lm:G-compare-solution}).  We overcome this by delicate analysis combining compactness argument with energy estimates. 
% Another new ingredient in the %proof of Theorem~\ref{vreg} is %the perturbation of the zero %order term $c(x,t)$. This is %different from \cite{B1, B2, BW1} %and the current work seems to be %the first that establishes the %$W^{1,p}$-regularity result using %such kind of perturbation. The %reference equations used in this %perturbation argument, see %\eqref{REQ}, are also nonlinear %diffusion ones; we employ the De %Giorgi-Nash-Moser techniques to %obtain their required regularity %estimates (Lemmas~\ref{W1infty-
%est} and \ref{G-W^{1,infty}-%est}).

Next, we outline  our strategy for proving Theorem~\ref{global-existence}. First note that the equation of 
$v$ in \eqref{KST-r} can be written in the form \eqref{vsys}. 
Therefore, if $u \in L^p(\Omega_T)$, for some $p > 2$, we can apply Theorem~\ref{vreg} to derive the 
$L^p$-estimates for $\nabla v$. Using this new information in the equation of $u$, we establish the $L^q$-estimates for $u$ with some  $q >p$ depending on $p$. We then repeat the process using the 
improved estimate $u\in L^q$ and applying Theorem~\ref{vreg} to the equation of $v$ to gain $\nabla v\in L^q$, and so on. 
With such iteration, we are able to obtain $\nabla v, u \in L^q$ for sufficiently large $q \in (2,\infty)$. Combining this with the classical regularity and the known results in \cite{LeD-Mis, TVP, Tuoc}, we derive a contradiction to \eqref{global-cond}. The full proof of Theorem \ref{global-existence} will be given in Section~\ref{sys}.

We close the introduction by noting that the partial differential equations in \eqref{KST-r} can be rewritten 
in the following divergence form:
\begin{equation} \label{gen.eqn}
\vec{u}_t = \nabla\cdot [J(x,t, \vec{u}) \, \nabla \vec{u}] + f(x,t, \vec{u}),
% \quad \text{in} \quad \Omega \times (0,\infty), 
\end{equation}
where 
\[
\vec{u} =
\begin{bmatrix}
u\\
 v
\end{bmatrix}
 \quad \text{and} \quad  
J(x,t, \vec{u})= 
\begin{bmatrix}
d_1 + 2 a_{11} u + a_{12} v &  a_{12} v \\
0 & d_2  + 2a_{22}v
\end{bmatrix}.
\]
Equations of the general form \eqref{gen.eqn} appear frequently in many areas of physical and biological applications  with different types of nonlinearities for $J$ and $f$ (see, for examples, \cite{La, W-Ni, KL, Vazquez, Yagib}). In the simple case when $J$ is independent of $\vec{u}$, 
they become the standard reaction-diffusion equations and have been studied extensively in the theory of parabolic equations (see \cite{La, Lib}). In our case, the dependence of $J$ on $\vec{u}$ creates mathematical and physical interesting phenomena and great technical complications.
%due to the degeneracy of the nonlinear parabolic equations.
%Many other interesting mathematical models use different types of nonlinearities for $J$ and $f$ (see, e.g., \cite{Bardos,  La, Lib, W-Ni, KL, Vazquez, Yagib}).  
Although we  focus only on the explicit system \eqref{KST-r}, the method in this paper might be extended to study general systems  of form \eqref{gen.eqn} with some structural conditions on $J$ and $f$.

\section{Regularity of Solutions to Self-Diffusion Equations}\label{regpart}
\subsection{Existence and uniqueness of weak solutions}\label{RefEqn}

This subsection proves the existence and uniqueness of solution of \eqref{vsys}. 
We first introduce some notation.  Let $\Omega \subset \mathbb{R}^n$ be an open bounded 
Lipschitz domain, $T>0$ and 
$\Omega_T = \Omega \times (0, T]$. Let $\Gamma$ be a 
relatively open connected subset of $\partial \Omega$. Denote 
\[
\partial_D \Omega_T = \Gamma \times (0,T) \cup \Omega \times \{ 0\}, \quad 
\partial_N \Omega_T = (\partial \Omega\setminus \Gamma) \times (0,T), \quad 
\partial_p \Omega_T = \partial\Omega \times (0,T) \cup \Omega \times \{0\}.
\]
We also denote the following spaces 
\begin{equation*}
\hat H_{0}^1(\Omega)= \{u\in H^1(\Omega): u=0 \quad \text{on}\quad\Gamma\},
\end{equation*}
\begin{equation*}
\mathcal W(\Omega_T) = \big\{u\in L^2(0,T;H^1(\Omega)): u_t\in L^2(0,T; H^{-1}(\Omega))\big\},
\end{equation*}
\beq\label{space-hat}
\hat \calW(\Omega_T) = \big\{u\in L^2(0,T;H^1(\Omega)): u_t\in L^2(0,T;\hat H^{-1}(\Omega))\big\},
\eeq
where $H^{-1}=(H^1_0)^*$ and $\hat H^{-1} = (\hat H_{0}^1)^*$. 
Moreover, the spaces $\mathcal{W}(\Omega_T)$ and $\mathcal{\hat{W}}(\Omega_T)$ are endowed with the following norms:
\begin{equation*}
\norm{u}_{\mathcal W(\Omega_T)} = \norm{u}_{L^2(\Omega_T)} + \norm{\nabla u}_{L^2(\Omega_T)} +  \norm{u_t}_{L^2(0,T;H^{-1}(\Omega))},
\end{equation*}
\begin{equation*}
\norm{u}_{\hat W(\Omega_T)} = \norm{u}_{L^2(\Omega_T)} + \norm{\nabla u}_{L^2(\Omega_T)} +  \norm{u_t}_{L^2(0,T;\hat H^{-1}(\Omega))}.
\end{equation*}
Note that $\hat \calW(\Omega_T)\subset\mathcal W(\Omega_T)$, and $\hat \calW(\Omega_T)=\mathcal W(\Omega_T)$ when $\Gamma=\partial \Omega$. It is well-known that the embedding 
\begin{equation*}
\calW(\Omega_T) \hookrightarrow C([0, T]; L^2(\Omega))\text{  is continuous},
\end{equation*}
and  the embedding 
\beq \label{cmpAL}
\calW(\Omega_T) \hookrightarrow L^2(\Omega_T) \text{ is compact}.
\eeq 
Therefore, if $u \in\mathcal W(\Omega_T)$ then $u(\cdot, t)$ is well-defined and in $L^2(\Omega)$ 
for each $t \in [0,T]$. Since $\hat \calW(\Omega_T)\subset\mathcal W(\Omega_T)$, these statements also hold true for $\hat \calW(\Omega_T)$ in place of $\mathcal W(\Omega_T)$. Finally, for the spaces of test functions, we define 
\begin{equation*}
\calE_0(\Omega_T) = \{\varphi\in L^2(0,T;H^1(\Omega)): \varphi=0 \quad \text{on }\partial_p\Omega_T \},
\end{equation*}
\begin{equation*}
\hat \calE_{0}(\Omega_T) = \{\varphi\in L^2(0,T;H^1(\Omega)): \varphi=0 \quad \text{on }\partial_D\Omega_T \}.
\end{equation*}
\begin{definition}  \label{weak-def} 
Let $g \in \mathcal W(\Omega_T)$, $f\in L^2(0,T;\hat H^{-1}(\Omega))$ and $\A$ satisfy 
\eqref{elipticity}. Let $\alpha \geq \theta \geq 0$ and  let $c$ be a measurable function on $\Omega_T$.
\begin{itemize}
\item[{\rm (a)}] We say that $u\in\mathcal W(\Omega_T)$ is a weak solution of 
\begin{equation*} \label{eqnonly}
u_t   =  \nabla \cdot[(1+ \alpha u)\A\nabla u] + \theta u(1-  u) -  c u +f  \quad\text{in} \quad \Omega_T
\end{equation*}
if $\alpha u\nabla u,\, \theta u^2, c u \in L^2(\Omega_T)$ and
\begin{equation} \label{weakonly}
\int_0^T \langle u_t,\varphi\rangle_{H^{-1},H_0^1} dt + \int_{\Omega_T} \Big\{(1 + \alpha u) \wei{\mathbf{A} \nabla u, \nabla \varphi} -[\theta u (1-u) -cu]\varphi \Big\}dx dt
-\int_0^T \langle f,\varphi\rangle_{\hat{H}^{-1},\hat H_0^1} dt=0,
\end{equation}
for all $\varphi \in \calE_0 (\Omega_T)$.
\item[{\rm (b)}] We say that $u\in \hat \calW(\Omega_T)$ is a weak solution of 
\begin{equation} \label{eqnf}
\left \{
\begin{array}{lll}
u_t  & =  \nabla \cdot[(1+ \alpha u)\A\nabla u] + \theta u(1-  u) -  c u +f&  \quad\text{in} \quad \Omega_T, \\
u & = g & \quad \text{on}\quad \partial_D \Omega_T, \\
\begin{displaystyle}
\frac{\partial u}{\partial \bnu}
\end{displaystyle}
 & = 0, & \quad\text{on}\quad  \partial_N \Omega_T, \\
\end{array} \right.
\end{equation}
if $\,\alpha u\nabla u,\, \theta u^2, cu \in L^2(\Omega_T)$, $u - g \in \hat\calE_{0} (\Omega_T)$ and
\begin{equation} \label{weakf}
\int_0^T \langle u_t,\varphi\rangle_{\hat{H}^{-1},\hat H_0^1} dt + \int_{\Omega_T} \Big\{(1 + \alpha u) \wei{\mathbf{A} \nabla u, \nabla \varphi} -[\theta u (1-u) -cu]\varphi\Big\} dx dt 
-\int_0^T \langle f,\varphi\rangle_{\hat{H}^{-1},\hat H_0^1} dt=0,
\end{equation}
for all $\varphi \in \hat\calE_{0} (\Omega_T)$.
\end{itemize}
% with $\varphi(\cdot, T)=0$.
\end{definition}
In fact, \eqref{weakonly} is equivalent to the following variational formulation: for any $v\in  H_{0}^1(\Omega)$ and almost every $t\in(0,T)$ one has 
\begin{equation*}
\langle u_t,v\rangle_{H^{-1},H_0^1} + ((1 + \alpha u) \mathbf{A} \nabla u, \nabla v)_{L^2(\Omega)} 
= ( \theta u (1-u) -cu,v)_{L^2(\Omega)} + \langle f,v\rangle_{\hat{H}^{-1},\hat H_0^1}.
\end{equation*}
Similar equivalence applies to \eqref{weakf}. From now on, if there is no confusion, we drop 
the subscripts $H^{-1},H_0^1$ and $\hat{H}^{-1},\hat H_0^1$ for the product notation $\langle \cdot, \cdot \rangle$. 
In the statements above and calculations below, $\langle \cdot, \cdot \rangle$ is also used to denote the scalar product in $\mathbb R^n$, but its meaning is clear in the context.

We now 
can state the main theorem of this subsection.
%============
\begin{theorem} \label{existence-u} Let $\A$ satisfy \eqref{elipticity}.  
Suppose the numbers $\alpha, \theta$ are non-negative,  $c\in L^2(\Omega_T)$ is  non-negative, and  $g \in \hat{\calW}(\Omega_T)$ satisfies $0 \leq g \leq 1$. 
Then there exists a unique weak solution $u \in \hat\calW(\Omega_T)$, $0 \leq u \leq 1$ of the equation
\begin{equation} \label{main-ref.eqn}
\left \{
\begin{array}{lll}
u_t  & =  \nabla \cdot[(1+ \alpha u)\A\nabla u] + \theta u(1-  u) -  c u &  \quad\text{in} \quad \Omega_T, \\
u & = g & \quad \text{on}\quad \partial_D \Omega_T, \\
\begin{displaystyle}
\frac{\partial u}{\partial \bnu}
\end{displaystyle} & = 0, & \quad\text{on}\quad  \partial_N \Omega_T.
\end{array} \right.
\end{equation}
Moreover,   there is a constant $C = C(T,\Lambda,\alpha,\theta)$ such that 
\begin{equation} \label{energy-v}
\sup_{0 \leq t \leq T} \int_\Omega u^2(x,t) dx  + \int_{\Omega_T} |\nabla u|^2dxdt \leq 
C\Big [|\Omega|+  \norm{c}_{L^2(\Omega_T)}^2+\norm{g}_{\hat \calW(\Omega_T)}^2 \Big].
\end{equation}
\end{theorem}
%============
We prove the uniqueness first.  This plays a key role in the existence of solutions to \eqref{main-ref.eqn} and in our paper.
%============
\begin{lemma} \label{uniqueness-u} Let $\A, \alpha, \theta, c, g$  be as in Theorem~\ref{existence-u}. Then \eqref{main-ref.eqn} has at most one weak solution $u \in \hat \calW(\Omega_T)$ with 
$0 \leq u \leq 1$.
\end{lemma}
%==============
\begin{proof}  
Suppose $u_1, u_2 \in \hat \calW(\Omega_T)$  are two weak solutions 
of \eqref{main-ref.eqn} satisfying $0 \leq u_1, u_2 \leq 1$.  Let
\[
w= \Big( u_1 + \alpha \frac{u_1^2}{2}\Big) - \Big( u_2 + \alpha \frac{u_2^2}{2}\Big)
= (u_1 - u_2) \Big(1 +  \alpha \frac{u_1 + u_2}{2}\Big).
\]
For each $k\in \N$, we define the Lipschitz approximations to the $\sgn^+$ function:
\begin{equation}\label{appp-sign+}
\sgn^+_k(z) = 
\begin{cases}
1  & \mbox{ for }\quad z\geq \frac{1}{k}, \\
k z& \mbox{ for }\quad 0< z<  \frac{1}{k}, \\ 
0  & \mbox{ for }\quad z\leq 0. 
\end{cases}
\end{equation}
 Since the function $z\mapsto \sgn^+_k(z)$ is Lipschitz continuous, $w\in L^2(0,T;H^1(\Omega))$ and $w =0$ on $\partial_D \Omega_T$, we have $\sgn^+_k(w)\in L^2(0,T;H^1(\Omega))$ with $\sgn^+_k(w) =0$ 
on $\partial_D \Omega_T$. Hence by using $\sgn^+_k(w)$ as a test function in equation \eqref{main-ref.eqn} for $u_1, \,u_2$ and 
using integration by parts, one gets
\[
\begin{split}
& \int_\Omega (u_1 - u_2)_t \sgn^+_k(w) dx 
= -\int_\Omega \langle \A(x,t) \big[(1+\alpha u_1)\nabla u_1 -(1+\alpha u_2)\nabla u_2\big], \nabla [\sgn^+_k(w)]\rangle dx\\
&\quad \quad + 
\theta \int_\Omega (u_1 -u_2)(1-u_1 -u_2) \sgn^+_k(w) dx   -\int_\Omega c (u_1 - u_2)\sgn^+_k(w) dx\\
& =  -\int_\Omega \langle \A(x,t)\nabla w, \nabla w\rangle (\sgn^+_k)'(w)dx + 
\theta \int_\Omega (u_1 -u_2)(1-u_1 -u_2) \sgn^+_k(w) dx  \\
& \quad \quad \quad  -\int_\Omega c (u_1 - u_2)\sgn^+_k(w) dx.
\end{split}
\]
As $\A(x,t)$ is non-negative definite and $(\sgn^+_k)'\ge 0$, we deduce that
\[
 \int_\Omega (u_1 - u_2)_t \sgn^+_k(w) dx 
\leq \theta \int_\Omega (u_1 -u_2)(1-u_1 -u_2) \sgn^+_k(w) dx   -\int_\Omega c (u_1 - u_2)\sgn^+_k(w) dx.
\]
Letting $k \rightarrow \infty$ and observing that $\sgn^+(w)=\sgn^+(u_1 - u_2)$, we obtain 
 \[
\frac{d}{dt} \int_\Omega (u_1 - u_2)^+ dx \leq \theta \int_\Omega (u_1 - u_2)^+ dx
 \]
 yielding
 \[
  \int_\Omega \Big(u_1(x,t) - u_2(x,t)\Big)^+ dx \leq e^{\theta t} \int_\Omega \Big(u_1(x,0) - u_2(x,0)\Big)^+ dx\quad\mbox{for every } t>0.
 \]
Since $u_1 - u_2 =0$ on $\Omega \times \{ 0\}$, it follows that $(u_1 - u _2)^+ =0$ a.e. on $\Omega_T$, which gives $u_1\leq u_2$  a.e. on $\Omega_T$. By interchanging the role of $u_1$ and $u_2$, we 
 infer that $u_1 = u_2$ a.e. on $\Omega_T$.
 \end{proof}

A modification of the proof of Lemma~\ref{uniqueness-u} gives the following comparison principle:
\begin{lemma} \label{CP} Assume that $\A$ satisfies \eqref{elipticity}. Suppose that $c\in L^2(\Omega_T)$ is  non-negative,   $g \in  \calW(\Omega_T)$ and 
 $f\in L^2(0,T;\hat H^{-1}(\Omega))$. Let $u_1, u_2\in \hat \calW(\Omega_T)$ be respectively weak  sub-solution and weak super-solution to the problem
 \begin{equation*} 
\left \{
\begin{array}{lll}
u_t  & =  \nabla \cdot[\A\nabla u] -  c u +f&  \quad\text{in} \quad \Omega_T, \\
u & = g & \quad \text{on}\quad \partial_D \Omega_T, \\
\begin{displaystyle}
\frac{\partial u}{\partial \bnu}
\end{displaystyle}
 & = 0, & \quad\text{on}\quad  \partial_N \Omega_T. \\
\end{array} \right.
\end{equation*}
That is, $(u_1 - g)^+ \in \hat\calE_{0} (\Omega_T)$, $(g - u_2)^+ \in \hat\calE_{0} (\Omega_T)$ and for every function $\varphi\in \hat\calE_0(\Omega_T)$ with $\varphi\geq 0$, we have
\begin{equation*} 
\int_0^T \langle (u_1)_t,\varphi\rangle dt \leq  \int_{\Omega_T} \Big\{ -\wei{\mathbf{A} \nabla u_1, \nabla \varphi} +[\theta u_1 (1-u_1) -cu_1]\varphi\Big\} dx dt 
+\int_0^T \langle f,\varphi\rangle dt
\end{equation*}
and
\begin{equation*} 
\int_0^T \langle (u_2)_t,\varphi\rangle dt \geq  \int_{\Omega_T} \Big\{ -\wei{\mathbf{A} \nabla u_2, \nabla \varphi} +[\theta u_2 (1-u_2) -cu_2]\varphi\Big\} dx dt 
+\int_0^T \langle f,\varphi\rangle dt.
\end{equation*}
Then $u_1\leq u_2$ almost everywhere on $\Omega_T$.
 \end{lemma}
\begin{proof} 
Let $w=u_1 - u_2$. Then $w\leq 0$ on $\partial_D \Omega_T$ since $u_1\leq g \leq u_2$ on $\partial_D \Omega_T$. Hence $\sgn^+_k(w)\in L^2(0,T;H^1(\Omega))$ with $\sgn^+_k(w) =0$ 
on $\partial_D \Omega_T$, where  $\sgn^+_k$ is the function given by  \eqref{appp-sign+}. Therefore, by arguing as in the proof of Lemma~\ref{uniqueness-u} we obtain
\[
  \int_\Omega \Big(u_1(x,t) - u_2(x,t)\Big)^+ dx \leq \int_\Omega \Big(u_1(x,0) - u_2(x,0)\Big)^+ dx\quad\mbox{for every } t>0.
 \]
As $u_1\leq u_2$ on $\Omega \times \{ 0\}$, it follows that $(u_1 - u _2)^+ =0$ a.e. on $\Omega_T$. That is,  $u_1\leq u_2$  a.e. on $\Omega_T$. 
\end{proof}
%========

Next, we prove the energy estimate \eqref{energy-v}.

\begin{lemma} \label{MPv} Let $\mathbf{A}, \alpha, \theta, g, c$ be as in Theorem \ref{existence-u},  and $f\in L^2(0,T;\hat H^{-1}(\Omega))$.  
Suppose  $u \in \hat \calW(\Omega_T)$ is a weak solution of \eqref{eqnf}
%\begin{equation}
%\left \{ 
%\begin{array}{lcll}
%u_t & = & \nabla\cdot [(1+\alpha u)\mathbf{A} \nabla u] + \theta u(1-u) - c(x,t) u + f & \quad\text{in}\quad \Omega_T, \\
%u & = & g  & \quad\text{on}\quad \partial_D \Omega_T, \\
%\begin{displaystyle}
%\frac{\partial u}{\partial \bnu} 
%\end{displaystyle}
% & = & 0 & \quad\text{on}\quad \partial_N \Omega_T,
%\end{array} \right.
%\end{equation}
with $ 0 \leq u \leq 1$. Then there exists $C>0$ depending only on $T$, $\Lambda$, $\alpha$ and $\theta$ such that
\begin{equation}\label{general-energy-est} 
\sup_{0 \leq t \leq T}\int_{\Omega} u^2  dx+  \int_{\Omega_T}|\nabla u|^2 dx dt 
\leq C\Big[|\Omega| + \norm{c}_{L^2(\Omega_T)}^2 +   \norm{f}_{L^2(0,T; \hat H^{-1}(\Omega))}^2 +\norm{g}_{\hat \calW(\Omega_T)}^2\Big].
\end{equation}
\end{lemma}
%========
\begin{proof} 
Let $w= u-g$ and use this as the test function for the equation of $u$. 
We then have
\begin{align*}
\langle u_t,w\rangle = - \int_\Omega (1+\alpha u)\langle \A\nabla u,\nabla w\rangle dx +\int_\Omega \theta u (1-u) w dx - \int_\Omega c u w dx +\langle f,w\rangle.
\end{align*}
This can be  rewritten as
\begin{align*}
  \langle w_t,w\rangle + \langle g_t,w\rangle 
&= - \int_\Omega (1+\alpha u) \langle \A\nabla w,\nabla w\rangle dx -\int_\Omega (1+\alpha u) \langle \A\nabla g,\nabla w\rangle dx 
\\
&\quad  +\int_\Omega \theta (1-u) w^2 dx + \int_\Omega \theta (1-u) g w dx - \int_\Omega c u wdx +\langle f,w\rangle.
\end{align*}
Using this together with \eqref{elipticity}, Cauchy-Schwartz inequality and the assumption $0\leq u\leq 1$, we get 
\[
\begin{split}
\frac{d}{dt} \int_\Omega w^2 dx  +  \Lambda^{-1} \int_\Omega|\nabla w|^2 dx 
& \leq C \Big[ \int_\Omega g^2 dx + \int_\Omega|\nabla g|^2 dx + 
\int_\Omega c^2(x,t) dx \\
& \quad + \big(\norm{f(\cdot,t)}_{\hat H^{-1}(\Omega)} + \norm{g_t(\cdot,t)}_{\hat H^{-1}(\Omega)}  \big)\big(\norm{w}_{L^2}+\norm{\nabla w}_{L^2}\big)  + \int_{\Omega}w^2 dx  \Big],
\end{split}
\]
where $C>0$ depends only on $\Lambda$, $\alpha$ and $\theta$. Hence
\begin{align*}
&\frac{d}{dt} \int_\Omega w^2 dx  +  (2\Lambda)^{-1} \int_\Omega|\nabla w|^2 dx\\ 
& \leq C \Big[ \|c(\cdot,t)\|_{L^2(\Omega)}^2 + \norm{f(\cdot,t)}_{\hat H^{-1}(\Omega)}^2+ \Big( \|g(\cdot,t)\|_{L^2(\Omega)}^2 + \|\nabla g(\cdot,t)\|_{L^2(\Omega)}^2 +
\norm{g_t(\cdot,t)}_{\hat H^{-1}(\Omega)}^2\Big)  + \int_{\Omega}w^2 dx  \Big].
\end{align*}
Note that $\norm{w(\cdot,0)}_{L^2(\Omega)}=0.$
Then applying Gronwall's inequality yields
\begin{equation*} 
\sup_{0 \leq t \leq T} \int_\Omega w^2  dx+  \int_{\Omega_T}|\nabla w|^2 dx dt 
\leq C\Big[ \norm{c}_{L^2(\Omega_T)}^2 +   \norm{f}_{L^2(0,T; \hat H^{-1}(\Omega))}^2 +\norm{g}_{\hat \calW(\Omega_T)}^2\Big],
\end{equation*}
for some constant $C = C(T,\Lambda,\alpha, \theta)$. 
Since $0\leq g\leq 1$, we therefore obtain the estimate \eqref{general-energy-est}.
%and the proof is complete.
\end{proof}
%========
We also need the following result for  linear parabolic equations 
with mixed boundary conditions.
\begin{lemma} \label{linear-eqn-mixed} 
Let $\A$ satisfy \eqref{elipticity} and $c, f \in L^2(\Omega_T)$ with $0 \leq f \leq c$. 
Suppose $g \in \hat \calW(\Omega_T)$ with $0 \leq g \leq 1$. 
Then there exists a unique weak solution $w \in \hat\calW(\Omega_T)$ of the  problem
\begin{equation} \label{linear-eqn}
\left\{
\begin{array}{lcll}
w_t & = & \nabla\cdot[\A\nabla w]   - cw  + f & \quad\text{in}\quad \Omega_T,\\
w & = & g & \quad\text{on}\quad\partial_D\Omega_T, \\
\begin{displaystyle}
\frac{\partial w}{\partial \bnu} 
\end{displaystyle}
 &=& 0 & \quad\text{on}\quad \partial_N \Omega_T.
\end{array}
\right.
\end{equation}
Moreover, $0 \leq w \leq 1$.
\end{lemma}
\begin{proof} For each $k \in \mathbb{N}$, let $c_k = \min\{c, k\}$ and $f_k = \min\{f, k\}$.
From the standard theory of linear parabolic equations in divergence forms with bounded coefficients (see, for example, \cite[Theorem~9.9]{Salsa}), there exists a unique 
weak solution $w_k \in \hat\calW(\Omega_T)$  of the approximation  problem
\begin{equation} \label{cut-sol}
\left\{
\begin{array}{lcll}
\partial_t w_k &=& \nabla\cdot[\A \nabla w_k] -c_k w_k + f_k & \quad\text{in}\quad \Omega_T, \\
w_k & = & g & \quad\text{on}\quad \partial_D \Omega_T, \\
\begin{displaystyle}
\frac{\partial w_k}{\partial \bnu} 
\end{displaystyle}
 &=& 0 & \quad \text{on}\quad\partial_N \Omega_T. 
\end{array}
\right.
\end{equation}
Since $0\leq f_k \leq c_k$ and $0\leq g\leq 1$, we see that $u_1 =0$ is a weak sub-solution  and $u_2=1$ is a weak super-solution   
to the problem \eqref{cut-sol}. Therefore,  it follows from Lemma~\ref{CP} that
\[
0 \leq w_k \leq 1 \quad\mbox{in}\quad \Omega_T, \quad \forall k \in \mathbb{N}.
\]
Also by Lemma~\ref{MPv}, we obtain
\begin{equation*} \label{cut-enqr}
\begin{split}
\sup_{0\leq t \leq T}\int_\Omega |w_k|^2 dx + \int_{\Omega_T}|\nabla w_k|^2 dxdt & 
\leq C\Big[|\Omega|  + \norm{c_k}_{L^2(\Omega_T)}^2 + \norm{f_k}_{L^2(\Omega_T)}^2 + 
\norm{g}_{\hat \calW(\Omega_T)}^2\Big] \\
& \leq
C\Big[|\Omega| + \norm{c}_{L^2(\Omega_T)}^2   + \norm{f}_{L^2(\Omega_T)}^2 + 
\norm{g}_{\hat \calW(\Omega_T)}^2\Big]. 
\end{split}
\end{equation*}
This together with \eqref{cut-sol} and the boundedness of $w_k$ yield
\begin{equation*}
\norm{w_k}_{\hat \calW(\Omega_T)} \leq C, \quad \forall \ k \in \mathbb{N}.
\end{equation*}
By the compact embedding 
\eqref{cmpAL} and the fact that $\{c_k w_k\}$ is bounded in $L^2(0,T;\Omega_T)$, there is a subsequence, still denoted by $\{w_k\}$, and a function $w\in \hat \calW(\Omega_T)$ such that
\begin{equation*}
 \left \{ \  
\begin{array}{lll}
&w_k \to w  \mbox{ strongly in } L^2(\Omega_T),\\
&c_k w_k \rightharpoonup cw \quad\text{and}\quad \nabla w_k \rightharpoonup \nabla w \text{ weakly in } L^2(\Omega_T),\\
&\partial_t w_k \rightharpoonup \partial_t  w \text{ weakly-* in } L^2(0,T;\hat H^{-1}(\Omega_T)).
\end{array}  \right . 
\end{equation*}
Clearly, $0\leq w \leq 1$.
Now, by taking $k \rightarrow \infty$, it follows from \eqref{cut-sol} that $w$ is a weak solution of 
\eqref{linear-eqn}.  The uniqueness  of  the solution $w$ is guaranteed by Lemma~\ref{CP}.
\end{proof}
Now, we are ready to prove Theorem \ref{existence-u}.
\begin{proof}[\textbf{Proof of Theorem \ref{existence-u}}]
Thanks to Lemma~\ref{uniqueness-u} and Lemma~\ref{MPv}, it remains to prove the existence of a weak solution $u\in \hat\calW(\Omega_T)$ to  \eqref{main-ref.eqn} 
satisfying $0\leq u\leq 1$.   We give the proof for the case $\theta >0$. The case $\theta =0$ is similar and, in fact, simpler. Without loss of generality, let us assume $\theta =1$. 
The proof is based on the Schauder fixed point theorem. Alternatively, one can also use the iterative monotone method based on lower and upper solutions (see \cite{Sat}). 
Define
\[ E= \{v \in L^2(\Omega_T) :  0 \leq v \leq 1 \}, \quad f(s) = 2s - s^2,\]
and let 
\[
\mathbb{L}_v [w] = \nabla \cdot[(1+\alpha v) \A\nabla w] - (c +1) w.
\]
Note that the function $f(s)$ is increasing on $[0,1]$. Hence
\[
0\leq f(v) \leq f(1) = 1 \leq c+1, \quad \forall \ v \in E.
\]
Therefore, it follows from Lemma \ref{linear-eqn-mixed} that for each $v \in E$, there exists a unique weak solution $w \in \hat \calW(\Omega_T)$ with $0\leq w \leq 1$ of the problem
\begin{equation} \label{L-def}
\left\{
\begin{array}{lcll}
w_t & = & \mathbb{L}_v[w] + f(v) & \quad\text{in}\quad  \Omega_T, \\
w & =& g & \quad\text{on}\quad \partial_D \Omega_T,\\
\begin{displaystyle}
\frac{\partial w}{\partial \bnu} 
\end{displaystyle}
 & =& 0 &\quad\text{on}\quad \partial_N \Omega_T.
\end{array}
\right.
\end{equation} 
Thus, we can define the map $\mathcal{L} : E
\rightarrow E \subset L^2(\Omega_T)$ by $\mathcal{L}(v) = w$, for each $v\in E$,
where $w$ is the solution of \eqref{L-def}.  
It is clear that $E$ is a closed, convex set in $L^2(\Omega_T)$. We now seek for 
$u \in E$ such that $u = \mathcal{L}(u)$. By \cite[Corollary 11.2]{GT}, it suffices to show that $\mathcal{L}$ is completely continuous. Note that from 
Lemma \ref{MPv}, there is $C = C(T, \Lambda)$ such that 
\beq\label{Lvbound}
\sup_{0\leq t \leq T}\int_{\Omega}|\mathcal{L}(v)|^2 dx 
+ \int_{\Omega_T} |\nabla \mathcal{L}(v)|^2 dxdt 
\leq C\Big [1 + |\Omega| + \norm{c}_{L^2(\Omega_T)}^2 + \norm{g}_{\hat\calW(\Omega_T)}^2\Big], 
\quad \forall v \in E.
\eeq
By \eqref{L-def}, the bound \eqref{Lvbound} and the fact $0 \leq \mathcal{L}(v) \leq 1$, we have
\begin{equation} \label{pre-compact}
\norm{\mathcal{L}(v)}_{\hat\calW(\Omega_T)} \leq C, \quad \forall \ v \in E.
\end{equation}
From this and the compact imbedding \eqref{cmpAL}, we conclude that 
$\mathcal{L}(E)$ is pre-compact in $L^2(\Omega_T)$.
Therefore, it remains to show that $\mathcal{L}$ is continuous in $L^2(\Omega_T)$-topology.  
Let $\{v_k\}$ be a sequence in 
$E$ such that $v_k \rightarrow v$ strongly in $L^2(\Omega_T)$. 
For each $k \in \mathbb{N}$, let $w_k = \mathcal{L}(v_k)$, 
i.e. $0 \leq w_k \leq 1$, $w_k \in \hat\calW(\Omega_T)$ and $w_k$ is a weak solution of
\begin{equation} \label{L(v-k)-eqn}
\left \{ 
\begin{array}{lcll}
\partial_t w_k & = & \mathbb{L}_{v_k} [w_k] + f(v_k) & \quad\text{in}\quad \Omega_T, \\
w_k & = & g  & \quad\text{on}\quad \partial_D \Omega_T, \\
\begin{displaystyle}
\frac{\partial w_k}{\partial \bnu} 
\end{displaystyle}
 & = & 0 & \quad\text{on}\quad \partial_N \Omega_T.
\end{array} \right.
\end{equation}
From \eqref{pre-compact}, we have 
\begin{equation} \label{w-k-uniform}
\norm{w_k}_{\hat\calW(\Omega_T)} \leq C, \quad \forall\ k \in \mathbb{N}.
\end{equation}
Now, let $w = \mathcal{L}(v)$, i.e. $0 \leq w \leq 1$, $w \in \hat\calW(\Omega_T)$ and $w$ is the  weak solution of \eqref{L-def}. We need to prove that
\begin{equation} \label{w-k-converge}
w_k \rightarrow w\quad \text{strongly in} \quad L^2(\Omega_T).
\end{equation}

Let  $\{w_{k'}\}_{k'}$ be any subsequence of $\{w_k\}$.
From \eqref{w-k-uniform} and \eqref{cmpAL}, 
there exists a subsequence of $\{k'_m\}_{m}$ of $\{k'\}$ and $\bar w\in \hat\calW(\Omega_T)$ such that as $m\to\infty$,
\begin{equation} \label{wk-con-cont}
\left\{ 
\begin{aligned}
&\quad w_{k'_m} \to \bar{w} \text{ strongly in }  L^2(\Omega_T),\\
&\quad  \nabla w_{k'_m} \rightharpoonup \nabla \bar{w} \quad\text{and}\quad
v_{k'_m}\nabla w_{k'_m} \rightharpoonup v\nabla \bar{w} \text{ weakly in } L^2(\Omega_T),\\
&\quad  \partial_t w_{k'_m} \rightharpoonup \partial_t \bar w \text{ weakly-* in } L^2(0,T;\hat H^{-1}(\Omega)).
\end{aligned}
\right.
\end{equation}
From \eqref{L(v-k)-eqn}, \eqref{wk-con-cont}, the convergence of $\{v_k\}$, and the uniform 
boundedness of $\{v_k\}, \{w_k\}$, we find that 
$\bar{w}$ is also a weak solution of \eqref{L-def}. By the uniqueness of the solution, Lemma \ref{linear-eqn-mixed}, we 
see that $\bar{w} = w = \mathcal{L}(v)$. 
Thus $ w_{k'_m} \to w$ strongly in $L^2(\Omega_T)$. Therefore, we infer that \eqref{w-k-converge} holds and conclude that the map $\mathcal{L}$ is continuous. 
The proof is complete.
\end{proof}
%========
%=================================
\subsection{Interior $W^{1,p}$-estimates}\label{interiorW}
In this subsection we study interior regularity for solutions to \eqref{vsys}.
We  consider the case $\alpha>0$ since the case $\alpha =0$ is much simpler. 
For the purpose of brevity, we take $\alpha~=~1$  from now on.
%Note that $\alpha$ does not play any essential role %in the proof of Theorem~\ref{vreg} and 
%the proof is much simpler in the case $\alpha =0$. 
%To derive the interior estimate of 
%the gradient of  weak solutions of the equation %\eqref{vsys}, 
We thus consider the following parabolic equation
\begin{equation}\label{PEQ}
u_t  =  \nabla\cdot[(1+\lambda u)\A \nabla u] + \theta^2 u(1-\lambda u) - \lambda \theta c u
\end{equation}
in $Q_6$, where $ \lambda, \theta>0$ are constants and $c(x,t)$ is a non-negative measurable function. The coefficient matrix $\A=(a_{ij}): Q_6 \to  \M^{n\times n}$ is assumed to be
symmetric, measurable  
and there exists a constant $\Lambda>0$ such that
\begin{equation}\label{interior-elipticity}
\Lambda^{-1} |\xi|^2 \leq \xi^T \A(x,t) \xi \leq \Lambda |\xi|^2\quad \mbox{for a.e. $(x,t)\in Q_6$ and for all }\xi\in\R^n.
\end{equation}
Hereafter, $Q_\rho(x,t) \eqdef B_\rho(x)\times (t-\rho^2, t+\rho^2]$ is a centered parabolic cube and $Q_\rho \eqdef Q_\rho(0,0)$. Observe that  $u$ is a weak solution of \eqref{PEQ} in $Q_6$ iff the function 
$\bar{u}  \eqdef \lambda u$ is a weak solution of
\begin{equation}\label{RPEQ}
\bar{u}_t  =  \nabla\cdot[(1+\bar{u} )\A \nabla \bar{u} ] + \theta^2 \bar{u} (1- \bar{u} ) - \lambda \theta c \bar{u}  \quad\mbox{in}\quad Q_6.
\end{equation}

We are going to derive interior $W^{1,p}$-estimates for solutions of \eqref{PEQ} by freezing its coefficient and comparing it to solutions  of the equation
\begin{equation}\label{REQ}
v_t  =  \nabla\cdot[(1+\lambda v)\bar{\A}_{B_4}(t)\nabla v] + \theta^2 v(1-\lambda v)\quad\mbox{in}\quad Q_4,
\end{equation}
where $\bar{\A}_{B_4}(t)$ is the average of $\A(\cdot,t)$ over $B_4$, that is, $\bar{\A}_{B_4}(t) :=\fint_{B_4}{\A(x,t) dx}$.
Notice that $v$ is a weak solution of \eqref{REQ} iff the function 
$\bar{v}  \eqdef \lambda v$ is a weak solution of 
\begin{equation}\label{reqv}
\bar{v}_t  =  \nabla\cdot[(1+\bar{v})\bar{\A}_{B_4}(t)\ \nabla \bar{v}] + \theta^2 \bar{v}(1-\bar{v})\quad\mbox{in}\quad Q_4.
\end{equation}
Our main interior regularity result  states as follows: 
\begin{theorem}\label{first-main-result} 
Assume that $ \lambda>0$, $0<\theta\leq 1$, $\A$ satisfies \eqref{interior-elipticity} and $c\in L^2(Q_6)$. 
For any $p>2$, there exists a constant $\delta=\delta(p, \Lambda, n)>0$ such that if 
\begin{equation*}
\sup_{0<\rho\leq 4}\sup_{(y,s)\in Q_1} \fint_{Q_\rho(y,s)}{|\A(x,t) - \bar{\A}_{B_\rho(y)}(t)|^2\, dx dt}\leq \delta,
\end{equation*}
and $u\in \mathcal W(Q_6)$ is a weak solution of \eqref{PEQ}   satisfying
$0 \leq u \leq \frac{1}{\lambda}\, \mbox{ in } Q_5$,
%\end{equation*}
%$\A\in \calA_{\delta, 4}(B_5)$ 
then
\begin{equation}\label{main-estimate}
\int_{Q_1}|\nabla u|^p \, dx dt\leq C\left\{  \Big(\frac{\theta}{\lambda} \vee  \|u\|_{L^2(Q_6)}\Big)^p + \int_{Q_6}|c|^p \, dx dt\right\}.
\end{equation}
Here  $C>0$ is a constant 
 depending only on  $p$, $\Lambda$ and $n$.
\end{theorem}

The proof of this theorem will be given at the end of subsection~\ref{interior-density-gradient} and will be based on a series of results presented in the next three subsections.

\subsubsection{Some fundamental estimates}\label{sub:fundamental}

Our first result is a $L^2$-estimate for $\nabla u$ in terms of $L^2$-norm of $u$.
%weak solutions of the equation \eqref{RPEQ}.

%=========
\begin{lemma}\label{W^{1,2}-est} Assume $\lambda, \theta>0$, $\A$ satisfies \eqref{interior-elipticity} and $c$ is a non-negative measurable function on $Q_4$. Let $u\in \calW(Q_4)$ be a  non-negative weak solution of \eqref{PEQ} in $Q_4$.
Then there exists $C=C(n,\Lambda)>0$ such that
\begin{equation}\label{guu}
\int_{Q_2} (1 + \lambda u) |\nabla u|^2\, dxdt
\leq C\int_{Q_3} (1+\lambda u +\theta^2)  u^2\, dxdt.
\end{equation}
\end{lemma}
\begin{proof}
Let $\varphi$ be the standard cut-off function which is $1$ 
on $Q_2$ and zero near $\partial_p Q_3$.  Then, by 
multiplying equation \eqref{PEQ} by $\varphi^2 u$ 
and using  integration by parts
  we get
\begin{align*}
&\int_{Q_4} \Big[(\varphi^2 \frac{u^2}{2})_t -\varphi \varphi_t u^2 \Big] \, dxdt =\int_{Q_4} u_t \,\varphi^2 u \, dxdt\\
&= -\int_{Q_4} (1+\lambda u) \langle \A \nabla u ,
\nabla (\varphi^2 u) \rangle\, dxdt+\theta^2 \int_{Q_4} u (1-\lambda u) \varphi^2 u \, dxdt
-\lambda \theta \int_{Q_4} c  \varphi^2 u^2 \, dxdt\\
&\le  -\int_{Q_4} (1+\lambda u) \langle \A \nabla u ,\nabla  u)\varphi^2 \rangle\, dxdt
-2\int_{Q_4} (1+\lambda u) \langle \A \nabla u ,\nabla \varphi )\varphi u \, dxdt
+\theta^2 \int_{Q_4} \varphi^2 u^2 \, dxdt.
\end{align*}  
Using the inequality $ |\langle \A \nabla u, \nabla \varphi\rangle|^2\leq  \langle \A \nabla u, \nabla u\rangle  \, \langle \A \nabla \varphi, \nabla \varphi\rangle$, we deduce from this that
\begin{align*}
 &\int_{Q_4} (1 + \lambda u) \langle \A \nabla u, \nabla u\rangle  \varphi^2 dxdt + \frac{1}{2}\int_{B_4} \varphi(x, 16)^2 u(x, 16)^2\, dx\\
& \leq  2 \int_{Q_4} (1+\lambda u)  \sqrt{\langle \A \nabla u, \nabla u\rangle } \sqrt{\langle \A \nabla \varphi, \nabla \varphi\rangle }\, \varphi u \, dxdt  + \int_{Q_4} \varphi\varphi_t u^2\, dxdt + 
\theta^2 \int_{Q_4}   \varphi^2 u^2 \, dxdt\\
&\leq \frac{1}{2} \int_{Q_4} (1+\lambda u) \langle \A \nabla u, \nabla u\rangle  \varphi^2 \, dxdt +  \int_{Q_4} \Big[2 (1+\lambda u) \langle \A \nabla \varphi, \nabla \varphi\rangle  + \varphi\varphi_t +  \theta^2 \varphi^2 \Big]u^2\, dxdt.
\end{align*} 
Hence it follows from  condition \eqref{interior-elipticity} for $\A$ that 
\begin{align*}
 &\frac{\Lambda^{-1}}{2}\int_{Q_4} (1 + \lambda u) |\nabla u|^2\varphi^2 dxdt\leq  \int_{Q_4} \Big[2 \Lambda(1+\lambda u) |
\nabla \varphi|^2  + \varphi|\varphi_t| +  \theta^2  \varphi^2 \Big]u^2\, dxdt,
\end{align*}  
which yields the conclusion \eqref{guu}.
\end{proof}

We need the following regularity result for   equation \eqref{reqv} whose proof is given in Appendix~\ref{apend}.
 
\begin{lemma}\label{W1infty-est} 
Assume $0<\theta\leq 1$ and $\A_0: (-16, 16]\to \M^{n\times n}$ is  measurable such that
\begin{equation}\label{eigenvalues-controlled}
\Lambda^{-1} |\xi|^2 \leq \xi^T \A_0(t) \xi \leq \Lambda |\xi|^2\quad \mbox{for a.e. $t\in (-16, 16)$ and for all } \xi\in\R^n.
\end{equation}
Let $\bar{v}\in \mathcal W(Q_4)$ be a weak solution of 
\begin{equation}\label{UEQ}
\bar{v}_t  =  \nabla\cdot[(1+\bar{v})\A_0(t)  \nabla \bar{v}] + \theta^2 \bar{v}(1-\bar{v})\quad\mbox{in}\quad Q_4
\end{equation}
satisfying $0\leq \bar{v} \leq 1$ in $Q_4$. Then there exists $C>0$ depending only on $n$ and $\Lambda$ such that
\begin{equation}\label{regularity-reference-eq}
\|\nabla \bar v\|_{L^\infty(Q_3)}^2\leq C\fint_{Q_4}{|\nabla \bar{v}|^2 \, dxdt}.
\end{equation}
%In particular, we have
%\begin{equation}\label{rough-est-reference-%eq}
%\|\nabla \bar{v}\|_{L^\infty(Q_3)}\leq C.
%\end{equation}
\end{lemma}

The next result will be useful for proving the approximation lemma (Lemma~\ref{lm:compare-solution}). 

\begin{lemma}\label{gradient-est-II} 
Assume $\bar{u}\in \mathcal W(Q_4)$ is a non-negative weak solution of \eqref{RPEQ} in $Q_4$.
%satisfying $0\leq \bar{u} %\leq 1$ in $Q_4$. 
Suppose $\bar{v}\in \mathcal W(Q_4)$ is a weak solution of  \eqref{reqv} with  $\bar{v}=\bar{u}$ on $\partial_p Q_4$ and $ 0\le \bar{v} \le 1$ in $Q_4$.
% \begin{equation*}
% \left \{
% \begin{array}{lcll}
% \bar{v}_t  &=&  \nabla\cdot[(1+\bar{v})\bar{\A}_{B_4}\nabla \bar{v}] + \theta^2 \bar{v}(1-\bar{v}) \quad &\text{in}\quad Q_4, \\
% \bar{v} & =& \bar{u}\quad &\text{on}\quad \partial_p Q_4.
% \end{array}\right.
% \end{equation*}
Then
\begin{multline}\label{est-3}
\int_{Q_4}| \bar{u}- \bar{v}|^2 \, dx dt +
\Lambda^{-1} \int_{Q_4}|\nabla \bar{u}- \nabla \bar{v}|^2 \, dx dt\\
\leq  33\left[ 2\Lambda^3 \int_{Q_4} \big(| \bar{u}- \bar{v}|^2 +8\big) |\nabla \bar u|^2\, dx dt +3\theta^2 \int_{Q_4}| \bar{u}- \bar{v}|^2 \, dx dt
+  \lambda^2 \int_{Q_4} \bar{u}^2 c^2 \, dx dt\right].
\end{multline}
\end{lemma}
\begin{proof}
%Note in this case that $\Gamma=\partial \Omega$.
%By the maximum principle, $ 0\le \bar{v} \le 1$ in $Q_4$. 
Let 
$w = \bar u - \bar v$. Then it is easy to see that $w\in \calW(Q_4)$ is a weak solution of 
\[
w_t = \nabla\cdot[(1+\bar{v})
\bar{\A}_{B_4}(t)\nabla w]+ \nabla\cdot\Big\{[w\A + (1+\bar v)(\A-\bar{\A}_{B_4}(t))]\nabla \bar{u}\Big\} +\theta^2 w(1-\bar{u}-\bar{v}) -\lambda\theta c \bar{u}\quad\mbox{in}\quad Q_4,
\]
with $w=0$ on $\partial_p Q_4$. Multiplying the above equation by $w$ and integrating by parts we obtain for each $s\in (-16,16)$ that
\begin{align*}
&\int_{B_4} \frac{w(x,s)^2}{2} \, dx
+ \int_{-16}^s\int_{B_4} (1 +\bar{v}) \langle \bar{\A}_{B_4}(t) \nabla w, \nabla w\rangle \, dx dt\\
&= -
\int_{-16}^s\int_{B_4} w \langle \A \nabla \bar{u} , \nabla w\rangle \, dx dt -
\int_{-16}^s\int_{B_4} (1+\bar v)\langle  (\A-\bar{\A}_{B_4}(t)) \nabla \bar{u}, \nabla w\rangle \, dx dt\\
&\quad  + \theta^2 \int_{-16}^s\int_{B_4} w^2\big(1-\bar{u}-\bar{v}\big)\, dx dt
- \lambda\theta  \int_{-16}^s\int_{B_4} c \bar{u} w\, dx dt.
\end{align*}
We deduce from this and condition \eqref{interior-elipticity} for $\A$, which also holds for $\bar{\A}_{B_4}(t)$, and the fact $ \bar u\geq 0$, $0\leq \bar v\leq 1$ that
\begin{align*}
&\frac{1}{2}\int_{B_4} w(x,s)^2 \, dx
+ \Lambda^{-1}\int_{-16}^s\int_{B_4}  | \nabla w|^2 \, dx dt\\
&\leq 
\Lambda \int_{-16}^s\int_{B_4} |w|\, | \nabla \bar{u}|\, |\nabla w| \, dx dt +
 4 \Lambda \int_{-16}^s\int_{B_4}    |\nabla \bar{u}| \, | \nabla w|\, dx dt\\
&\quad  + \theta^2 \int_{-16}^s\int_{B_4} w^2\, dx dt
+ \lambda\theta  \int_{-16}^s\int_{B_4} c \bar{u} |w|\, dx dt.
\end{align*}
Hence, applying Cauchy's inequality and collecting like-terms give
\begin{align}\label{first-derivation}
&\frac{1}{2} \int_{B_4} w(x,s)^2\, dx
+ \frac{\Lambda^{-1}}{2}\int_{-16}^s\int_{B_4}  |\nabla w|^2 \, dx dt
\leq 
\Lambda^3 \int_{-16}^s\int_{B_4} \big( w^2 + 8\big)  | \nabla \bar{u}|^2 \, dx dt\\
& + \frac{3 \theta^2}{2} \int_{-16}^s\int_{B_4} w^2\, dx dt
+\frac{ \lambda^2}{2}  \int_{-16}^s\int_{B_4} \bar{u}^2 c^2\, dx dt\quad\mbox{ for each } s\in (-16,16).\nonumber
\end{align}
On the one hand, this immediately yields
\begin{align}\label{est-1}
 \Lambda^{-1}\int_{Q_4}  |\nabla w|^2 \, dx dt
 \leq
  2\Lambda^3  \int_{Q_4} \big( w^2 + 8\big)  | \nabla \bar{u}|^2 \, dx dt +
3 \theta^2
\int_{Q_4} w^2 \, dx dt 
+ \lambda^2 \int_{Q_4} \bar{u}^2 c^2  \, dx dt.
\end{align}
On the other hand, we can drop the second term in \eqref{first-derivation} and then
integrate in $s$ to obtain
\begin{align}\label{est-2}
 \int_{Q_4}  w^2 \, dx dt
&\leq
64 \Lambda^3\int_{Q_4} \big( w^2 + 8\big) | \nabla \bar{u}|^2 \, dx dt +
96\theta^2
\int_{Q_4} w^2 \, dx dt 
+32 \lambda^2 \int_{Q_4} \bar{u}^2 c^2  \, dx dt.
\end{align}
By adding \eqref{est-1} and \eqref{est-2}, we get \eqref{est-3}.
\end{proof}

\subsubsection{Interior approximation estimates}\label{sub:interior-approx}
We begin this subsection with a result allowing us to approximate a weak solution of \eqref{PEQ} by that of the reference equation. 
\begin{lemma}\label{lm:compare-solution}
For any $\e>0$, there exists $\delta>0$ depending only on $\e$, $\Lambda$ and  $n$  such that:  if $0<\theta \leq \lambda$,
\begin{equation}\label{AAcCond}
\int_{Q_4} \Big[|\A(x,t) - \bar{\A}_{B_4}(t)|^2 +  |c(x,t)|^2\Big]\, dx dt \leq \delta,
\end{equation}
and $u\in\mathcal W(Q_5)$ is a weak solution of \eqref{PEQ} in $Q_5$ satisfying
\begin{equation}\label{gradone}
0\leq u \leq \frac{1}{\lambda}\, \mbox{ in } Q_4 \quad\mbox{and}\quad \int_{Q_4}{|\nabla u|^2 \, dxdt}\leq 1,
\end{equation}
and $v\in \mathcal W(Q_4)$ is the weak solution of \eqref{REQ} with  $v=u$ on $\partial_p Q_4$ and $0\leq v \leq \frac{1}{\lambda}$ in $Q_4$,
then 
\begin{equation}\label{u-v-close}
\int_{Q_4}{|u - v|^2\, dx dt}\leq \e^2,
\end{equation}
and, consequently,
\begin{equation}\label{gradient-v-controlled}
\int_{Q_4} |\nabla v|^2\, dx dt \leq 2+ 66\Lambda\Big(18 \Lambda^2 + 3 \theta^2 \e^2 + \delta\Big).
\end{equation}
\end{lemma}
\begin{proof}
We first prove \eqref{u-v-close} by contradiction. Suppose that estimate \eqref{u-v-close} is not true. Then there exist $\e_0,\, \Lambda,\,  n$,  sequences of  numbers  $\{\lambda_k\}_{k=1}^\infty$ and $\{\theta_k\}_{k=1}^\infty$
with $0<\theta_k\leq \lambda_k$, a sequence of coefficient matrices $\{\A_k\}_{k=1}^\infty$, 
 and sequences of non-negative functions  $\{c_k\}_{k=1}^\infty$ and $\{u^k\}_{k=1}^\infty$ 
such that
\begin{equation}\label{c_k-condition}
\int_{Q_4} \Big[|\A_k(x, t) - \bar{\A}_k(t)|^2 +  |c_k(x,t)|^2\Big]\, dx dt \leq \frac{1}{k},
\end{equation}
 $u^k\in \mathcal W(Q_5)$ is a weak solution of 
\begin{equation}\label{eq-u_k}
u^k_t  =  \nabla\cdot[(1+\lambda_k u^k)\A_k \nabla u^k] + \theta_k^2 u^k(1-\lambda_k u^k) - \lambda_k \theta_k c_k u^k\quad\mbox{in}\quad Q_5
\end{equation}
with
\begin{equation}\label{ulam-assum}
0\leq u^k \leq \frac{1}{\lambda_k}\quad \mbox{in}\quad Q_4,
\end{equation}
\begin{equation}\label{gradient-bounded-ass}
\int_{Q_4}{|\nabla u^k|^2 \, dxdt}\leq 1,
\end{equation}
\begin{equation}\label{contradiction-conclusion}
\int_{Q_4} |u^k - v^k|^2 \, dx dt > \e_0^2  \quad\mbox{for all } k.
\end{equation}
Here  $\bar{\A}_k(t) = \fint_{B_4}{\A_k(x,t) dx}$ and $v^k\in \calW(Q_4)$  is a weak solution  of
\begin{equation*}
\left \{
\begin{array}{lcll}
v^k_t  &=&  \nabla\cdot[(1+\lambda_k v^k)\bar{\A}_k(t) \nabla v^k] + \theta_k^2 v^k(1-\lambda_k v^k) \quad &\text{in}\quad Q_4, \\
v^k & =& u^k\quad &\text{on}\quad \partial_p Q_4
\end{array}\right.
\end{equation*}
satisfying $0\leq v^k \leq 1/\lambda_k$ in $Q_4$. Since
$0\leq u^k, v^k \leq 1/\lambda_k$ in $Q_4$, we infer from \eqref{contradiction-conclusion}  that
\[
 \theta_k \leq \lambda_k\leq \frac{|Q_4|^{\frac{1}{2}}}{\e_0},
\]
that is, the sequences $\{\lambda_k\}$ and $\{\theta_k\}$  are  bounded. Also $\{\bar{\A}_k\}$ is   bounded in $L^\infty(-16, 16; \M^{n\times n})$ due to condition \eqref{interior-elipticity} for $\A_k$. Then, by taking subsequences if necessary, we can assume that $\lambda_k \to \lambda$, 
$\theta_k \to\theta$, $\bar{\A}_k \rightharpoonup  \A_0$ weakly-* in 
$L^\infty(-16,16; \M^{n\times n})$  and $\bar{\A}_k \rightharpoonup  \A_0$ weakly in 
$L^2(Q_4)$ for some constants
$\lambda, \theta$ satisfying $0\leq \theta\leq \lambda<\infty$ and some  $\A_0\in L^\infty(-16,16; \M^{n\times n})$. For each vector $\xi\in \R^n$, we have
\[
\int_{-16}^{16} \xi^T \A_0(t) \xi \, \phi(t) dt =\lim_{k\to\infty} \int_{-16}^{16} \xi^T \bar{\A}_k(t) \xi \, \phi(t) dt \quad \mbox{for all non-negative functions } \phi\in L^1(-16, 16).
\]
This together with the denseness of $Q^n$ in $\R^n$ implies that $A_0$ satisfies condition \eqref{eigenvalues-controlled}.
We are going to derive a contradiction by proving the following claim.

\textbf{Claim.} There are subsequences $\{u^{k_m}\}$ and $\{v^{k_m}\}$ such that $u^{k_m} - v^{k_m} \to 0$ in $L^2(Q_4)$ as $m\to\infty$.

Let us consider the case $\lambda>0$ first. Then, thanks to \eqref{ulam-assum}, the sequence $\{u^k\}$ is bounded in $Q_4$. 
This together with \eqref{c_k-condition}, \eqref{eq-u_k}, \eqref{gradient-bounded-ass} and the boundedness of $\{\A_k\}$ and $\{\theta_k\}$ implies that the sequence $\{u^k\}$ is 
bounded in  $\mathcal W(Q_4)$. Next, we apply Lemma~\ref{gradient-est-II} for $\bar{u} \rightsquigarrow \lambda_k u^k$ and $\bar{v} \rightsquigarrow \lambda_k v^k$ to obtain
\begin{align*}
& \int_{Q_4}|\nabla u^k- \nabla v^k|^2 \, dx dt\\
&\leq  33\Lambda\left[18 \Lambda^3 \int_{Q_4} |\nabla u^k|^2\, dx dt +3\theta_k^2 \int_{Q_4}| u^k- v^k|^2 \, dx dt
+  \lambda_k^2 \int_{Q_4} (u^k)^2 c_k^2 \, dx dt\right]\\
&\leq  33\Lambda \left[ 18 \Lambda^3 \int_{Q_4} |\nabla u^k|^2\, dx dt +3 |Q_4| 
+  \int_{Q_4}  c_k^2 \, dx dt\right].
\end{align*}
Thanks to  \eqref{c_k-condition} and \eqref{gradient-bounded-ass}, this gives
 \[\int_{Q_4} |\nabla v^k |^2 \, dx dt\leq C\quad \mbox{for all } k.
\]
Thus, by reasoning as in the case of $\{u^k\}$,  the sequence $\{v^k\}$ is also bounded in  $\mathcal W(Q_4)$.
We infer from these facts and the compact embedding 
\eqref{cmpAL} that there exist  subsequences, still denoted by $\{u^k\}$ and $\{v^k\}$, and  functions $u, v\in\mathcal W(Q_4)$ such that 
\begin{equation*}
 \left \{ \  
\begin{array}{lll}
&u^k \to u  \mbox{ strongly in } L^2(Q_4),\quad \nabla u^k \rightharpoonup \nabla u \text{ weakly in } L^2(Q_4),\\
&\partial_t u^k \rightharpoonup \partial_t  u \text{ weakly-* in } L^2(0,T; H^{-1}(B_4)),
\end{array} \right . 
\end{equation*}
and
\begin{equation*}
 \left \{ \  
\begin{array}{lll}
&v^k \to v  \mbox{ strongly in } L^2(Q_4),\quad \nabla v^k \rightharpoonup \nabla v \text{ weakly in } L^2(Q_4),\\
&\partial_t v^k \rightharpoonup \partial_t  v \text{ weakly-* in } L^2(0,T; H^{-1}(B_4)).
\end{array} \right . 
\end{equation*}
Moreover, from the boundedness of $\{u^k\}$ and \eqref{c_k-condition}, we see that
\begin{equation*}
\left \{ 
\begin{array}{lll}
& c_k u^k \to 0 \text{ strongly in } L^2(Q_4)\\
&(1+\lambda_k u^k)(\A_k-\bar \A_k(t)) \nabla u^k \rightharpoonup 0 \ \text{ weakly  in}\ L^2(Q_4).
\end{array}  \right . 
\end{equation*}
Then, for all $\varphi\in C_0^\infty(Q_4)$, we see that 
\begin{align}\label{eq:takecare-limit-I}
 &\lim_{k\to\infty}\int_{Q_4}{(1+\lambda_k u^k) \langle \bar \A_k(t) \nabla u^k, \nabla \varphi\rangle dx dt}\\
 &=-\lim_{k\to\infty}\int_{Q_4}{   u^k \bar \A_k(t) \cdot  D^2 \varphi\, dx dt} - \lim_{k\to\infty} \frac{\lambda_k}{2} \int_{Q_4}{ (u^k)^2\bar \A_k(t) \cdot D^2 \varphi\, dx dt}\nonumber\\
 &=-\int_{Q_4}{   u \A_0(t) \cdot  D^2 \varphi\, dx dt} -  \frac{\lambda}{2} \int_{Q_4}{ u^2 \A_0(t) \cdot D^2 \varphi\, dx dt}\nonumber\\
 &=\int_{Q_4}{(1+\lambda u) \langle  \A_0(t) \nabla u, \nabla \varphi\rangle dx dt}.\nonumber
 \end{align}
Thus  by passing $k\to\infty$ for the equation \eqref{eq-u_k} and using the boundedness of $\{\lambda_k\}$ and $\{\theta_k\}$, one sees that $u$ is weak solution of the equation
\begin{equation}\label{limit-equation}
u_t  =  \nabla\cdot[(1+\lambda u)\A_0(t) \nabla u] + \theta^2 u(1-\lambda u)\quad\mbox{in}\quad Q_4.
\end{equation}
Similarly, $v$ is a weak solution of
\begin{equation*}%\label{limit-equation}
v_t  =  \nabla\cdot[(1+\lambda v)\A_0(t) \nabla v] + \theta^2 v(1-\lambda v)\quad\mbox{in}\quad Q_4.
\end{equation*}
In addition, we infer from the strong convergence of $u^k$ and $v^k$ in $L^2(Q_4)$ that 
\[
0 \leq u, v \leq \frac{1}{\lambda}, \quad \text{and} \quad u = v \quad \text{on} \quad  \partial_p Q_4.
\]
Hence by the uniqueness of the solution of equation \eqref{limit-equation} given by Lemma~\ref{uniqueness-u}, we conclude that 
$u \equiv  v$ in $Q_4$. Therefore, $u^k-v^k \to u-v=0$ strongly in $L^2(Q_4)$.

Now, consider the case $\lambda =0$, that is, $\lambda_k \to 0$. Due to $\theta_k\leq \lambda_k$, we also have $\theta_k \to 0$.
Let $w^k = u^k - v^k$. Then    $w^k$ is a bounded weak solution of 
\begin{equation}\label{u^k-v^k:equation}
\left \{
\begin{array}{lcll}
w^k_t  &=&  \nabla\cdot[(1+\bar v^k)\bar{\A}_k(t)\nabla w^k] +  \nabla\cdot\Big\{\big[\bar w^k\A_k + (1+\bar v^k)(\A_k - \bar{\A}_k(t))\big] \nabla u^k\Big\}\\
& & \qquad \quad + \theta_k^2 w^k(1- \bar u^k -\bar v^k) 
- \theta_k c_k\bar u^k &\text{ in }\, Q_4, \\
w^k & =& 0 &\text{ on } \, \partial_p Q_4,
\end{array}\right.
\end{equation}
where $\bar u^k=\lambda_k u^k$, $\bar v^k=\lambda_k v^k$ and $\bar w^k=\bar u^k-\bar v^k$.

Note that $0\le \bar u^k\le 1$ in $Q_4$ and $\|\nabla \bar u^k\|_{L^2(Q_4)}\to 0$ as $k\to\infty$.
%\begin{equation*} \label{gradz}
%\|\nabla \bar u^k\|_{L^2(Q_4)}\to 0\quad\text{  as }\quad k\to\infty.
%\end{equation*}
We have $\bar u^k$ is a weak solution of 
\begin{equation*}
(\bar u^k)_t  =  \nabla\cdot[(1+\bar{u}^k )\A_k \nabla \bar{u}^k ] + \theta^2 \bar{u}^k (1- \bar{u}^k ) - \lambda_k \theta_k c \bar{u}^k  \quad\mbox{in}\quad Q_4.
\end{equation*}
Then $\{\bar u^k\}$ is bounded in $\mathcal W(Q_4)$.
Also, $0\leq \bar{v}^k\leq 1$ in $Q_4$ and $\bar{v}^k$ is a weak solution of
\begin{equation}\label{eq-bar-v-k}
 (\bar{v}^k)_t  =  \nabla\cdot[(1+\bar{v}^k)\bar{\A}_k(t)\nabla \bar{v}^k] + \theta_k^2 \bar{v}^k(1-\bar{v}^k) \quad \text{in}\quad Q_4,\quad \bar v^k=\bar u^k\quad\text{on}\quad \partial_p Q_4.
\end{equation}

By applying Lemma~\ref{gradient-est-II} for $\bar{u} \rightsquigarrow \bar u^k$, $\bar{v} \rightsquigarrow \bar v^k$ and using 
the fact $\theta_k$ is small for large $k$, we get
for all sufficiently large $k$ that
\begin{equation}\label{w^k-gradient}
 \int_{Q_4}|\nabla w^k|^2 \, dx dt
\leq  33\Lambda \left[18\Lambda^3 \int_{Q_4} |\nabla u^k|^2\, dx dt 
+ \int_{Q_4}  c_k^2 \, dx dt\right]\leq 33\Lambda \left(18\Lambda^3 +1\right)
\end{equation}
and
\begin{equation*}\label{w^k-function}
\int_{Q_4} |w^k|^2 \, dx dt
\leq 66\left[18\Lambda^3  \int_{Q_4} |\nabla u^k|^2\, dx dt 
+   \int_{Q_4} c_k^2 \, dx dt\right]\leq 66  \left(18\Lambda^3 +1\right).
\end{equation*}
Therefore, $\{\nabla w^k\}$ is bounded in $L^2(Q_4)$, and consequently, $\{w^k\}$ is bounded in $\mathcal W(Q_4)$.
Moreover, $\{\nabla \bar{v}^k\}$ is bounded in $L^2(Q_4)$ since it follows from \eqref{gradient-bounded-ass} and \eqref{w^k-gradient} that
\begin{equation}\label{eq-grad-v-k}
\int_{Q_4} |\nabla v^k |^2 \, dx dt\leq C\quad \mbox{for all large } k.
\end{equation}
Consequently there are subsequences, still denoted by $\{w^k\}$ and $\{\bar v^k\}$ and two functions $w,\bar v\in \mathcal W(Q_4)$ with $0\leq \bar{v}\leq 1$ in $Q_4$ such that
$w^k \to w$ and  $\bar{v}^k \to \bar v$ strongly in $L^2(Q_4)$,
$\nabla w^k \to \nabla w$ and  $\nabla \bar{v}^k \rightharpoonup \nabla \bar v$ weakly in $L^2(Q_4)$, 
$\partial_t w^k \rightharpoonup \partial_t  w$ and  $\partial_t \bar{v}^k  \rightharpoonup \partial_t  \bar v$ weakly-* in $L^2(0,T; H^{-1}(B_4))$. 

Since $\nabla \bar{v}^k =\lambda_k \nabla v^k \to 0$  
in $L^2(Q_4)$ thanks to \eqref{eq-grad-v-k}, we infer further that   $\nabla \bar{v}^k \to \nabla \bar v\equiv 0$ strongly in $L^2(Q_4)$.
Also, by passing to the limit in \eqref{eq-bar-v-k} we 
see that $\bar v$ is a weak solution of 
$\bar{v}_t =\nabla\cdot[(1+\bar{v})\A_0(t)\nabla \bar{v}]$ in $Q_4$. Thus, we deduce that $(\nabla \bar v, \bar v_t)\equiv 0$ in $Q_4$ and hence $\bar v$ is a constant function. 
Due to this fact and by arguing as in \eqref{eq:takecare-limit-I}, one gets for all $\varphi\in C_0^\infty(Q_4)$ that
\begin{align*}
 &\lim_{k\to\infty}\int_{Q_4}{(1+\bar v^k) \langle \bar \A_k(t) \nabla w^k, \nabla \varphi\rangle dx dt}=-\lim_{k\to\infty}\int_{Q_4}{   w^k \bar \A_k(t) \cdot  D^2 \varphi\, dx dt}\\
 &+\lim_{k\to\infty}\int_{Q_4}{  (\bar v^k -\bar v) \langle \bar \A_k(t) \nabla w^k, \nabla \varphi\rangle dx dt} - \lim_{k\to\infty}  \int_{Q_4}{ \bar v \, \bar \A_k(t) \cdot D^2 \varphi\, dx dt}\\
 &=-\int_{Q_4}{   w \A_0(t) \cdot  D^2 \varphi\, dx dt} - \int_{Q_4}{ \bar v\, \A_0(t) \cdot D^2 \varphi\, dx dt}
 =\int_{Q_4}{(1+\bar v) \langle  \A_0(t) \nabla w, \nabla \varphi\rangle dx dt}.
 \end{align*}

Using the aforementioned convergences and similar to the case $\lambda>0$, we can pass to the limit in \eqref{u^k-v^k:equation} to 
conclude that $w$ is a weak solution of
\begin{equation}\label{limit-equation-case2}
\left \{
\begin{array}{lcll}
w_t  &=&  \nabla\cdot[(1+\bar v)\A_0(t) \nabla w] &\text{ in}\quad Q_4, \\
w & =& 0 &\text{ on}\quad \partial_p Q_4.
\end{array}\right.
\end{equation}
% These estimates together with \eqref{u^k-v^k:equation}, \eqref{gradient-bounded-ass} and the boundedness of $\{\A_k\}$ imply that
% there exist  a subsequence $\{w^k\}$  and  a function $w\in W_*(Q_4)$ such that
% \begin{equation}\label{w^k-compactness}
% \begin{aligned}
% &w^k \rightarrow w \text{ strongly in }L^2(Q_4),\\
% &\nabla w^k \rightharpoonup \nabla w \text{ weakly in } L^2(Q_4),\\
% &\partial_t w^k \rightharpoonup \partial_t w  \text{ weakly-* in } L^2(0,T;(H^1_{0}(B_4))^*).
% \end{aligned}
% \end{equation}
% Thus $\bar{v}^k \longrightarrow \bar v$ strongly in $L^2(Q_4)$ for some function   $\bar{v} \in W_*(Q_4)$ with $0\leq \bar{v}\leq 1$ in $Q_4$. From this and the fact
% \[
% (1+\lambda_k v^k)\bar{\A}_k - (1+\bar v) \A_0 =  (\bar{v}^k -\bar v) \bar{\A}_k + (1+\bar v)(\bar{\A}_k - \A_0),
% \]
% we  also infer that $(1+\lambda_k v^k)\bar{\A}_k \longrightarrow (1+\bar v) \A_0 $ strongly in $L^2(Q_4)$.
% 
% Using these convergences, \eqref{c_k-condition} and  \eqref{w^k-compactness}, we can pass 
% to the limits
% in  \eqref{u^k-v^k:equation} to conclude that  $w$ is a weak solution of the equation
% \begin{equation}\label{limit-equation-case2}
% \left \{
% \begin{array}{lcll}
% w_t  &=&  \nabla\cdot[(1+\bar v)\A_0 \nabla w] &\text{ in}\quad Q_4, \\
% w & =& 0 &\text{ on}\quad \partial_p Q_4.
% \end{array}\right.
% \end{equation}
By the uniqueness of the trivial solution of the linear equation \eqref{limit-equation-case2}, we conclude that $w \equiv 0$ in $Q_4$. This gives, again, $u^k-v^k=w^k\to 0$ in $L^2(Q_4)$ as $k\to\infty$.

% \[
% \lim_{k\to\infty}\int_{Q_4} |u^k- v^k|^2\, dx dt =\lim_{k\to\infty}\int_{Q_4} |w^k|^2\, dx dt
% =\int_{Q_4} |w|^2\, dx dt=0,
% \]
Therefore, we have proved the Claim which contradicts \eqref{contradiction-conclusion}. Thus the proof of \eqref{u-v-close} is complete.
To prove  \eqref{gradient-v-controlled}, we apply Lemma~\ref{gradient-est-II} for $\bar{u} \rightsquigarrow \lambda u$ and $\bar{v} \rightsquigarrow \lambda v$ to obtain
\begin{align*}
 \int_{Q_4}|\nabla u- \nabla v|^2 \, dx dt\leq  33\Lambda\left[ 18\Lambda^3 \int_{Q_4}  |\nabla  u|^2\, dx dt +3\theta^2 \int_{Q_4}| u-v|^2 \, dx dt
+ \int_{Q_4} c^2 \, dx dt\right].
\end{align*}
This, \eqref{u-v-close}  and the assumptions \eqref{AAcCond}, \eqref{gradone} give 
\begin{equation}\label{difference-control}
\int_{Q_4}|\nabla u- \nabla v|^2 \, dx dt\leq  33\Lambda\left[ 18\Lambda^3  +3\theta^2 \e^2 
+\delta \right].
\end{equation}
Since $\|\nabla v\|_{L^2(Q_4)}\leq \|\nabla u\|_{L^2(Q_4)} + \|\nabla u- \nabla v\|_{L^2(Q_4)}\leq 1+ \|\nabla u- \nabla v\|_{L^2(Q_4)}$, the estimate \eqref{gradient-v-controlled} follows immediately from \eqref{difference-control}.
\end{proof}

%\begin{remark} The function $\bar v(x,t)$ in equation
%\eqref{limit-equation-case2} is actually a constant. Indeed, since $\nabla \bar{v}^k =\lambda_k \nabla v^k \to 0$  in $L^2(Q_4)$, we first conclude that $\nabla \bar v=0$. On the other hand, by passing to the limit we 
%see that $\bar v$ is a weak solution of 
%$\bar{v}_t =\nabla\cdot[(1+\bar{v})\A_0\nabla \bar{v}]$ in $Q_4$. Thus, we infer that $(\nabla \bar v, \bar v_t)=0$ and hence $\bar v$ is a constant function. 
%By an alternative argument using the Poincar\'e inequality, it is possible to replace
%$\bar v$ in 
%\eqref{limit-equation-case2}  
%by a constant which is a limit point of the bounded sequence
%\[
%\bar{v}^k_{Q_4} =\fint_{Q_4}{\bar v^k\, dx dt}.
%\]
%\end{remark}

The next lemma is a localized version of Lemma \ref{lm:compare-solution} together with a comparison between gradients of solutions. It is crucial for establishing interior $W^{1,p}$-estimates in 
Subsection~\ref{interior-density-gradient}.

\begin{lemma}\label{lm:localized-compare-solution}
Assume that $0<\theta \leq \lambda$, $\theta\leq 1$ and $0<r\leq 1$. For any $\e>0$, there exists $\delta>0$ depending only on $\e$, $\Lambda$ and  $n$  such that:  if
\begin{equation}\label{smallness-condition-c}
\fint_{Q_{4 r}}  \Big[|\A(x,t) - \bar{\A}_{B_{4r}}(t)|^2 + |c(x,t)|^2\Big]\, dx dt \leq \delta,
\end{equation}
then for any weak solution $u\in \mathcal W(Q_{5 r})$ of \eqref{PEQ} in $Q_{5 r}$ satisfying
\begin{equation*}
0\leq u \leq \frac{1}{\lambda}\, \mbox{ in } Q_{4 r} \quad\mbox{and}\quad \fint_{Q_{4 r}}{|\nabla u|^2 \, dxdt}\leq 1,
\end{equation*}
 and a weak solution  $v\in \mathcal W(Q_{4 r})$ of  
\begin{equation*}\label{limiting-eqn}
\left \{
\begin{array}{lcll}
v_t  &=&  \nabla\cdot[(1+\lambda v)\bar{\A}_{B_{4r}}(t)\nabla v] + \theta^2 v(1-\lambda v) \quad &\text{in}\quad Q_{4 r}, \\
v & =& u\quad &\text{on}\quad \partial_p Q_{4 r}
\end{array}\right.
\end{equation*}
satisfying $0\leq v \leq \frac{1}{\lambda}$ in $Q_{4 r}$, we have 
\begin{equation}\label{localized-compare-uv}
\fint_{Q_{4 r}}{|u - v|^2\, dx dt}\leq \e^2 r^2,
\end{equation}
 \begin{equation}\label{rescaled-gradient-v-controlled}
\fint_{Q_{4r}} |\nabla v|^2\, dx dt \leq 4^{n+2} 2 \omega_n \left[2+ 66\Lambda\Big(18 \Lambda^2 + 3 \theta^2 \e^2 + \delta\Big)\right],
\end{equation}
\beq \label{compare-graduv}
\fint_{Q_{2 r}}{|\nabla u - \nabla v|^2\, dx dt}\leq \e^2,
\eeq
where $\omega_n$ is the volume of the unit ball $B_1$ in $\R^n$.
\end{lemma}
\begin{proof}
Estimates \eqref{localized-compare-uv} and \eqref{rescaled-gradient-v-controlled} are a localized version of Lemma~\ref{lm:compare-solution}.
Define
\[
u'(x,t) =\frac{u(r x, r^2 t)}{r}, \quad v'(x,t)= \frac{v(r x, r^2 t)}{r},\quad \A'(x, t) = \A(r x, r^2 t) \quad\mbox{and}\quad c'(x,t) = c(r x, r^2 t).
\]
Let  $\lambda' = \lambda r$ and $\theta' = \theta r$. Then $u'$ is a weak solution of
\begin{equation*}
u'_t  =  \nabla\cdot[(1+\lambda' u')\A' \nabla u'] + \theta'^2 u'(1-\lambda' u') - \lambda' \theta' c' u' \quad\mbox{in}\quad Q_5
\end{equation*}
and $v'$ is  a weak solution of  
\begin{equation*}
\left \{
\begin{array}{lcll}
v'_t  &=&  \nabla\cdot[(1+\lambda' v')\bar{\A'}_{B_4}(t)\nabla v'] + \theta'^2 v'(1-\lambda' v') \quad &\text{in}\quad Q_{4}, \\
v' & =& u'\quad &\text{on}\quad \partial_p Q_{4}.
\end{array}\right.
\end{equation*}
We also have $0\leq u', v'\leq 1/\lambda'$ in $Q_4$, $\Lambda^{-1}|\xi|^2 \leq \xi^T \A'(x)\xi \leq \Lambda |\xi|^2$ and
\begin{align*}
&\int_{Q_{4}}{|\nabla u'(x,t)|^2 \, dxdt}=  4^{n+2} 2 \omega_n\fint_{Q_{4 r}}{|\nabla u(y,s)|^2 \, dyds}\leq 4^{n+2} 2 \omega_n,\\
&\int_{Q_{4}} \Big[ |\A'(x,t) -\bar{\A'}_{B_4}(t)|^2 +  |c'(x,t)|^2\Big]\, dx dt =4^{n+2} 2 \omega_n \fint_{Q_{4 r}} \Big[ |\A(y, s) -\bar{\A}_{B_{4 r}}(s)|^2  +  |c(y,s)|^2\Big]\, dy ds.
\end{align*}
%where $\omega_n$ is the volume of the %unit ball $B_1$ in $\R^n$.
Therefore, given any $\e>0$, by Lemma~\ref{lm:compare-solution} there exists a constant $\delta=\delta(\e,\Lambda, n)>0$ such that if 
condition \eqref{smallness-condition-c} for $\A$ and $c$ is satisfied then we have
\begin{equation*}
\int_{Q_{4}}{|u'(x,t) - v'(x,t)|^2\, dx dt}\leq 4^{n+2} 2 \omega_n\e^2.
\end{equation*}
By changing variables, we obtain the desired estimate
\eqref{localized-compare-uv}. On the other hand, the estimate \eqref{rescaled-gradient-v-controlled} is a consequence of \eqref{gradient-v-controlled} (see also the calculations at the end of the proof of Lemma~\ref{lm:compare-solution}).

We now prove \eqref{compare-graduv}. 
%Since $0\leq u \leq \frac{1}{\lambda}$ in $Q_{4 r}$,  we have from the maximum principle that $0\leq v \leq \frac{1}{\lambda}$ in $Q_{4 r}$.
Define $w= u-v$. Then $w\in \calW(Q_{4 r})$ is a bounded weak solution of 
\begin{align}\label{u-v:equation}
w_t 
= \nabla\cdot[(1+\lambda u )\A \nabla w] &+ \nabla\cdot \Big\{[\lambda w \bar{\A}_{B_{4 r}}(t) + (1+\lambda u)(\A- \bar{\A}_{B_{4 r}}(t))]\nabla v \Big\}\\
& + \theta^2 w(1-\lambda u -\lambda v) -\lambda \theta c u\qquad\mbox{in}\quad Q_{4 r}.\nonumber
\end{align}
Let $\varphi$ be the standard cut-off function which is $1$ 
on $Q_{2 r}$, $\text{supp}(\varphi) \subset \overline{Q_{3 r}}$, $|\nabla \varphi| \leq C_n/r$ and $|\varphi_t|\leq C_n/ r^2$.  We multiply  equation \eqref{u-v:equation}  by $\varphi^2 w$ and use  integration by parts to obtain
\begin{align*}
&\int_{Q_{4 r}} \Big[(\varphi^2 \frac{w^2}{2})_t -\varphi \varphi_t w^2 \Big] \, dxdt =\int_{Q_{4 r}} w_t \,\varphi^2 w \, dxdt
 = -\int_{Q_{4 r}} (1+\lambda u) \langle \A \nabla w,
\nabla (\varphi^2 w)\rangle \, dxdt\\ &\quad -\int_{Q_{4 r}} w \langle \bar{\A}_{B_{4 r}}(t) \nabla (\lambda v) ,
\nabla (\varphi^2 w)\rangle \, dxdt -\int_{Q_{4 r}}(1+\lambda u) \langle (\A-\bar{\A}_{B_{4 r}}(t)) \nabla v ,
\nabla (\varphi^2 w)\rangle \, dxdt\\ &\quad +\theta^2 \int_{Q_{4 r}} w (1-\lambda u-\lambda v) \varphi^2 w \, dxdt - \theta \int_{Q_{4 r}} c (\lambda u) \varphi^2 w \, dxdt.
\end{align*}
We deduce from this,  condition \eqref{interior-elipticity}  for $\A$ and the  assumption $\theta\leq 1$ that
\begin{align*}
&  \int_{Q_{4 r}} (1+\lambda u) \langle \A \nabla w, \nabla w\rangle\varphi^2 \, dxdt
\leq 2  \int_{Q_{4 r}} (1+\lambda u) \sqrt{\langle \A \nabla w, \nabla w\rangle} \sqrt{\langle \A \nabla  \varphi, \nabla  \varphi\rangle} \, \varphi |w| \, dxdt\\ 
& +\Lambda  \left(\int_{Q_{4 r}} |\nabla (\lambda v)| |\nabla w |  \varphi^2 |w|  \, dxdt
+ 2 \int_{Q_{4 r}}  |\nabla (\lambda v)|  | \nabla \varphi|  \varphi w^2  \, dxdt \right)\\
& + \left(\int_{Q_{4 r}}(1+\lambda u) |\A-\bar{\A}_{B_{4 r}}(t)|  |\nabla v| 
|\nabla w| \varphi^2 \, dxdt
+2\int_{Q_{4 r}}(1+\lambda u) |\A-\bar{\A}_{B_{4 r}}(t)|  |\nabla v| 
|\nabla\varphi| \varphi |w| \, dxdt\right) 
 \\ 
&+\int_{Q_{4 r}} \varphi \varphi_t w^2 \, dxdt + \int_{Q_{4 r}} \varphi^2 w^2  \, dxdt +  \int_{Q_{4 r}} |c|   \varphi^2 |w| \, dxdt.
\end{align*} 
Using the Cauchy--Schwarz inequality and moving terms around, we get
\begin{align}\label{Cauchy--Schwarz-step}
 &\frac{\Lambda^{-1}}{4}\int_{Q_{4 r}} (1 + \lambda u) |\nabla w|^2\varphi^2 dxdt \leq  (4 \Lambda^2 +1) \int_{Q_{4 r}}|
\nabla \varphi|^2 w^2 \, dxdt\\
&   + \Lambda \|\nabla (\lambda v)\|_{L^\infty(Q_{3 r})}\left( \int_{Q_{4 r}}  |\nabla w|\varphi^2 |w| \, dxdt
+ 2 \int_{Q_{4 r}}  |
\nabla \varphi| \varphi w^2 \, dxdt\right)\nonumber\\
& +2(\Lambda +2) \int_{Q_{4 r}}|\A-\bar{\A}_{B_{4 r}}(t)|^2  |\nabla v|^2 
 \varphi^2 \, dxdt
\nonumber\\
&+ \int_{Q_{4 r}} \varphi\varphi_t w^2\, dxdt  +\int_{Q_{4 r}}    w^2 \, dxdt + \int_{Q_{4 r}}  |c|  |w| \, dxdt.\nonumber
\end{align} 
We estimate $\|\nabla (\lambda v)\|_{L^\infty(Q_{3 r})}$ and $\|\nabla v\|_{L^\infty(Q_{3 r})}$ next. Let us define 
$\bar{v}(x,t) = \lambda v(rx, r^2 t)$ for $(x,t)\in Q_4$.  Then $0\leq \bar{v}\leq 1$ in $Q_4$ and  $\bar{v}$ is a weak solution of
\begin{equation*}
\bar{v}_t  =  \nabla\cdot[(1+\bar{v})\bar{\A'}_{B_{4}}(t) \nabla \bar{v}] + (\theta r)^2 \bar{v}(1-\bar{v})\quad\mbox{in}\quad Q_4.
\end{equation*}
Thanks to $\theta r\leq \theta\leq 1$, we then can use Lemma~\ref{W1infty-est} to get
 \begin{equation}\label{handle-gradients-v}
 \|\nabla \bar v\|_{L^\infty(Q_3)}\leq C(\Lambda, n)\left(\fint_{Q_{\frac{7}{2}}}{|\nabla \bar{v}|^2 \, dxdt}\right)^{\frac{1}{2}}.
  \end{equation}
This together with Lemma~\ref{W^{1,2}-est} yields  $\|\nabla \bar v\|_{L^\infty(Q_3)}\leq C(\Lambda, n)$. By rescaling back from $Q_3$ to $Q_{3r}$, we obtain
\begin{equation}\label{localized-L^infinity-est-I}
\|\nabla (\lambda v)\|_{L^\infty(Q_{3 r})}\leq \frac{C(\Lambda, n)}{r}.
\end{equation}
On the other hand,  \eqref{handle-gradients-v} also   gives
\begin{equation}\label{localized-L^infinity-est-II}
 \|\nabla v\|_{L^\infty(Q_{3 r})}\leq  C(\Lambda, n)  \left(\fint_{Q_4}{|\nabla v(rx, r^2 t)|^2 \, dxdt}\right)^{\frac{1}{2}}
 =C(\Lambda, n)  \left(\fint_{Q_{4 r}}{|\nabla v(y,s)|^2 \, dyds}\right)^{\frac{1}{2}}.
\end{equation}
It follows from \eqref{Cauchy--Schwarz-step}, \eqref{localized-L^infinity-est-I} and \eqref{localized-L^infinity-est-II}  that 
\begin{align*}
 \frac{\Lambda^{-1}}{4}\int_{Q_{4 r}} (1 + \lambda u) |\nabla w|^2\varphi^2 dxdt
 &\leq \frac{C}{r^2}\int_{Q_{4 r}}    w^2 \, dxdt + \frac{C}{r} \int_{Q_{4 r}}  |\nabla w|\varphi^2 |w| \, dxdt\\
&\quad +C \left(\fint_{Q_{4 r}}{|\nabla v|^2 \, dxdt}\right) \left( \int_{Q_{4 r}}|\A-\bar{\A}_{B_{4 r}}(t)|^2 \, dxdt \right) + \int_{Q_{4 r}}  |c|  |w| \, dxdt,
\end{align*}  
which together with  the Cauchy--Schwarz inequality yields
\begin{align}\label{u-v:gradient-est}
\fint_{Q_{4 r}}  |\nabla w|^2\varphi^2 dxdt
&\leq \frac{C}{r^2} \fint_{Q_{4 r}}  w^2 \, dxdt + C \left(\fint_{Q_{4 r}}{|\nabla v|^2 \, dxdt}\right) \left( \fint_{Q_{4 r}}|\A-\bar{\A}_{B_{4 r}}(t)|^2 \, dxdt \right)\\
& \quad  +\Big(\fint_{Q_{4 r}} c^2  \, dxdt\Big)^{\frac{1}{2}}\Big(\fint_{Q_{4 r}}  w^2  \, dxdt\Big)^{\frac{1}{2}}.\nonumber
\end{align}
Next notice that we can assume $\delta<\e^2$. 
Then by using \eqref{localized-compare-uv} and \eqref{rescaled-gradient-v-controlled}, we have
\begin{equation*}
\fint_{Q_{4 r}}  w^2 \, dxdt=\fint_{Q_{4 r}}{|u - v|^2\, dx dt}\leq \e^2 r^2
\end{equation*}
and
\begin{equation*}
\fint_{Q_{4r}} |\nabla v|^2\, dx dt \leq 4^{n+2} 2 \omega_n \left[2+ 66\Lambda\Big(18 \Lambda^2 + 4  \e^2 \Big)\right]\leq C'  (1+\e^2).
\end{equation*}
By combining these with \eqref{u-v:gradient-est} and \eqref{smallness-condition-c} we get
\begin{align*}
\fint_{Q_{4 r}}  |\nabla w|^2\varphi^2 dxdt
\leq C  \e^2 + C  C' (1+ \e^2)\e^2
+\e^2\leq\big( C+ 2 C C' +1\big)\e^2,
\end{align*}
where $C, C'>0$ depend only on $\Lambda$ and $n$. Thus $
\fint_{Q_{2 r}}  |\nabla w|^2 dxdt
\leq C(\Lambda, n)\e^2$ and
the proof is complete.
\end{proof}

% \begin{lemma}\label{lm:compare--gradient-solution}
% Assume $0<\theta\leq \lambda$, $\theta\leq 1$ and $0<r\leq 1$. For any $\e>0$, there exists $\delta>0$ small depending only on $\e$, $\Lambda$ and  $n$  such that:  if
% \begin{equation}\label{smallness}
% \fint_{Q_{4 r}}  \Big[|\A(x) - \bar{\A}_{B_{4r}}|^2 + |c(x,t)|^2\Big]\, dx dt \leq \delta,
% \end{equation}
% then for any weak solution $u\in W_*(Q_{5 r})$ of \eqref{PEQ} in $Q_{5 r}$ satisfying
% \begin{equation*}
% 0\leq u \leq \frac{1}{\lambda}\, \mbox{ in } Q_{4 r} \quad\mbox{and}\quad \fint_{Q_{4 r}}{|\nabla u|^2 \, dxdt}\leq 1,
% \end{equation*}
% we have 
% Here $v\in W_*(Q_{4 r})$ is a weak solution of  
% the Dirichlet problem \eqref{limiting-eqn}.
% \end{lemma}
% \begin{proof}
% \end{proof}

\begin{remark}\label{rm:translation-invariant}
Since our equations are invariant under the translation $(x,t)\mapsto (x+ y, t+s)$, Lemma~\ref{lm:localized-compare-solution} still holds true if $Q_r$ is replaced by $Q_r(y,s)$.
\end{remark}

%\myclearpage

\subsubsection{Interior density and gradient estimates}\label{interior-density-gradient}
We will derive interior $W^{1,p}$-estimates for solution $u$ of \eqref{PEQ} by estimating the distribution functions of the maximal function of $|\nabla u|^2$. The precise maximal operators will be used are:
\begin{definition}
{\rm (i)} The parabolic-Littlewood maximal function of a locally integrable function 
$f$ on $\R^n\times \R$ is defined by
\[
 (\M f)(x,t) = \sup_{\rho>0}\fint_{Q_\rho(x,t)}{|f(y,s)|\, dy ds}.
\]
{\rm (ii)} If $f$ is defined in a region $U\subset \R^n\times \R$, then we   denote
\[
 \M_U f = \M (\chi_U f).
\]
 \end{definition}
  
The next result gives a density estimate for the distribution  of $\M_{Q_5}(|\nabla u|^2)$.
\begin{lemma}\label{initial-density-est}
Assume that $0<\theta\leq \lambda$, $\theta\leq 1$, $\A$ satisfies \eqref{interior-elipticity} and $c\in L^2(Q_6)$. There exists a constant $N>1$ depending only on $\Lambda$ and $n$ such that   
for  any $\e>0$, we can find  $\delta=\delta(\e,\Lambda, n)>0$  satisfying:  if
\begin{equation}\label{interior-SMO}
\sup_{0<\rho\leq 4}\sup_{(y,s)\in Q_1} \fint_{Q_\rho(y, s)}{|\A(x, t) - \bar{\A}_{B_\rho(y)}(t)|^2\, dx dt}\leq \delta,
\end{equation}
then
% for any $\A\in \calA_{\delta, 4}
  for any weak solution $u\in \mathcal W(Q_6)$ of \eqref{PEQ} with
%\begin{equation*}
$0\leq u \leq \frac{1}{\lambda}$ in $Q_5$
%\end{equation*}
and for any $(y,s)\in Q_1$, $0<r\leq 1$ with
\begin{equation}\label{one-point-condition}
 Q_r(y,s)\cap Q_1\cap \big\{Q_5:\, \M_{Q_5}(|\nabla u|^2)\leq 1 \big\}\cap \{ Q_5: \M_{Q_5}(c^2)\leq \delta\}\neq \emptyset,
\end{equation}
we have
\begin{equation*}
 \big| \{ Q_1:\, \M_{Q_5}(|\nabla u|^2)> N \}\cap Q_r(y,s)\big|
\leq \e  |Q_r(y,s)|.
\end{equation*}
\end{lemma}
\begin{proof}
By condition \eqref{one-point-condition}, there exists a point $(x_0,t_0)\in Q_r(y,s)\cap  Q_1$ such that
\begin{align}
  \M_{Q_5}(|\nabla u|^2)(x_0,t_0)\leq 1\quad \mbox{and}\quad \M_{Q_5}(c^2)(x_0,t_0) \leq
\delta.\label{maximal-fns-control}
\end{align}
Notice that $Q_{5 r}(y,s) \subset Q_6$. Since $Q_{4 r}(y,s) \subset Q_{5 r}(x_0,t_0) \cap Q_5$, it follows from \eqref{maximal-fns-control} that
\begin{align*}
&\fint_{Q_{4 r}(y,s)}|\nabla u|^2 \, dx dt \leq \frac{|Q_{5 r}(x_0,t_0)|}{|Q_{4 r}(y,s)|} \frac{1}{|Q_{5 r}(x_0,t_0)|}\int_{Q_{5 r}(x_0,t_0)\cap Q_5}|\nabla u|^2 \, dx dt\leq \Big(\frac{5}{4}\Big)^{n+2},\\
 &\fint_{Q_{4 r}(y,s)}c^2 \, dx dt \leq \frac{|Q_{5 r}(x_0,t_0)|}{|Q_{4 r}(y,s)|} \frac{1}{|Q_{5 r}(x_0,t_0)|}\int_{Q_{5 r}(x_0,t_0)\cap Q_5}c^2 \, dx dt\leq \Big(\frac{5}{4}\Big)^{n+2} \delta.
\end{align*}
Also the assumption \eqref{interior-SMO} gives
\begin{align*}
 &\fint_{Q_{4r}(y, s)} |\A(x,t)-\bar{\A}_{B_{4r}(y)}(t) |^2 \, dx dt  \leq  \delta.
\end{align*}
Therefore, we can use \eqref{compare-graduv} and Remark~\ref{rm:translation-invariant} to obtain
\begin{equation}\label{gradient-comparison}
\fint_{Q_{2r}(y,s)}{|\nabla u- \nabla v|^2\, dx dt}\leq \eta^2,
\end{equation}
where $v \in \calW(Q_{4r}(y,s))$ is the  unique weak solution of 
\begin{equation*}
\left \{
\begin{array}{lcll}
v_t  &=&  \nabla\cdot[(1+\lambda v)\bar{\A}_{B_{4r}(y)}(t)\nabla v] + \theta^2 v(1-\lambda v) \quad &\text{in}\quad Q_{4r}(y,s), \\
v & =& u\quad &\text{on}\quad \partial_p Q_{4r}(y,s)
\end{array}\right.
\end{equation*}
satisfying $0\leq v\leq 1/\lambda$ in $Q_{4r}(y,s)$,   and $\delta =\delta(\eta,\Lambda, n)$ with $\eta$ being determined later. We remark that the existence and uniqueness 
of such weak solution $v$ is guaranteed by Theorem~\ref{existence-u}.

Let $\bar{v}(x,t) = \lambda v(rx +y, r^2 t +s)$ and $\A'(x,t) = \A(rx +y, r^2 t +s)$ for $(x,t)\in Q_4$.  Then $0\leq \bar{v}\leq 1$ in $Q_4$ and  $\bar{v}$ is a weak solution of
\begin{equation*}
\bar{v}_t  =  \nabla\cdot[(1+\bar{v})\bar{\A'}_{B_{4}}(t) \nabla \bar{v}] + (\theta r)^2 \bar{v}(1-\bar{v})\quad\mbox{in}\quad Q_4.
\end{equation*}
Since $\theta r\leq \theta\leq 1$, applying Lemma~\ref{W1infty-est} we get 
\begin{align*}
\|\nabla \bar v\|_{L^\infty(Q_{\frac{3}{2}})}^2\leq C \fint_{Q_2}{|\nabla \bar{v}|^2 \, dxdt},
\end{align*}
which together with \eqref{gradient-comparison} and \eqref{maximal-fns-control}   gives
\begin{align}\label{gradient-is-bounded}
\|\nabla v\|_{L^\infty(Q_{\frac{3 r}{2}}(y,s))}^2
&\leq C \fint_{Q_{2r}(y,s)}{|\nabla v|^2 \, dxdt}\\
&\leq  2 C \left( \fint_{Q_{2r}(y,s)}{|\nabla u -\nabla v|^2 \, dxdt}
+ \fint_{Q_{2r}(y,s)}{|\nabla u|^2 \, dxdt}\right)
\leq C(\Lambda, n) (\eta^2 +1).\nonumber
\end{align}
We claim that \eqref{maximal-fns-control}, \eqref{gradient-comparison} and \eqref{gradient-is-bounded}
yield
\begin{equation}\label{set-relation-claim}
 \big\{ Q_r(y,s):  \M_{Q_{2r}(y,s)}(|\nabla u - \nabla v|^2) \leq C(\Lambda, n) \big\}\subset \big\{ Q_r(y,s):\, \M_{Q_5}(|\nabla u |^2) \leq N\big\}
\end{equation}
with $N = \max{\{6 C(\Lambda, n), 5^{n+2}\}}$. Indeed, let $(x,t)$ be a point in the set on the left hand side of \eqref{set-relation-claim}, and consider 
$Q_\rho(x,t)$. If $\rho\leq r/2$, then $Q_\rho(x,t) \subset Q_{3 r/2}(y,s)\subset Q_3$
and hence
\begin{align*}
 &\frac{1}{|Q_\rho(x,t)|}\int_{Q_\rho(x,t) \cap Q_5} |\nabla u|^2 \, dx dt\\
 &\leq 
 \frac{2}{|Q_\rho(x,t)|}\Big[\int_{Q_\rho(x,t) \cap Q_5} |\nabla u-\nabla v|^2 \, dx dt
 +\int_{Q_\rho(x,t) \cap Q_5} |\nabla v|^2 \, dx dt \Big]\\
 &\leq 2 \M_{Q_{2r}(y,s)}(|\nabla u - \nabla v|^2)(x,t) +2 \|\nabla v \|_{L^\infty(Q_{\frac{3r}{2}}(y,s))}^2
 \leq 2C(\Lambda, n) \big( \eta^2 + 2\big)\leq 6 C(\Lambda, n).
\end{align*}
On the other hand if $\rho>r/2$, then  $Q_\rho(x,t)\subset Q_{5\rho}(x_0, t_0)$. This and  the first inequality in \eqref{maximal-fns-control} imply that
\begin{align*}
 \frac{1}{|Q_\rho(x,t)|}\int_{Q_\rho(x,t) \cap Q_5} |\nabla u|^2 \, dx dt
 \leq \frac{5^{n+2}}{|Q_{5\rho}(x_0, t_0)|} \int_{Q_{5\rho}(x_0, t_0) \cap Q_5} |\nabla u|^2 \, dx dt\leq 
 5^{n+2}.
\end{align*}
Therefore, $\M_{Q_5}(|\nabla u |^2)(x,t) \leq N$ and the claim \eqref{set-relation-claim} is proved.
Note that   \eqref{set-relation-claim} is equivalent to
 \begin{align*}
 \big\{Q_r(y,s):\, \M_{Q_5}(|\nabla u |^2) > N\big\} \subset
 \big\{ Q_r(y,s):\, \M_{Q_{2r}(y,s)}(|\nabla u - \nabla v|^2) >C(\Lambda, n) \big\}.
 \end{align*}
It follows from this,  the weak type $1-1$ estimate and \eqref{gradient-comparison}   that
\begin{align*}
 &\big|\big\{Q_r(y,s):\, \M_{Q_5}(|\nabla u |^2) > N\big\} \big|\leq  \big|
 \big\{ Q_r(y,s):\, \M_{Q_{2r}(y,s)}(|\nabla u - \nabla v|^2) > C(\Lambda, n) \big\}\big|\\
 &\leq C  \int_{Q_{2r}(y,s)}{|\nabla u - \nabla v|^2 \, dx dt}\leq C' \eta^2 \, |Q_r(y,s)|,
 \end{align*}
 where $C'>0$ depends only on $\Lambda$ and $n$.
By choosing $\eta = \sqrt{\frac{\e}{C'}}$, we obtain the desired result.
\end{proof}

 In view of Lemma~\ref{initial-density-est}, we can apply the Vitali covering lemma (see \cite[Lemma~2.4]{B2}) 
for $E=\{ Q_1:\, \M_{Q_5}(|\nabla u|^2)> N \} $ and 
$F=\{ Q_1:\, \M_{Q_5}(|\nabla u|^2)> 1 \}\cup \{ Q_1:\, \M_{Q_5}(c^2)> \delta \}
$
 to obtain:
\begin{lemma}\label{second-density-est}
Assume that $0<\theta\leq \lambda$, $\theta\leq 1$, $\A$ satisfies \eqref{interior-elipticity} and $c\in L^2(Q_6)$. There exists a constant $N>1$ depending only on $\Lambda$ and  $n$ such that   
for  any $\e>0$, we can find  $\delta=\delta(\e, \Lambda, n)>0$  satisfying: if
\begin{equation*}
\sup_{0<\rho\leq 4}\sup_{(y,s)\in Q_1} \fint_{Q_\rho(y,s)}{|\A(x, t) - \bar{\A}_{B_\rho(y)}(t)|^2\, dx dt}\leq \delta,
\end{equation*}
 then  for any weak solution $u\in \mathcal W(Q_6)$ of \eqref{PEQ} satisfying
\begin{equation*}
0\leq u \leq \frac{1}{\lambda}\quad\mbox{in}\quad Q_5 \quad \mbox{and}\quad
\big| \{Q_1: 
\M_{Q_5}(|\nabla u|^2)> N \}\big| \leq \e |Q_1|,
\end{equation*}
we have
\begin{align*}
\big|\{Q_1: \M_{Q_5}(|\nabla u|^2)> N\}\big|
\leq 2 (10)^{n+2}\e \, \Big\{
\big|\{Q_1: \M_{Q_5}(|\nabla u|^2)> 1\}\big|
+ \big|\{ Q_1: \M_{Q_5}(c^2)> \delta \}\big|\Big\}.\nonumber
\end{align*}
\end{lemma}
 
We are now ready to prove  Theorem~\ref{first-main-result}.

\begin{proof}[\textbf{Proof of Theorem~\ref{first-main-result}}]
Let  $N>1$ be as in  Lemma~\ref{second-density-est}, and let $q=p/2>1$ . We choose  $\e=\e(p,\Lambda ,n)>0$ be such that
\[
\e_1 \eqdef 2 (10)^{n+2}\e = \frac{1}{2 N^q}, 
\]
and let $\delta=\delta(p,\Lambda, n)$ be the corresponding constant given by Lemma~\ref{second-density-est}.

Assuming for a moment that $u$  satisfies 
\begin{equation}\label{initial-distribution-condition}
\big| \{Q_1: 
\M_{Q_5}(|\nabla u|^2)> N \}\big| \leq \e |Q_1|.
\end{equation}
We first consider the case  $\theta \leq\lambda$.
Then it follows from  Lemma~\ref{second-density-est} that
\beq\label{initial-distribution-est}
\big|\{Q_1: \M_{Q_5}(|\nabla u|^2)> N\}\big|
\leq \e_1  \left\{
\big|\{Q_1: \M_{Q_5}(|\nabla u|^2)> 1\}\big|
+ \big|\{ Q_1: \M_{Q_5}(c^2)> \delta \}\big|\right\}.
\eeq
 Let us iterate this estimate by considering
\[
u_1(x,t) = \frac{u(x,t)}{\sqrt{N}}, \quad c_1(x,t) = \frac{c(x,t)}{\sqrt{N}} \quad \mbox{and}\quad \lambda_1 = \sqrt{N}\lambda\geq \theta.
\]
It is easy to see that $u_1\in \mathcal W(Q_6)$ is a weak solution of
\begin{equation*}
(u_1)_t  =  \nabla\cdot[(1+\lambda_1 u_1)\A\nabla u_1] + \theta^2 u_1(1-\lambda_1 u_1) - \lambda_1 \theta c_1 u_1 \quad\mbox{in}\quad Q_6.
\end{equation*}
Moreover, thanks to \eqref{initial-distribution-condition} we have
\begin{align*}
\big| \{Q_1: 
\M_{Q_5}(|\nabla u_1|^2)> N \}\big| &= \big| \{Q_1: 
\M_{Q_5}(|\nabla u|^2)> N^2 \}\big| \leq \e |Q_1|.
\end{align*}
Therefore, by applying  Lemma~\ref{second-density-est} to $u_1$ we obtain
\begin{align*}
\big|\{Q_1: \M_{Q_5}(|\nabla u_1|^2)> N\}\big|
&\leq \e_1 \left( 
\big|\{Q_1: \M_{Q_5}(|\nabla u_1|^2)> 1 \}\big|
+ \big|\{ Q_1: \M_{Q_5}(|c_1|^2)> \delta \}\big|\right)\\
&= \e_1  \left( 
\big|\{Q_1: \M_{Q_5}(|\nabla u|^2)> N \}\big|
+ \big|\{ Q_1: \M_{Q_5}(c^2)> \delta N\}\big| \right).
\end{align*}
We infer from this and  \eqref{initial-distribution-est} that
\begin{align}\label{first-iteration-est}
&\big|\{Q_1: \M_{Q_5}(|\nabla u|^2)> N^2\}\big|
\leq \e_1^2 
\big|\{Q_1: \M_{Q_5}(|\nabla u|^2)> 1 \}\big|\\
&\qquad + \e_1^2\big|\{ Q_1: \M_{Q_5}(c^2)> \delta\} \big|+ \e_1\big|\{ Q_1: \M_{Q_5}(c^2)> \delta N\}\big|. \nonumber
\end{align}
 Next, let
\[
u_2(x,t) = \frac{u(x,t)}{N}, \quad c_2(x,t) = \frac{c(x,t)}{N} \quad \mbox{and}\quad \lambda_2 = N\lambda\geq \theta.
\]
Then  $u_2\in \mathcal W(Q_6)$ is a weak solution of
\begin{equation*}
(u_2)_t  =  \nabla\cdot[(1+\lambda_2 u_2)\A\nabla u_2] + \theta^2 u_2(1-\lambda_2 u_2) - \lambda_2 \theta c_2 u_2 \quad\mbox{in}\quad Q_6
\end{equation*}
and 
\begin{align*}
\big| \{Q_1: 
\M_{Q_5}(|\nabla u_2|^2)> N \}\big| &= \big| \{Q_1: 
\M_{Q_5}(|\nabla u|^2)> N^3 \}\big| \leq \e  |Q_1|.
\end{align*}
Hence by applying  Lemma~\ref{second-density-est} to $u_2$ we get
\begin{align*}
\big|\{Q_1: \M_{Q_5}(|\nabla u_2|^2)> N\}\big|
&\leq \e_1 \left( 
\big|\{Q_1: \M_{Q_5}(|\nabla u_2|^2)> 1 \}\big|
+ \big|\{ Q_1: \M_{Q_5}(|c_2|^2)> \delta\}\big|\right)\\
&= \e_1  \left( 
\big|\{Q_1: \M_{Q_5}(|\nabla u|^2)> N^2 \}\big|
+ \big|\{ Q_1: \M_{Q_5}(c^2)> \delta N^2\}\big| \right).
\end{align*}
This together with  \eqref{first-iteration-est} gives
\begin{align*}
\big|\{Q_1: \M_{Q_5}(|\nabla u|^2)> N^3\}\big|
\leq \e_1^3 
\big|\{Q_1: \M_{Q_5}(|\nabla u|^2)> 1 \}\big|
+ \sum_{i=1}^3\e_1^i\big|\{ Q_1: \M_{Q_5}(c^2)> \delta N^{3-i}\} \big|.
\end{align*}
By repeating the iteration, we then conclude that
\begin{align*}
\big|\{Q_1: \M_{Q_5}(|\nabla u|^2)> N^k\}\big|
\leq \e_1^k 
\big|\{Q_1: \M_{Q_5}(|\nabla u|^2)> 1 \}\big|+ \sum_{i=1}^k\e_1^i\big|\{ Q_1: \M_{Q_5}(c^2)> \delta N^{k-i}\} \big|
\end{align*}
for all $k=1,2,\dots$ Since 
\begin{align*}
&\int_{Q_1}\M_{Q_5}(|\nabla u|^2)^q \, dx dt =q \int_0^\infty t^{q-1} \big|\{Q_1: \M_{Q_5}(|\nabla u|^2)>t\}\big|\, dt\\
&=q \int_0^{N} t^{q-1} \big|\{Q_1: \M_{Q_5}(|\nabla u|^2)>t\}\big|\, dt
+q \sum_{k=1}^\infty\int_{N^{k}}^{N^{k+1}} t^{q-1} \big|\{Q_1: \M_{Q_5}(|\nabla u|^2)>t\}\big|\, dt\\
&\leq N^q |Q_1| + (N^q -1) \sum_{k=1}^\infty N^{qk}\big|\{Q_1: \M_{Q_5}(|\nabla u|^2)>N^k\}\big|,
\end{align*}
we obtain 
\begin{align*}
&\int_{Q_1}\M_{Q_5}(|\nabla u|^2)^q \, dx dt 
\leq N^q |Q_1| + (N^q -1) \sum_{k=1}^\infty N^{qk}\big|\{Q_1: \M_{Q_5}(|\nabla u|^2)>N^k\}\big|\\
&\leq N^q |Q_1| + (N^q -1)|Q_1| \sum_{k=1}^\infty (\e_1 N^q)^k
+\sum_{k=1}^\infty\sum_{i=1}^k (N^q -1)N^{q k}\e_1^i\big|\{ Q_1: \M_{Q_5}(c^2)> \delta N^{k-i}\} \big|.
\end{align*}
But we have
\begin{align*}
&\sum_{k=1}^\infty\sum_{i=1}^k (N^q -1)N^{q k}\e_1^i\big|\{ Q_1: \M_{Q_5}(c^2)> \delta N^{k-i}\} \big|\\
&=\big(\frac{N}{\delta}\big)^q\sum_{i=1}^\infty (\e_1 N^q)^i\left[\sum_{k=i}^\infty (N^q -1)\delta^{q}N^{q (k-i-1)}\big|\{ Q_1: \M_{Q_5}(c^2)> \delta N^{k-i}\} \big|\right]\\
&=\big(\frac{N}{\delta}\big)^q\sum_{i=1}^\infty (\e_1 N^q)^i\left[\sum_{j=0}^\infty (N^q -1)\delta^{q}N^{q (j-1)}\big|\{ Q_1: \M_{Q_5}(c^2)> \delta N^{j}\} \big|\right]\\
&\leq \big(\frac{N}{\delta}\big)^q \Big[\int_{Q_1}\M_{Q_5}(c^2)^q \, dx dt\Big]\sum_{i=1}^\infty (\e_1 N^q)^i,
\end{align*}
where we have used Remark~\ref{rm:lower-est-L^p-norm} below to get the last inequality.
Thus we infer that
\begin{align*}
\int_{Q_1}\M_{Q_5}(|\nabla u|^2)^q \, dx dt 
&\leq N^q |Q_1| + \left[ (N^q -1)|Q_1| +\big(\frac{N}{\delta}\big)^q \int_{Q_1}\M_{Q_5}(c^2)^q \, dx dt\right] \sum_{k=1}^\infty (\e_1 N^q)^k\\
&= N^q |Q_1| + \left[ (N^q -1)|Q_1| +\big(\frac{N}{\delta}\big)^q \int_{Q_1}\M_{Q_5}(c^2)^q \, dx dt\right] \sum_{k=1}^\infty 2^{-k}\\
&\leq C\left( 1+ \int_{Q_1}\M_{Q_5}(c^2)^q \, dx dt \right)
\end{align*}
with the constant $C$ depending only on $p$, $\Lambda$ and $n$. On the other hand, by the Lebesgue differentiation theorem one has
\[
|\nabla u(x,t)|^2 =
\lim_{\rho\to 0^+}\fint_{Q_\rho(x,t)} |\nabla u(y,s)|^2\, dy ds\leq  
\M_{Q_5}(|\nabla u|^2)(x,t)
\]
for almost every $(x,t)\in Q_1$. Therefore,
it follows from the strong type $q-q$ estimate for the maximal function and the fact $q=p/2$ that 
\begin{equation}\label{initial-desired-L^p-estimate}
\int_{Q_1}|\nabla u|^{p} \, dx dt\leq C\left( 1+ \int_{Q_5}|c|^{p} \, dx dt\right).
\end{equation}
The estimate  \eqref{initial-desired-L^p-estimate} was derived under the assumption that $\theta\leq \lambda$. In the case $\theta>\lambda$, we define  $u' = u/K$, $c' = c/K$ and $\lambda' = \lambda K$ where $K= \theta/\lambda>1$. 
Then  $u'$ is a weak solution of
\begin{equation*}
u'_t  =  \nabla\cdot[(1+\lambda' u')\A \nabla u'] + \theta^2 u'(1-\lambda' u') - \lambda' \theta c' u' \quad\mbox{in}\quad Q_6.
\end{equation*}
Since $\theta\leq \lambda'$ and $u'$ inherits the property \eqref{initial-distribution-condition} from that of $u$, we can employ \eqref{initial-desired-L^p-estimate} to conclude that
 \begin{equation*}
\int_{Q_1}|\nabla u'|^{p} \, dx dt\leq C\left( 1+ \int_{Q_5}|c'|^{p} \, dx dt\right).
\end{equation*}
This implies that
\begin{equation}\label{second-desired-L^p-estimate}
\int_{Q_1}|\nabla u|^p \, dx dt\leq C\left[ \Big(\frac{\theta}{\lambda}\Big)^p+ \int_{Q_5}|c|^p \, dx dt\right].
\end{equation}
Combining \eqref{initial-desired-L^p-estimate} and \eqref{second-desired-L^p-estimate} yields
\begin{equation}\label{desired-L^p-estimate}
\int_{Q_1}|\nabla u|^p \, dx dt\leq C\left[ \Big(\frac{\theta}{\lambda}\Big)^p \vee 1+ \int_{Q_5}|c|^p \, dx dt\right]
\end{equation}
as long as $\lambda>0$ and $0<\theta\leq 1$.
We next remove the extra assumption  
\eqref{initial-distribution-condition} for $u$. Notice that for any $M>0$, by using the weak type $1-1$ estimate for the maximal function and Lemma~\ref{W^{1,2}-est} we get
\begin{align*}
&\big| \{Q_1: 
\M_{Q_5}(|\nabla u|^2)> N  M^2\}\big|
\leq \frac{C }{N M^2}\int_{Q_5} |\nabla u|^2 \, dxdt 
\leq \frac{C_n }{M^2}\int_{Q_6} u^2 \, dxdt.
\end{align*}
Therefore, if we let
\[\bar{u}(x,t) = \frac{u(x,t)}{M}\quad \mbox{with}\quad M^2= \frac{C_n  \|u\|_{L^2(Q_6)}^2}{ \e |Q_1|}
\]
then 
$\big| \{Q_1: 
\M_{Q_5}(|\nabla \bar{u}|^2)> N\}\big|
\leq \e |Q_1|$. Hence we can apply \eqref{desired-L^p-estimate} to $\bar{u}$ with $c$ and $\lambda$ being replaced by $\bar c= c/M$ and $\bar \lambda = \lambda M$. By reversing back to the  functions $u$ and $c$, we obtain \eqref{main-estimate}.
%\begin{equation*}\label{main-
%estimate-PartI}
%\int_{Q_1}|\nabla u|^p \, dx dt\leq %C\left\{ \Big(\frac{\theta}{\lambda} %\vee  \|u\|_{L^2(Q_6)}\Big)^p + %\int_{Q_5}|c|^p \, dx dt\right\}.
%\end{equation*}
%In the above process, one could  %choose to  remove the restriction %\eqref{initial-distribution-
%condition}  before removing the %condition $\theta \leq \lambda$. That %is, we consider the function $\bar u$ %first (with the constant $M$ is now %chosen as $M^2:= \Big(\frac{\theta}
%{\lambda}\Big)^2 \vee \frac{C_n  \|%u\|_{L^2(Q_6)}^2}{ \e |Q_1|}  $
%to make sure that $\bar{\lambda} %:=\lambda M\geq \theta$) and then %apply the resulting estimate for the %function $u'$.
%As a consequence, one gets 
%the following alternative estimate %for $u$:
%\begin{equation}\label{main-estimate-%PartII}
%\int_{Q_1}|\nabla u|^p \, dx dt\leq %C\left\{ \Big(\frac{\theta}{\lambda} %\Big)^p  \Big(\frac{\theta}{\lambda} %\vee  \|u\|_{L^2(Q_6)}\Big)^p+ %\int_{Q_5}|c|^p \, dx dt\right\}.
%\end{equation}
%The   estimate \eqref{main-estimate}
%now follows by taking the minimum %value  between  the two quantities on %the right hand side of \eqref{main-
%estimate-PartI}  and \eqref{main-
%estimate-PartII}.
\end{proof}

%\begin{remark} By interchanging the %order of rescaling between $u'$ and %$\bar{u}$ at the end of the proof of %Theorem~\ref{main-result}, we obtain 
%the following alternative estimate for %$u$:
%\begin{equation*}
%\int_{Q_1}|\nabla u|^p \, dx dt\leq %C\left\{ \Big[\frac{\theta}{\lambda}  %\Big( \|u\|_{L^2(Q_7)} \vee %1\Big)\Big]^p + \int_{Q_6}|c|^p \, dx %dt\right\}.
%\end{equation*}
%\end{remark}

\begin{remark}\label{rm:lower-est-L^p-norm}
Assume that $V\subset U\subset \R^n\times\R$,  $c\in L^2(U)$ and $q>1$. 
Then for any $\delta>0$ and $N>1$, we have 
\beq
\sum_{j=0}^\infty (N^q -1) \delta^{ q} N^{q(j-1)} \big|\{ V: \M_{U}(c^2)> \delta N^j\} \big|
\leq \int_{V}\M_{U}(c^2)^q \, dx dt.
\eeq  
Indeed,
\begin{align*}
\int_{V}\M_{U}(c^2)^q \, dx dt
&=q \int_0^\infty t^{q-1} \big|\{V: \M_{U}(c^2)>t\}\big|\, dt
\geq 
q \sum_{j=0}^\infty\int_{\delta N^{j-1}}^{\delta N^j} t^{q-1} \big|\{V: \M_{U}(c^2)>t\}\big|\, dt\\
&\geq \sum_{j=0}^\infty \big[(\delta N^j)^q - (\delta N^{j-1})^q\big]\, \big|\{ V: \M_{U}(c^2)> \delta N^j\} \big|.
\end{align*}
\end{remark}

Note that our interior gradient estimate for $u$ in Theorem~\ref{first-main-result} is independent of the boundary 
values of $u$ on $\partial_p Q_6$. On the contrary, the interior $W^{1,p}$-estimates obtained in \cite{B1, B2}  for 
linear parabolic equations  depend essentially on the boundary values of the solutions.
%\myclearpage
%================================
\subsection{Boundary $W^{1,p}$-estimates on flat domains}\label{boundaryW}

We will use the following notation: 
\begin{align*}
B_\rho^+ &=\{x\in B_\rho:\, x_n>0\}, 
& \partial_c B_\rho^+ &= \partial B_\rho\cap \{x:\, x_n >0\},\\
T_\rho &= B_\rho\cap \{x:\, x_n =0\}, 
& \tilde T_\rho &= T_\rho \times (-\rho^2, \rho^2], \\
Q_\rho &= B_\rho \times  (-\rho^2, \rho^2], 
& Q_\rho^+ &= B_\rho^+ \times  (-\rho^2, \rho^2],\\
\partial_c Q_\rho^+ &= \partial_c B_\rho^+ \times (-\rho^2, \rho^2],
& \partial_b Q_\rho^+ &=  B_\rho^+ \times \{-\rho^2\}. 
\end{align*}
Our aim is to derive boundary $W^{1,p}$-estimates for  solutions to the problem
\begin{equation}\label{MP-flat-domain}
\left \{
\begin{array}{lcll}
u_t  &=&  \nabla\cdot[(1+\lambda u)\A \nabla u] + \theta^2 u(1-\lambda u) - \lambda \theta c u  \quad &\text{in}\quad Q_4^+, \\
\begin{displaystyle}
\frac{\partial u}{\partial \bnu} 
\end{displaystyle}
 &=& 0 \quad &\text{on}\quad \tilde T_4, 
\end{array}\right.
\end{equation}
where  $\theta,\, \lambda>0$ are constants and $c(x,t)$ is a non-negative measurable function. We assume that $\A: Q_4^+ \to  \M^{n\times n}$ is  symmetric, measurable  and there exists  a constant $\Lambda>0$ such that
\begin{equation}\label{boundary-elipticity}
\Lambda^{-1} |\xi|^2 \leq \xi^T \A(x,t) \xi  \leq \Lambda |\xi|^2\quad \mbox{for a.e. $(x,t)\in Q_4^+$ and for all }\xi\in\R^n.
\end{equation}
Throughout this subsection, the space $\hat\calW(Q_4^+)$ is defined as in \eqref{space-hat} with $\Omega=Q_4^+$ and $\Gamma = \partial_c B_4^+$. Note also that  in this case  $\partial_D Q_4^+ = \partial_b Q_4^+ \cup \partial_c Q_4^+$.
\subsubsection{Boundary approximation estimates} \label{boundary-approx}

Let us consider the parabolic equation
\begin{equation}\label{MP-HS}
\left \{
\begin{array}{lcll}
u_t  &=&  \nabla\cdot[(1+\lambda u)\A \nabla u] + \theta^2 u(1-\lambda u) - \lambda \theta c u  \quad &\text{in}\quad Q_4^+, \\
\begin{displaystyle}
\frac{\partial u}{\partial \bnu} 
\end{displaystyle}
 &=& 0 \quad &\text{on}\quad \tilde T_4, \\
u & =& \psi  \quad &\text{on}\quad \partial_b Q_4^+ \cup \partial_c Q_4^+.
\end{array}\right.
\end{equation}
 Observe that  $u$ is a weak solution of \eqref{MP-HS} iff the function 
$
\bar{u}  \eqdef \lambda u
$
is a weak solution of
\begin{equation}\label{RMP-HS}
\left \{
\begin{array}{lcll}
\bar{u}_t  &=&  \nabla\cdot[(1+\bar{u} )\A \nabla \bar{u} ] + \theta^2 \bar{u} (1- \bar{u} ) - \lambda \theta c \bar{u}  \quad &\text{in}\quad Q_4^+, \\
\begin{displaystyle}
\frac{\partial \bar u}{\partial \bnu} 
\end{displaystyle}
 &=& 0 \quad &\text{on}\quad \tilde T_4, \\
\bar u& =& \bar\psi \eqdef \lambda\psi  \quad &\text{on}\quad \partial_b Q_4^+ \cup \partial_c Q_4^+.
\end{array}\right.
\end{equation}
We will establish boundary $W^{1,p}$ estimates for solutions to \eqref{MP-HS} by freezing its coefficient and comparing it to solutions  of the equation
\begin{equation}\label{REQ-HS}
\left \{
\begin{array}{lcll}
v_t  &=&  \nabla\cdot[(1+\lambda v)\bar{\A}_{B_4^+}(t)\nabla v] + \theta^2 v(1-\lambda v) \quad &\text{in}\quad Q_4^+, \\
\begin{displaystyle}
\frac{\partial v}{\partial \bnu} 
\end{displaystyle}
 &=& 0 \quad &\text{on}\quad \tilde T_4, \\
v & =& u  \quad &\text{on}\quad \partial_b Q_4^+ \cup \partial_c Q_4^+.
\end{array}\right.
\end{equation}
Notice that $v$ is a weak solution of \eqref{REQ-HS} iff the function 
$
\bar{v}  \eqdef \lambda v
$
is a weak solution of 
\begin{equation}\label{UEQ-HS}
\left \{
\begin{array}{lcll}
\bar{v}_t  &=&  \nabla\cdot[(1+\bar{v})
\bar{\A}_{B_4^+}(t) \nabla \bar{v}] + \theta^2 \bar{v}(1-\bar{v}) \quad &\text{in}\quad Q_4^+, \\
\begin{displaystyle}
\frac{\partial \bar v}{\partial \bnu} 
\end{displaystyle}
 &=& 0 \quad &\text{on}\quad \tilde T_4, \\
\bar v & =& \bar u \quad &\text{on}\quad \partial_b Q_4^+ \cup \partial_c Q_4^+.
\end{array}\right.
\end{equation}
By arguing similarly to the proof of Lemma~\ref{W^{1,2}-est}, we have:
\begin{lemma}\label{G-W^{1,2}-est} Assume that $\lambda, \theta>0$, $\A$ satisfies \eqref{boundary-elipticity} and $c$ is a non-negative measurable function on $Q_4^+$. Let $u\in \hat\calW(Q_4^+)$ be a  non-negative weak solution of \eqref{MP-flat-domain}. 
%\begin{equation*}
%\left \{
%\begin{array}{lcll}
%u_t  &=&  \nabla\cdot[(1+\lambda u)\A \nabla u] + %\theta^2 u(1-\lambda u) - \lambda \theta c u  \quad %&\text{in}\quad Q_4^+, \\
%\begin{displaystyle}
%\frac{\partial u}{\partial \bnu} 
%\end{displaystyle}
% &=& 0 \quad &\text{on}\quad \tilde T_4.
%\end{array}\right.
%\end{equation*}
Then there exists a constant $C>0$ depending only on $\Lambda$ and $n$ such that
\begin{equation*}
\int_{Q_2^+} (1 + \lambda u) |\nabla u|^2\, dxdt
\leq C\int_{Q_3^+} (1+\lambda u +\theta^2)  u^2\, dxdt.
\end{equation*}
\end{lemma}
We will need the following boundary $W^{1,\infty}$-estimate for solutions of the reference equation.
\begin{lemma}\label{G-W^{1,infty}-est} 
Assume that $0<\theta\leq 1$ and $\A_0: (-16, 16]\to \M^{n\times n}$ is  a measurable matrix-valued function satisfying \eqref{eigenvalues-controlled}. Let $\bar{v}\in \hat\calW(Q_4^+)$ be a weak solution of 
\begin{equation*}
\left \{
\begin{array}{lcll}
\bar{v}_t  &=&  \nabla\cdot[(1+\bar{v})\A_0(t)\nabla \bar{v}] + \theta^2 \bar{v}(1-\bar{v}) \quad &\text{in}\quad Q_4^+, \\
\begin{displaystyle}
\frac{\partial\bar v}{\partial \bnu} 
\end{displaystyle}
 &=& 0 \quad &\text{on}\quad \tilde T_4
\end{array}\right.
\end{equation*}
satisfying  $0\leq \bar{v} \leq 1$ in $Q_4^+$. Then
there exists $C=C(\Lambda, n)>0$ such that 
\begin{equation}\label{global-lipschitz}
\|\nabla \bar v\|_{L^\infty(Q_3^+)}^2\leq C\fint_{Q_4^+}{|\nabla \bar{v}|^2 \, dxdt}.
\end{equation}
\end{lemma}
\begin{proof}
%\marginpar{Add\\a reference}
From the classical boundary regularity result, we have 
$\bar v \in C^1(\overline{Q_{\frac{7}{2}}^+})$.  Therefore, the reflected function 
\begin{equation*}
v^*(x',x_n, t) \eqdef \left \{
\begin{array}{lcll}
\bar{v}(x', x_n, t)   \quad &\text{when}\quad x_n\geq 0, \\
\bar{v}(x', - x_n, t)\quad &\text{when}\quad x_n< 0
\end{array}\right.
\end{equation*}
belongs to the class $C^1(\overline{Q_{\frac{7}{2}}})$. Consequently, it is clear that the function  $v^*$ is a weak solution of 
\[
v^*_t  =  \nabla\cdot[(1+v^*)\A_0(t) \nabla v^*] + \theta^2 v^*(1-v^*) \quad \text{in}\quad Q_{\frac{7}{2}}.
\]
Thus, by applying the interior estimate in
Lemma~\ref{W1infty-est} we obtain 
\begin{equation*}
\|\nabla v^*\|_{L^\infty(Q_3)}^2\leq C(\Lambda, n)\fint_{Q_{\frac{7}{2}}}{|\nabla v^*|^2 \, dxdt},
\end{equation*}
yielding the estimate \eqref{global-lipschitz}.
\end{proof}

\begin{lemma}\label{G-gradient-est-II} Let
$\bar{u}\in \hat\calW(Q_4^+)$ be a non-negative weak solution of \eqref{RMP-HS} and  $\bar{v}\in \hat\calW(Q_4^+)$ be a weak solution of \eqref{UEQ-HS} satisfying $0\leq \bar v\leq 1$ in $Q_4^+$. Then 
\begin{align*}
&\int_{Q_4^+}| \bar{u}- \bar{v}|^2 \, dx dt +
\Lambda^{-1} \int_{Q_4^+}|\nabla \bar{u}- \nabla \bar{v}|^2 \, dx dt\\
&\leq  33\left[ 2\Lambda^3 \int_{Q_4^+} \big(| \bar{u}- \bar{v}|^2 +8\big) |\nabla \bar u|^2\, dx dt +3\theta^2 \int_{Q_4^+}| \bar{u}- \bar{v}|^2 \, dx dt
+  \lambda^2 \int_{Q_4^+} \bar{u}^2 c^2 \, dx dt\right].
\end{align*}
\end{lemma}
\begin{proof}
%Observe that since $ \bar{u} \geq 0$ in $Q_4^+$, we %have from the maximum principle that 
%$ \bar{v} \geq 0$ in $Q_4^+$. 
Let $w = \bar u - \bar v$. Then  $w\in \hat\calW(Q_4^+)$ is a weak solution of 
\begin{equation*}
\left \{
\begin{array}{lcll}
w_t &=& \nabla\cdot[(1+\bar{v})\bar{\A}_{B_4^+}(t)\nabla w]+ \nabla\cdot\big\{[w\A + (1+\bar v)(\A-\bar{\A}_{B_4^+}(t))]\nabla \bar{u}\big\}\\
 && +\theta^2 w(1-\bar{u}-\bar{v}) -\lambda\theta c \bar{u} \quad &\text{in}\quad Q_4^+, \\
\begin{displaystyle}
\frac{\partial w}{\partial \bnu} 
\end{displaystyle}
 &=& 0 \quad &\text{on}\quad \tilde T_4, \\
w & =& 0  \quad &\text{on}\quad \partial_b Q_4^+ \cup \partial_c Q_4^+.
\end{array}\right.
\end{equation*}
Since $\partial \bar u/\partial \bnu=0$ on $\tilde T_4$, by multiplying the above equation by $w$ and integrating by parts we obtain for each $s\in (-16,16)$
\begin{align*}
&\int_{B_4^+} \frac{w(x,s)^2}{2} \, dx
+ \int_{-16}^s\int_{B_4^+} (1 +\bar{v}) \langle \bar{\A}_{B_4^+}(t) \nabla w, \nabla w\rangle \, dx dt\\
&= -
\int_{-16}^s\int_{B_4^+} w \langle \A \nabla \bar{u} , \nabla w\rangle \, dx dt -
\int_{-16}^s\int_{B_4^+} (1+\bar v)\langle  (\A-\bar{\A}_{B_4^+}(t)) \nabla \bar{u}, \nabla w\rangle \, dx dt\\
&\quad  + \theta^2 \int_{-16}^s\int_{B_4^+} w^2\big(1-\bar{u}-\bar{v}\big)\, dx dt
- \lambda\theta  \int_{-16}^s\int_{B_4^+} c \bar{u} w\, dx dt.
\end{align*}
The result then follows by the same arguments as in the proof of Lemma~\ref{gradient-est-II} .
\end{proof}

The following approximation result is a global version of Lemma~\ref{lm:compare-solution}.
\begin{lemma}\label{lm:G-compare-solution}
Assume that $0<\theta \leq \lambda$. For any $\e>0$, there exists $\delta>0$ depending only on $\e$, $\Lambda$ and  $n$  such that:  if
\begin{equation*}
\int_{Q_4^+}  \Big[ |\A(x, t) - \bar{\A}_{B_4^+}(t)|^2 + |c(x,t)|^2\Big]\, dx dt \leq \delta,
\end{equation*}
and  $u\in \hat\calW(Q_4^+)$ is a weak solution of \eqref{MP-HS} satisfying
\begin{equation*}
0\leq u \leq \frac{1}{\lambda}\, \mbox{ in } Q_4^+ \quad\mbox{and}\quad \int_{Q_4^+}{|\nabla u|^2 \, dxdt}\leq 1,
\end{equation*}
then  
\begin{equation}\label{g-uv-comparison}
\int_{Q_4^+}{|u - v|^2\, dx dt}\leq \e^2,
\end{equation}
where $v\in\hat\calW(Q_4^+)$ is a weak solution of \eqref{REQ-HS} with $0\leq v\leq 1/\lambda$ in $Q_4^+$. Moreover, 
\begin{equation}\label{g-gradient-v-controlled}
\int_{Q_4^+} |\nabla v|^2\, dx dt \leq 2+ 66\Lambda\Big(18 \Lambda^2 + 3 \theta^2 \e^2 + \delta\Big).
\end{equation}
\end{lemma}
\begin{proof}
We first prove \eqref{g-uv-comparison} by contradiction. Suppose that estimate \eqref{g-uv-comparison}  is not true. Then there exist $\e_0,\, \Lambda,\, n$,  sequences of  numbers  $\{\lambda_k\}_{k=1}^\infty$ and $\{\theta_k\}_{k=1}^\infty$
with $0<\theta_k\leq \lambda_k$, a sequence of coefficient matrices $\{\A_k\}_{k=1}^\infty$, 
 and sequences of non-negative functions  $\{c_k\}_{k=1}^\infty$, $\{\psi_k\}_{k=1}^\infty$  and $\{u^k\}_{k=1}^\infty$ 
such that
\begin{equation}\label{G-c_k-condition}
\int_{Q_4^+}  \Big[ |\A_k(x,t) - \bar{\A}_k(t)|^2 + |c_k(x,t)|^2\Big]\, dx dt \leq \frac{1}{k},
\end{equation}
 $u^k\in \hat\calW(Q_4^+)$ is a weak solution of 
\begin{equation}\label{G-eq-u_k}
\left \{
\begin{array}{lcll}
u^k_t  &=&  \nabla\cdot[(1+\lambda_k u^k)\A_k\nabla u^k] + \theta_k^2 u^k(1-\lambda_k u^k) - \lambda_k \theta_k c_k u^k  \quad &\text{in}\quad Q_4^+, \\
\begin{displaystyle}
\frac{\partial u^k}{\partial \bnu} 
\end{displaystyle}
 &=& 0 \quad &\text{on}\quad \tilde T_4, \\
u^k & =& \psi_k  \quad &\text{on}\quad \partial_b Q_4^+ \cup \partial_c Q_4^+
\end{array}\right.
\end{equation}
with $0\leq u^k \leq 1/\lambda_k$ in $Q_4^+$,
\begin{equation}\label{G-gradient-bounded-ass}
\int_{Q_4^+}{|\nabla u^k|^2 \, dxdt}\leq 1,
\end{equation}
and 
\begin{equation}\label{G-contradiction-conclusion}
\int_{Q_4^+} |u^k - v^k|^2 \, dx dt > \e_0^2  \quad\mbox{for all}\quad k.
\end{equation}
Here  $\bar{\A}_k(t) := \fint_{B_4^+}{\A_k(x,t)\, dx}$, $\, 0\leq v^k \leq 1/\lambda_k$ in $Q_4^+$ and $v^k\in \hat\calW(Q_4^+)$ is a weak solution  of
\begin{equation*}
\left \{
\begin{array}{lcll}
v^k_t  &=&  \nabla\cdot[(1+\lambda_k v^k)\bar{\A}_k(t) \nabla v^k] + \theta_k^2 v^k(1-\lambda_k v^k)  \quad &\text{in}\quad Q_4^+, \\
\begin{displaystyle}
\frac{\partial v^k}{\partial \bnu} 
\end{displaystyle}
&=& 0 \quad &\text{on}\quad \tilde T_4, \\
v^k & =& u^k  \quad &\text{on}\quad \partial_b Q_4^+ \cup \partial_c Q_4^+.
\end{array}\right.
\end{equation*}
Since $0\leq u^k, v^k \leq 1/\lambda_k$ in $Q_4^+$,  we infer from \eqref{G-contradiction-conclusion}  that
\[
 \theta_k \leq \lambda_k\leq \frac{|Q_4^+|^{\frac{1}{2}}}{\e_0},
\]
that is, the sequences $\{\lambda_k\}$ and $\{\theta_k\}$  are  bounded. 
Also $\{\bar{\A}_k\}$ is   bounded in $L^\infty(-16, 16; \M^{n\times n})$ due to condition \eqref{boundary-elipticity}
 for $\A_k$.
Then as in the proof of Lemma~\ref{lm:compare-solution} and by taking subsequence if necessary, we can assume that $\lambda_k \to \lambda$, 
$\theta_k \to\theta$ and $\bar{\A}_k \rightharpoonup  \A_0$ weakly in 
$L^2(Q_4^+)$ for some constants
$\lambda, \theta$ satisfying $0\leq \theta\leq \lambda<\infty$ and some $\A_0\in L^\infty(-16,16; \M^{n\times n})$ satisfying \eqref{eigenvalues-controlled}. We are going to derive a contradiction by showing that there are subsequences $\{u^{k_m}\}$ and $\{v^{k_m}\}$ such that $u^{k_m} - v^{k_m} \to 0$ in $L^2(Q_4^+)$.

Let us first consider the case $\lambda>0$. Then the sequence $\{u^k\}$ is bounded in $Q_4^+$. 
This together with \eqref{G-c_k-condition}, \eqref{G-eq-u_k}, \eqref{G-gradient-bounded-ass} and the boundedness of $\{\A_k\}$ and  $\{\theta_k\}$ implies that the sequence $\{u^k\}$ is 
bounded in  $\hat\calW(Q_4^+)$. Next, we apply Lemma~\ref{G-gradient-est-II} for $\bar{u} \rightsquigarrow \lambda_k u^k$ and $\bar{v} \rightsquigarrow \lambda_k v^k$ to obtain
\begin{align*}
& \int_{Q_4^+}|\nabla u^k- \nabla v^k|^2 \, dx dt\\
&\leq  33\Lambda \left[ 18 \Lambda^3 \int_{Q_4^+} |\nabla u^k|^2\, dx dt +3\theta_k^2 \int_{Q_4^+}| u^k- v^k|^2 \, dx dt
+  \lambda_k^2 \int_{Q_4^+} (u^k)^2 c_k^2 \, dx dt\right]\\
&\leq  33 \Lambda\left[18 \Lambda^3 \int_{Q_4^+} |\nabla u^k|^2\, dx dt +3 |Q_4^+| 
+  \int_{Q_4^+}  c_k^2 \, dx dt\right].
\end{align*}
Thanks to  \eqref{G-c_k-condition}, \eqref{G-gradient-bounded-ass} and the triangle inequality, this gives
 \[\int_{Q_4^+} |\nabla v^k |^2 \, dx dt\leq C\quad \mbox{for all}\quad k.
\]
Thus, by reasoning as in the case of $\{u^k\}$,  the sequence $\{v^k\}$ is also bounded in  $\hat\calW(Q_4^+)$.
As in the proof of Lemma~\ref{lm:compare-solution}, we infer from these facts and the compact embedding 
\eqref{cmpAL} that there exist  subsequences, still denoted by $\{u^k\}$ and $\{v^k\}$, and  $u, v\in \hat\calW(Q_4^+)$ such that for $w^k = u^k$ (or $w^k = v^k$) and $w = u$ (or $w = v$) we have
\begin{equation*}
 \left \{ \  
\begin{array}{lll}
&w^k \to  w  \mbox{ strongly in } L^2(Q_4^+),\quad \nabla w^k \rightharpoonup \nabla w\mbox{ weakly  in } L^2(Q_4^+),\\
& \partial_t w^k \rightharpoonup \partial_t w \quad \text{ weakly-* in } L^2(0,T; \hat H^{-1}(B_4^+)),\\
&(1+\lambda_k u^k)(\A_k-\bar \A_k(t)) \nabla u^k  \rightharpoonup 0 \mbox{ weakly  in } L^2(Q_4^+),
\end{array}  \right . 
\end{equation*}
and furthermore,
\begin{align*}
 \lim_{k\to\infty}\int_{Q_4}{(1+\lambda_k w^k) \langle \bar \A_k(t) \nabla w^k, \nabla \varphi\rangle dx dt}=\int_{Q_4}{(1+\lambda w) \langle  \A_0(t) \nabla w, \nabla \varphi\rangle dx dt}
 \end{align*}
for all $\varphi\in C^\infty(\overline{Q_4^+})$ satisfying $\varphi =0$ on $\partial_c Q_4^+$.
Since $u^k = v^k$ on $\partial_b Q_4^+ \cup \partial_c Q_4^+$, we  also have $u =v$
on $\partial_b Q_4^+ \cup \partial_c Q_4^+$. 
%Notice that as 
%\begin{align*}
%&(1+\lambda_k u^k) \A_k - %(1+\lambda u) \A_0\\
%& = (1+\lambda_k u^k) (\A_k %-\bar{\A}_k)+ (\bar{\A}_k - %\A_0) + (u^k -u ) \lambda_k %\bar{\A}_k + u(\lambda_k %\bar{\A}_k -\lambda \A_0),
%\end{align*}
%it follows from \eqref{G-c_k-
%condition} that $(1+\lambda_k %u^k) \A_k \longrightarrow %(1+\lambda u) \A_0$ strongly %in $L^2(Q_4^+)$. 
%Similarly, we get %$(1+\lambda_k v^k) \bar{\A}_k %\longrightarrow (1+\lambda v) %\A_0$ strongly in %$L^2(Q_4^+)$.
Thus by passing to limits and using \eqref{G-c_k-condition} together with the boundedness of $\{\lambda_k\}$ and $\{\theta_k\}$, one  sees that $u$ 
and $v$ are weak solutions of the equation
\begin{equation*}
\left \{
\begin{array}{lcll}
w_t  &=& \nabla\cdot[(1+\lambda w)\A_0(t) \nabla w] + \theta^2 w(1-\lambda w) \quad &\text{in}\quad Q_4^+, \\
\begin{displaystyle}
\frac{\partial w}{\partial \bnu} 
\end{displaystyle}
 &=& 0 \quad &\text{on}\quad \tilde T_4,\\
w^k & =& u=v   &\text{on } \partial_b Q_4^+ \cup \partial_c Q_4^+.
\end{array}\right.
\end{equation*}
In addition, we infer from the strong convergence of $u^k$ and $v^k$ in $L^2(Q_4^+)$ and $\lambda_k \to \lambda$  that $0\leq u,v\leq 1/\lambda$ in $Q_4^+$.
By the uniqueness of solutions given by Lemma~\ref{uniqueness-u}, we conclude that $\lambda u \equiv \lambda v$ in $Q_4^+$. Therefore, $u^k-v^k \longrightarrow u-v=0$ strongly in $L^2(Q_4^+)$
giving a contradiction to \eqref{G-contradiction-conclusion}.

It remains to consider the case $\lambda =0$, that is, $\lambda_k \to 0$. Due to $\theta_k\leq \lambda_k$, we also have $\theta_k \to 0$.
Let $w^k = u^k - v^k$. Then    $w^k$ is a weak solution of 
\begin{equation}\label{G-u^k-v^k:equation}
\left \{
\begin{array}{lcll}
w^k_t  &=&  \nabla\cdot[(1+\lambda_k v^k)\bar{\A}_k(t)\nabla w^k] + \nabla\cdot\Big\{\big[\lambda_k w^k\A_k + (1+\lambda_k v^k)(\A_k - \bar{\A}_k(t))\big] \nabla u^k\Big\}\\
& & + \theta_k^2 w^k(1-\lambda_k u^k -\lambda_k v^k) 
- \lambda_k \theta_k c_k u^k &\text{in }\, Q_4^+, \\
\begin{displaystyle}
\frac{\partial w^k}{\partial \bnu} 
\end{displaystyle}
&=& 0  &\text{on } \tilde T_4, \\
w^k & =& 0   &\text{on } \partial_b Q_4^+ \cup \partial_c Q_4^+.
\end{array}\right.
\end{equation}
On the other hand, by applying Lemma~\ref{G-gradient-est-II} for $\bar{u} \rightsquigarrow \lambda_k u^k$, $\bar{v} \rightsquigarrow \lambda_k v^k$ and using 
the fact $\theta_k$ is small for large $k$, we get
for all sufficiently large $k$ that
\begin{equation}\label{G-w^k-gradient}
 \int_{Q_4^+}|\nabla w^k|^2 \, dx dt
\leq  33\Lambda\left[18\Lambda^3 \int_{Q_4^+} |\nabla u^k|^2\, dx dt 
+ \int_{Q_4^+}  c_k^2 \, dx dt\right]\leq  33\Lambda (18\Lambda^3 +1) 
\end{equation}
and
\begin{equation*}
\int_{Q_4^+} |w^k|^2 \, dx dt
\leq 66\left[ 18\Lambda^3 \int_{Q_4^+} |\nabla u^k|^2\, dx dt 
+   \int_{Q_4^+} c_k^2 \, dx dt\right]\leq 66(18\Lambda^3 +1) .
\end{equation*}
These estimates together with \eqref{G-u^k-v^k:equation}, \eqref{G-gradient-bounded-ass} and the boundedness of $\{\A_k\}$ imply that $\{w^k\}$ is bounded in $\hat\calW(Q_4^+)$. Hence
there exist  a subsequence $\{w^k\}$  and  a function $w\in\hat\calW(Q_4^+)$ such that
\begin{equation}\label{G-w^k-compactness}
 \left \{ \  
\begin{array}{lll}
&w^k \to  w  \mbox{ strongly in } L^2(Q_4^+),\quad \nabla w^k \rightharpoonup \nabla w\mbox{ weakly  in } L^2(Q_4^+),\\
& \partial_t w^k \rightharpoonup \partial_t w \quad \text{ weakly-* in } L^2(0,T; \hat H^{-1}(B_4^+)).
\end{array}  \right . 
\end{equation}
Now let $\bar{v}^k = \lambda_k v^k$. Then $0\leq \bar{v}^k\leq 1$ in $Q_4^+$, and $\bar{v}^k$ is a weak solution of
\[
 (\bar{v}^k)_t  =  \nabla\cdot[(1+\bar{v}^k)
 \bar{\A}_k(t) \nabla \bar{v}^k] + \theta_k^2 \bar{v}^k(1-\bar{v}^k) \quad \text{in}\quad Q_4^+.
\]
Moreover, $\nabla \bar{v}^k =\lambda_k \nabla v^k \to 0$ strongly in $L^2(Q_4^+)$ since it follows from \eqref{G-gradient-bounded-ass} and \eqref{G-w^k-gradient} that
\[\int_{Q_4^+} |\nabla v^k |^2 \, dx dt\leq C\quad \mbox{for all large}\quad k.
\]
Thus  $\{\bar{v}^k\}$ is bounded in $\hat\calW(Q_4^+)$ and so, up to a subsequence,  $\bar{v}^k \longrightarrow \bar v$ strongly in $L^2(Q_4^+)$ for some function   $\bar{v} \in \hat\calW(Q_4^+)$ with $0\leq \bar{v}\leq 1$ in $Q_4^+$. 
By arguing as in the proof of Lemma~\ref{lm:compare-solution}, we infer in addition that $\bar v$ is a constant function in $Q_4^+$, and 
\begin{align*}
 \lim_{k\to\infty}\int_{Q_4}{(1+\bar v^k) \langle \bar \A_k(t) \nabla w^k, \nabla \varphi\rangle dx dt}
 =\int_{Q_4}{(1+\bar v) \langle  \A_0(t) \nabla w, \nabla \varphi\rangle dx dt}
 \end{align*}
for all $\varphi\in C^\infty(\overline{Q_4^+})$ satisfying $\varphi =0$ on $\partial_c Q_4^+$.
Using these convergences, \eqref{G-c_k-condition}  and 
\eqref{G-w^k-compactness}, we can pass 
to the limits
in  \eqref{G-u^k-v^k:equation} to conclude that  $w$ is a weak solution of the equation
\begin{equation}\label{G-limit-equation-case2}
\left \{
\begin{array}{lcll}
w_t  &=&  \nabla\cdot[(1+\bar v)\A_0(t)\nabla w] &\text{ in}\quad Q_4^+, \\
\begin{displaystyle}
\frac{\partial w}{\partial \bnu} 
\end{displaystyle}
&=& 0 \quad &\text{on}\quad \tilde T_4, \\
w & =& 0  \quad &\text{on}\quad \partial_b Q_4^+ \cup \partial_c Q_4^+.
\end{array}\right.
\end{equation}
Thanks to the uniqueness (see Lemma~\ref{CP}) of the trivial solution to the linear equation \eqref{G-limit-equation-case2}, we conclude that $w \equiv 0$ in $Q_4^+$. This gives
\[
\lim_{k\to\infty}\int_{Q_4^+} |u^k- v^k|^2\, dx dt =\lim_{k\to\infty}\int_{Q_4^+} |w^k|^2\, dx dt
=\int_{Q_4^+} |w|^2\, dx dt=0,
\]
 which contradicts to \eqref{G-contradiction-conclusion}. Thus the proof of \eqref{g-uv-comparison} is complete and it remains to prove  \eqref{g-gradient-v-controlled}. For  this we apply Lemma~\ref{G-gradient-est-II} for $\bar{u} \rightsquigarrow \lambda u$ and $\bar{v} \rightsquigarrow \lambda v$ to obtain
\begin{align*}
 \int_{Q_4^+}|\nabla u- \nabla v|^2 \, dx dt\leq  33\Lambda\left[ 18\Lambda^3 \int_{Q_4^+}  |\nabla  u|^2\, dx dt +3\theta^2 \int_{Q_4^+}| u-v|^2 \, dx dt
+ \int_{Q_4^+} c^2 \, dx dt\right].
\end{align*}
This together with \eqref{g-uv-comparison}  and the assumptions gives 
\begin{equation}\label{g-difference-control}
\int_{Q_4^+}|\nabla u- \nabla v|^2 \, dx dt\leq  33\Lambda\left[ 18\Lambda^3  +3\theta^2 \e^2 
+\delta \right].
\end{equation}
Since $\|\nabla v\|_{L^2(Q_4^+)}\leq 1+ \|\nabla u- \nabla v\|_{L^2(Q_4^+)}$, 
the estimate \eqref{g-gradient-v-controlled} follows immediately from \eqref{g-difference-control}.
\end{proof}

%===================
In the next lemma, we establish an approximation of gradients of solutions near the flat boundary. This will play a key role in our derivation of boundary $W^{1,p}$-estimates in Subsection~\ref{flat-domain}.
\begin{lemma}\label{lm:G-compare--gradient-solution}
Assume that $0<\theta \leq \lambda$, $\theta\leq 1$ and $0<r\leq 1$. For any $\e>0$, there exists $\delta>0$ small depending only on $\e$, $\Lambda$ and  $n$  such that:  if
\begin{equation}\label{G-smallness-condition-c}
\fint_{Q_{4 r}^+}  \Big[ |\A(x,t) - \bar{\A}_{B_{4 r}^+}(t)|^2 +  |c(x,t)|^2\Big]\, dx dt \leq \delta,
\end{equation}
then for any weak solution $u\in \hat\calW(Q_{4 r}^+)$  of \eqref{MP-HS} satisfying
\begin{equation*}
0\leq u \leq \frac{1}{\lambda}\, \mbox{ in } Q_{4 r}^+ \quad\mbox{and}\quad \fint_{Q_{4 r}^+}{|\nabla u|^2 \, dxdt}\leq 1,
\end{equation*}
and a weak solution $v\in \hat \calW(Q_{4 r}^+)$  of  
\begin{equation}\label{G-limiting-eqn}
\left \{
\begin{array}{lcll}
v_t  &=&  \nabla\cdot[(1+\lambda v)\bar{\A}_{B_{4 r}^+}(t) \nabla v] + \theta^2 v(1-\lambda v) \quad &\text{in}\quad Q_{4 r}^+, \\
\begin{displaystyle}
\frac{\partial v}{\partial \bnu} 
\end{displaystyle}
 &=& 0 \quad &\text{on}\quad \tilde T_{4r}, \\
v & =& u  \quad &\text{on}\quad \partial_b Q_{4 r}^+ \cup \partial_c Q_{4 r}^+
\end{array}\right.
\end{equation}
satisfying $0\leq v\leq 1/\lambda$ in $Q_{4 r}^+$, we have  
\begin{equation}\label{G-localied-compare-uv}
\fint_{Q_{4 r}^+}{|u - v|^2\, dx dt}\leq \e^2 r^2,
\end{equation}
 \begin{equation}\label{g-rescaled-gradient-v-controlled}
\fint_{Q_{4r}^+} |\nabla v|^2\, dx dt \leq 4^{n+2}  \omega_n \left[2+ 66\Lambda\Big(18 \Lambda^2 + 3 \theta^2 \e^2 + \delta\Big)\right],
\end{equation}
and
\begin{equation}\label{boundary-gradient-comp}
\fint_{Q_{2 r}^+}{|\nabla u - \nabla v|^2\, dx dt}\leq \e^2.
\end{equation}
\end{lemma}
\begin{proof}
Define
\[
u'(x,t) =\frac{u(r x, r^2 t)}{r}, \quad v'(x,t)= \frac{v(r x, r^2 t)}{r}, \quad \A'(x, t) =\A(rx, r^2 t) \quad\mbox{and}\quad c'(x,t) = c(r x, r^2 t).
\]
Let  $\lambda' = \lambda r$ and $\theta' = \theta r$. Then $u'$ is a weak solution of
\begin{equation*}
\left \{
\begin{array}{lcll}
u'_t  &=&  \nabla\cdot[(1+\lambda' u')\A' \nabla u'] + \theta'^2 u'(1-\lambda' u') - \lambda' \theta' c' u' \quad &\text{in}\quad Q_4^+, \\
\begin{displaystyle}
\frac{\partial u'}{\partial \bnu} 
\end{displaystyle}
 &=& 0 \quad &\text{on}\quad \tilde T_4
\end{array}\right.
\end{equation*}
and $v'$ is  a weak solution of  
\begin{equation*}
\left \{
\begin{array}{lcll}
v'_t  &=&  \nabla\cdot[(1+\lambda' v')\bar{\A'}_{B_4^+}(t)\nabla v'] + \theta'^2 v'(1-\lambda' v') \quad &\text{in}\quad Q_4^+, \\
\begin{displaystyle}
\frac{\partial v'}{\partial \bnu} 
\end{displaystyle}
 &=& 0 \quad &\text{on}\quad \tilde T_4, \\
v' & =& u'  \quad &\text{on}\quad \partial_b Q_4^+ \cup \partial_c Q_4^+.
\end{array}\right.
\end{equation*}
Hence  by applying Lemma~\ref{lm:G-compare-solution} 
for the solutions $u', v'$ and rescaling back, we obtain the estimates \eqref{G-localied-compare-uv} and \eqref{g-rescaled-gradient-v-controlled}.

In order to prove \eqref{boundary-gradient-comp}, we define  $w= u-v$. Then  $w\in \hat\calW(Q_{4r}^+)$ is a weak solution of 
\begin{equation}\label{G-u-v:equation}
\left \{
\begin{array}{lcll}
w_t 
&=& \nabla\cdot[(1+\lambda u )\A \nabla w] + \nabla\cdot \Big\{[\lambda w \bar{\A}_{B_{4 r}^+}(t) + (1+\lambda u)(\A- \bar{\A}_{B_{4 r}^+}(t))]\nabla v \Big\}\\
&& + \theta^2 w(1-\lambda u -\lambda v) -\lambda \theta c u  &\text{in}\quad Q_{4 r}^+, \\
\begin{displaystyle}
\frac{\partial w}{\partial \bnu} 
\end{displaystyle}
&=& 0  &\text{on}\quad \tilde T_{4 r}, \\
w & =& 0  &\text{on}\quad \partial_b Q_{4 r}^+ \cup \partial_c Q_{4 r}^+.
\end{array}\right.
\end{equation}
Let $\varphi$ be the standard cut-off function which is $1$ 
on $Q_{2 r}$, $\text{supp}(\varphi) \subset \overline{Q_{3 r}}$, $|\nabla \varphi| \leq C_n/r$ and $|\varphi_t|\leq C_n/ r^2$.  Let us 
multiply the equation \eqref{G-u-v:equation}  by $\varphi^2 w$ 
and use  integration by parts together with the fact $\partial v/\partial \bnu = 0$ on $\tilde T_{4 r}$. Then by arguing as in the proof of Lemma~\ref{lm:localized-compare-solution}, we obtain
\begin{align}\label{G-Cauchy--Schwarz-step}
 &\frac{\Lambda^{-1}}{4}\int_{Q_{4 r}^+} (1 + \lambda u) |\nabla w|^2\varphi^2 dxdt \leq  (4 \Lambda^2 +1) \int_{Q_{4 r}^+}|
\nabla \varphi|^2 w^2 \, dxdt\\
&   + \Lambda \|\nabla (\lambda v)\|_{L^\infty(Q_{3 r}^+)}\left( \int_{Q_{4 r}^+}  |\nabla w|\varphi^2 |w| \, dxdt
+ 2 \int_{Q_{4 r}^+}  |
\nabla \varphi| \varphi w^2 \, dxdt\right)\nonumber\\
& +2(\Lambda +2) \int_{Q_{4 r}^+}|\A-\bar{\A}_{B_{4 r}^+}(t)|^2  |\nabla v|^2 
 \varphi^2 \, dxdt
\nonumber\\
&+ \int_{Q_{4 r}^+} \varphi\varphi_t w^2\, dxdt  +\int_{Q_{4 r}^+}    w^2 \, dxdt + \int_{Q_{4 r}^+}  |c|  |w| \, dxdt.\nonumber
\end{align} 
We next estimate $\|\nabla (\lambda v)\|_{L^\infty(Q_{3 r}^+)}$ and $\|\nabla v\|_{L^\infty(Q_{3 r}^+)}$.  Let us define 
$\bar{v}(x,t) := \lambda v(rx, r^2 t)$ for $(x,t)\in Q_4^+$.  Then $0\leq \bar{v}\leq 1$ in $Q_4^+$ and  $\bar{v}$ is a weak solution of
\begin{equation*}
\left \{
\begin{array}{lcll}
\bar{v}_t  &=&  \nabla\cdot[(1+\bar{v})\bar{\A'}_{B_{4}^+}(t)\nabla \bar{v}] + (\theta r)^2 \bar{v}(1-\bar{v})\quad &\text{in}\quad Q_4^+, \\
\nabla \bar v\cdot \bnu &=& 0 \quad &\text{on}\quad \tilde T_4.
\end{array}\right.
\end{equation*}
Thanks to $\theta r\leq \theta\leq 1$, we then can use Lemma~\ref{G-W^{1,infty}-est}
to get
 \begin{equation}\label{g-handle-gradients-v}
 \|\nabla \bar v\|_{L^\infty(Q_3^+)}\leq C(\Lambda, n)\left(\fint_{Q_{\frac{7}{2}}^+}{|\nabla \bar{v}|^2 \, dxdt}\right)^{\frac{1}{2}}.
  \end{equation}
This together with Lemma~\ref{G-W^{1,2}-est}  yields  $\|\nabla \bar v\|_{L^\infty(Q_3^+)}\leq C(\Lambda, n)$. By rescaling back, we obtain
\begin{equation}\label{g-localized-L^infinity-est-I}
\|\nabla (\lambda v)\|_{L^\infty(Q_{3 r}^+)}\leq \frac{C(\Lambda, n)}{r}.
\end{equation}
On the other hand,  \eqref{g-handle-gradients-v} also   gives
\begin{equation}\label{g-localized-L^infinity-est-II}
 \|\nabla v\|_{L^\infty(Q_{3 r}^+)}\leq  C(\Lambda, n)  \left(\fint_{Q_4^+}{|\nabla v(rx, r^2 t)|^2 \, dxdt}\right)^{\frac{1}{2}}
 =C(\Lambda, n)  \left(\fint_{Q_{4 r}^+}{|\nabla v(y,s)|^2 \, dyds}\right)^{\frac{1}{2}}.
\end{equation}
It follows from \eqref{G-Cauchy--Schwarz-step}, \eqref{g-localized-L^infinity-est-I}, \eqref{g-localized-L^infinity-est-II} and  the Cauchy--Schwarz
inequality  that
\begin{align*}\label{G-u-v:gradient-est}
\fint_{Q_{4 r}^+}  |\nabla w|^2\varphi^2 dxdt
&\leq \frac{C}{r^2} \fint_{Q_{4 r}^+}  w^2 \, dxdt + C \left(\fint_{Q_{4 r}^+}{|\nabla v|^2 \, dxdt}\right) \left( \fint_{Q_{4 r}^+}|\A-\bar{\A}_{B_{4 r}^+}(t)|^2 \, dxdt \right)\\
& \quad  +\Big(\fint_{Q_{4 r}^+} c^2  \, dxdt\Big)^{\frac{1}{2}}\Big(\fint_{Q_{4 r}^+}  w^2  \, dxdt\Big)^{\frac{1}{2}}.\nonumber
\end{align*}
This together with 
\eqref{G-localied-compare-uv}, \eqref{g-rescaled-gradient-v-controlled} and the assumption \eqref{G-smallness-condition-c} gives the estimate
\eqref{boundary-gradient-comp}.
\end{proof}

\begin{remark}\label{G-translation-invariant}
Since our equations are invariant under the translation $(x,t)\mapsto (x+ y, t+s)$, Lemma~\ref{lm:G-compare--gradient-solution} still holds true if $Q_r^+$ is replaced by $Q_r^+(y,s)$.
\end{remark}

\subsubsection{Boundary density and gradient estimates on flat domains}\label{flat-domain}
We begin this subsection with a density estimate which is the boundary version of Lemma~\ref{initial-density-est}.
\begin{lemma}\label{G-initial-density-est}
Assume that $0<\theta\leq \lambda$, $\theta\leq 1$, $\A$ satisfies \eqref{boundary-elipticity} and $c\in L^2(Q_4^+)$. There exists a constant $N>1$ depending only on $\Lambda$ and  $n$ such that   
for  any $\e>0$, we can find  $\delta=\delta(\e, \Lambda, n)>0$  satisfying: if
\begin{equation}\label{boundary-SMO}
\sup_{0<\rho\leq 2}\sup_{(y,s)\in \overline{Q_1^+}} \frac{1}{|Q_\rho(y, s)|}\int_{
Q_\rho(y,s)\cap Q_3^+}{|\A(x,t) - \bar{\A}_{B_\rho(y)\cap B_3^+}(t)|^2\, dx dt}\leq \delta,
\end{equation} 
then for any weak solution $u\in \hat\calW(Q_4^+)$ of \eqref{MP-HS} with
$0\leq u \leq \frac{1}{\lambda}$ in $Q_3^+$
and for any $(y,s)\in \overline{Q_1^+}$, $0<r\leq 1/12$ with
\begin{equation*}
 Q_r(y,s)\cap Q_1^+\cap \big\{Q_3^+:\, \M_{Q_3^+}(|\nabla u|^2)\leq 1 \big\}\cap \{ Q_3^+: \M_{Q_3^+}(c^2)\leq \delta\}\neq \emptyset,
\end{equation*}
we have
\begin{equation}\label{g-density-est}
 \big| \{ Q_1^+:\, \M_{Q_3^+}(|\nabla u|^2)> N \}\cap Q_r(y,s)\big|
\leq \e  |Q_r(y,s)|.
\end{equation}
\end{lemma}
\begin{proof}
Let $Q_{\tilde r}(\tilde y, \tilde s)$ be a parabolic cube satisfying  $(\tilde y,\tilde s)\in \tilde T_1$, $0<\tilde r\leq 1/2$  and
\begin{equation}\label{g-one-point-condition}
Q_{\tilde r}(\tilde y, \tilde s)\cap Q_1^+\cap \big\{Q_3^+:\, \M_{Q_3^+}(|\nabla u|^2)\leq 1 \big\}\cap \{ Q_3^+: \M_{Q_3^+}(c^2)\leq \delta\}\neq \emptyset.
\end{equation}
We then claim that there exists $N>0$ depending only on $\Lambda$ and $n$ such that
\begin{equation}\label{g-claim-density-est}
 \big| \{ Q_1^+:\, \M_{Q_3^+}(|\nabla u|^2)> N \}\cap Q_{\tilde r}(\tilde y, \tilde s)\big|
\leq \frac{\e}{6^{n+2}}  |Q_{\tilde r}(\tilde y, \tilde s)|.
\end{equation}
Indeed, it follows from \eqref{g-one-point-condition} that
\begin{align}
  \M_{Q_3^+}(|\nabla u|^2)(x_0,t_0)\leq 1\quad \mbox{and}\quad \M_{Q_3^+}(c^2)(x_0,t_0) \leq
\delta\label{g-maximal-fns-control}
\end{align}
for some point $(x_0,t_0)\in Q_{\tilde r}(\tilde y, \tilde s)\cap  Q_1^+$.
Since $Q_{4 \tilde r}^+(\tilde y, \tilde s) \subset Q_{5 \tilde r}(x_0,t_0)\cap Q_3^+$, we deduce from \eqref{g-maximal-fns-control} that
\begin{align*}
&\fint_{Q_{4 \tilde r}^+(\tilde y, \tilde s)}|\nabla u|^2 \, dx dt \leq \frac{|Q_{5\tilde  r}(x_0,t_0)|}{|Q_{4 \tilde r}^+(\tilde y, \tilde s)|}\frac{1}{|Q_{5\tilde  r}(x_0,t_0)|}\int_{Q_{5 \tilde r}(x_0,t_0)\cap Q_3^+}|\nabla u|^2 \, dx dt\leq 2\Big(\frac{5}{4}\Big)^{n+2},\\
 &\fint_{Q_{4 \tilde r}^+(\tilde y, \tilde s)}c^2 \, dx dt \leq \frac{|Q_{5 \tilde r}(x_0,t_0)|}{|Q_{4 \tilde r}^+(\tilde y, \tilde s)|}\frac{1}{|Q_{5\tilde  r}(x_0,t_0)|} \int_{Q_{5 \tilde r}(x_0,t_0)\cap Q_3^+}c^2 \, dx dt\leq 2 \Big(\frac{5}{4}\Big)^{n+2} \delta.
\end{align*}
Also, as $B_{4 \tilde r}(\tilde y)\cap B_3^+=B_{4 \tilde r}^+(\tilde y) $ the assumption \eqref{boundary-SMO} gives
\begin{equation*}
\fint_{Q_{4 \tilde r}^+(\tilde y, \tilde s)}{|\A(x,t) - \bar{\A}_{B_{4 \tilde r}^+(\tilde y)}(t)|^2\, dx dt}\leq 2 \delta.
\end{equation*}
Therefore, we can use Lemma~\ref{lm:G-compare--gradient-solution} and Remark~\ref{G-translation-invariant} to obtain
\begin{equation}\label{g-gradient-comparison}
\fint_{Q_{2 \tilde r}^+(\tilde y, \tilde s)}{|\nabla u- \nabla v|^2\, dx dt}\leq \eta^2,
\end{equation}
where $v \in \hat\calW(Q_{4 \tilde r}^+(\tilde y, \tilde s))$ is the  weak solution of 
\begin{equation*}
\left \{
\begin{array}{lcll}
v_t  &=&  \nabla\cdot[(1+\lambda v)\bar{\A}_{B_{4 \tilde r}^+(\tilde y)}(t) \nabla v] + \theta^2 v(1-\lambda v) \quad &\text{in}\quad Q_{4 \tilde r}^+(\tilde y, \tilde s), \\
\begin{displaystyle}
\frac{\partial v}{\partial \bnu} 
\end{displaystyle}
 &=& 0 \quad &\text{on}\quad \tilde T_{4\tilde r}(\tilde y), \\
v & =& u  \quad &\text{on}\quad \partial_b Q_{4 \tilde r}^+(\tilde y, \tilde s) \cup \partial_c Q_{4 \tilde r}^+(\tilde y, \tilde s)
\end{array}\right.
\end{equation*}
satisfying $0\leq v\leq 1/\lambda$ in $Q_{4 \tilde r}^+(\tilde y, \tilde s)$, and $\delta =\delta(\eta,\Lambda, n)$ with $\eta$ being determined later. We note that the existence and uniqueness 
of such weak solution $v$ is guaranteed by Theorem~\ref{existence-u}.

Let $\bar{v}(x,t) = \lambda v(\tilde r x +\tilde y, \tilde r^2 t +\tilde s)$  and $\A'(x,t) = \A(\tilde{r} x +\tilde{y}, \tilde{r}^2 t +\tilde{s})$ for $(x,t)\in Q_4^+$.  Then $0\leq \bar{v}\leq 1$ in $Q_4^+$ and  $\bar{v}$ is a weak solution of
\begin{equation*}
\left \{
\begin{array}{lcll}
\bar{v}_t  &=&  \nabla\cdot[(1+\bar{v})
\bar{\A'}_{B_{4}^+}(t)\nabla \bar{v}] + (\theta \tilde r)^2 \bar{v}(1-\bar{v})\quad &\text{in}\quad Q_4^+, \\
\begin{displaystyle}
\frac{\partial \bar v}{\partial \bnu} 
\end{displaystyle}
 &=& 0 \quad &\text{on}\quad \tilde T_4.
\end{array}\right.
\end{equation*}
Thanks to $\theta \tilde r\leq \theta\leq 1$, we can apply Lemma~\ref{G-W^{1,infty}-est} to get 
\begin{align*}
\|\nabla \bar v\|_{L^\infty(Q_{\frac{3}{2}}^+)}^2\leq C \fint_{Q_2^+}{|\nabla \bar{v}|^2 \, dxdt},
\end{align*}
which together with \eqref{g-gradient-comparison} and \eqref{g-maximal-fns-control}   gives
\begin{align}\label{g-gradient-is-bounded}
\|\nabla v\|_{L^\infty(Q_{\frac{3 \tilde r}{2}}^+(\tilde y,\tilde s))}^2
&\leq C \fint_{Q_{2\tilde r}^+(\tilde y,\tilde s)}{|\nabla v|^2 \, dxdt}\\
&\leq  2 C \left( \fint_{Q_{2\tilde r}^+(\tilde y, \tilde s)}{|\nabla u -\nabla v|^2 \, dxdt}
+ \fint_{Q_{2\tilde r}^+(\tilde y,\tilde s)}{|\nabla u|^2 \, dxdt}\right)\leq C(\Lambda, n) (\eta^2 +1).\nonumber
\end{align}
 We assert  that \eqref{g-maximal-fns-control}, \eqref{g-gradient-comparison} and \eqref{g-gradient-is-bounded}
yield
\begin{equation}\label{g-set-relation-claim}
 \big\{ Q_{\tilde r}(\tilde y, \tilde s):  \M_{Q_{2\tilde r}^+(\tilde y,\tilde s)}(|\nabla u - \nabla v|^2) \leq C(\Lambda, n) \big\}\subset \big\{ Q_{\tilde r}(\tilde y, \tilde s):\, \M_{Q_3^+}(|\nabla u |^2) \leq N\big\}
\end{equation}
with $N = \max{\{6 C(\Lambda, n), 5^{n+2}\}}$. To see this, let $(x,t)$ be a point in the set on the left hand side of \eqref{g-set-relation-claim}, and consider 
$Q_\rho(x,t)$. If $\rho\leq \tilde r/2$, then as $\tilde y_n =0$ we have  $Q_\rho(x,t) \cap Q_3^+\subset Q_{3 \tilde r/2}^+(\tilde y,\tilde s)$
and hence
\begin{align*}
 &\frac{1}{|Q_\rho(x,t)|}\int_{Q_\rho(x,t) \cap Q_3^+} |\nabla u|^2 \, dx dt\\
 &\leq 
 \frac{2}{|Q_\rho(x,t)|}\Big[\int_{Q_\rho(x,t) \cap Q_3^+} |\nabla u-\nabla v|^2 \, dx dt
 +\int_{Q_\rho(x,t) \cap Q_3^+} |\nabla v|^2 \, dx dt \Big]\\
 &\leq 2 \M_{Q_{2\tilde r}^+(\tilde y,\tilde s)}(|\nabla u - \nabla v|^2)(x,t) +2 \|\nabla v \|_{L^\infty(Q_{\frac{3\tilde r}{2}}^+(\tilde y,\tilde s))}^2
 \leq 2C(\Lambda, n) \big( \eta^2 + 2\big)\leq 6 C(\Lambda, n).
\end{align*}
On the other hand if $\rho> \tilde r/2$, then  $Q_\rho(x,t)\subset Q_{5\rho}(x_0, t_0)$. This and  the first inequality in \eqref{g-maximal-fns-control} imply that
\begin{align*}
 \frac{1}{|Q_\rho(x,t)|}\int_{Q_\rho(x,t) \cap Q_3^+} |\nabla u|^2 \, dx dt
 \leq \frac{5^{n+2}}{|Q_{5\rho}(x_0, t_0)|} \int_{Q_{5\rho}(x_0, t_0) \cap Q_3^+} |\nabla u|^2 \, dx dt\leq 
 5^{n+2}.
\end{align*}
Therefore, we conclude that $\M_{Q_3^+}(|\nabla u |^2)(x,t) \leq N$ and  \eqref{g-set-relation-claim}  is proved.
Note that   \eqref{g-set-relation-claim} is equivalent to
 \begin{align*}
 \big\{Q_{\tilde r}(\tilde y,\tilde s):\, \M_{Q_3^+}(|\nabla u |^2) > N\big\} \subset
 \big\{ Q_{\tilde r}(\tilde y,\tilde s):\, \M_{Q_{2\tilde r}^+(\tilde y,\tilde s)}(|\nabla u - \nabla v|^2) > C(\Lambda, n) \big\}.
 \end{align*}
It follows from this,  the weak type $1-1$ estimate and \eqref{g-gradient-comparison}   that
\begin{align*}
 &\big|\big\{Q_{\tilde r}(\tilde y, \tilde s):\, \M_{Q_3^+}(|\nabla u |^2) > N\big\} \big|\leq  \big|
 \big\{ Q_{\tilde r}(\tilde y, \tilde s):\, \M_{Q_{2 \tilde r}^+(\tilde y,\tilde s)}(|\nabla u - \nabla v|^2) > C(\Lambda, n)  \big\}\big|\\
 &\leq C  \int_{Q_{2\tilde r}^+(\tilde y,\tilde s)}{|\nabla u - \nabla v|^2 \, dx dt}\leq C' \eta^2 \, |Q_{\tilde r}(\tilde y, \tilde s)|,
 \end{align*}
 where $C'>0$ depends only on $\Lambda$ and $n$.
By choosing $\eta := \sqrt{\frac{\e}{6^{n+2} C'}}$, we obtain the claim \eqref{g-claim-density-est}.

To proceed with the proof, we consider the following two cases:

{\bf Case 1:}  $\dist(y, T_1)> 5r$. Then $B_{4 r}(y) \Subset B_3^+$, $Q_{4r}(y,s)\Subset Q_3^+$ and $\fint_{Q_{4 r}(y, s)}{ |\A(x, t) - \bar{\A}_{B_{4 r}(y)}(t)|^2 dx dt}\leq \delta$ by \eqref{boundary-SMO}. Hence  \eqref{g-density-est} follows from the interior estimate in Lemma~\ref{initial-density-est} (see the proof of Lemma~\ref{initial-density-est}).

{\bf Case 2:} $\dist(y, T_1)\leq 5r$. Then there exists $\tilde y \in T_1$ such that $B_r(y)\subset B_{6r}(\tilde y)$. Consequently, $Q_r(y,s)\subset Q_{6 r}(\tilde y,s)$ and due to the assumption \eqref{g-one-point-condition} we have 
\begin{equation*}
Q_{6 r}(\tilde y,s)\cap Q_1^+\cap \big\{Q_3^+:\, \M_{Q_3^+}(|\nabla u|^2)\leq 1 \big\}\cap \{ Q_3^+: \M_{Q_3^+}(c^2)\leq \delta\}\neq \emptyset.
\end{equation*}
Therefore, it follows from the claim \eqref{g-claim-density-est} that 
\begin{equation*}
 \big| \{ Q_1^+:\, \M_{Q_3^+}(|\nabla u|^2)> N \}\cap Q_{6 r}(\tilde y,s)\big|
\leq \frac{\e}{6^{n+2}}  |Q_{6 r}(\tilde y,s)|=\e |Q_{r}(y,s)|
\end{equation*}
yielding \eqref{g-density-est}.
\end{proof}

 In view of Lemma~\ref{G-initial-density-est}, we can apply the Vitali covering lemma (see \cite[Theorem~2.6]{B1}) 
for $E=\{ Q_1^+:\, \M_{Q_3^+}(|\nabla u|^2)> N \} $ and 
$F=\{ Q_1^+:\, \M_{Q_3^+}(|\nabla u|^2)> 1 \}\cup \{ Q_1^+:\, \M_{Q_3^+}(c^2)> \delta \}$
 to obtain:
\begin{lemma}\label{G-second-density-est}
Assume that $0<\theta\leq \lambda$, $\theta\leq 1$, $\A$ satisfies \eqref{boundary-elipticity} and $c\in L^2(Q_4^+)$. There exists a constant $N>1$ depending only on $\Lambda$ and $n$ such that   
for  any $\e>0$, we can find  $\delta=\delta(\e, \Lambda, n)>0$  satisfying:  if
\begin{equation*}
\sup_{0<\rho\leq 2}\sup_{(y,s)\in \overline{Q_1^+}} \frac{1}{|Q_\rho(y, s)|}\int_{
Q_\rho(y, s)\cap Q_3^+}{|\A(x) - \bar{\A}_{B_\rho(y)\cap B_3^+}(t)|^2\, dx dt}\leq \delta,
\end{equation*} 
 then for any weak solution $u\in \hat\calW(Q_4^+)$ of \eqref{MP-HS} satisfying
\begin{equation*}
0\leq u \leq \frac{1}{\lambda}\quad\mbox{in}\quad Q_3^+ \quad \mbox{and}\quad
\big| \{Q_1^+: 
\M_{Q_3^+}(|\nabla u|^2)> N \}\big| \leq \e |Q_1^+|,
\end{equation*}
we have
\begin{align*}
\big|\{Q_1^+: \M_{Q_3^+}(|\nabla u|^2)> N\}\big|
\leq 2 (10)^{n+2}\e \, \Big\{
\big|\{Q_1^+: \M_{Q_3^+}(|\nabla u|^2)> 1\}\big|
+ \big|\{ Q_1^+: \M_{Q_3^+}(c^2)> \delta \}\big|\Big\}.\nonumber
\end{align*}
\end{lemma}

We are ready to state and prove the boundary $W^{1,p}$-estimates for flat domains.
\begin{theorem}\label{Global-half-space} Assume that $ \lambda>0$, $0<\theta\leq 1$, $\A$ satisfies \eqref{boundary-elipticity} and $c\in L^2(Q_4^+)$. 
Let   $u\in \hat\calW(Q_4^+)$  be a weak solution of  \eqref{MP-flat-domain} satisfying
$0 \leq u \leq 1/\lambda$ in $Q_3^+$.
Then for any $p>2$, there exists a constant $\delta =\delta(p,\Lambda,n)>0$ such that if
\begin{equation*}
\sup_{0<\rho\leq 2}\sup_{(y,s)\in \overline{Q_1^+}} \frac{1}{|Q_\rho(y, s)|}\int_{Q_\rho(y, s)\cap Q_3^+}{|\A(x, t) - \bar{\A}_{B_\rho(y)\cap B_3^+}(t)|^2\, dx dt}\leq \delta,
\end{equation*}  
we have 
\begin{equation}\label{g-main-estimate}
\int_{Q_1^+}|\nabla u|^p \, dx dt\leq C\left\{  \Big(\frac{\theta}{\lambda} \vee  \|u\|_{L^2(Q_4^+)}\Big)^p + \int_{Q_4^+}|c|^p \, dx dt\right\}.
\end{equation}
Here  $C>0$ is a constant 
 depending only on  $p$, $\Lambda$ and $n$.
\end{theorem}
\begin{proof}
The arguments follow similar lines as in the proof of 
Theorem~\ref{first-main-result} using   Lemma~\ref{G-second-density-est} and Lemma~\ref{G-W^{1,2}-est} in place of Lemma~\ref{second-density-est} and  Lemma~\ref{W^{1,2}-est}. Therefore, we will 
only present the main points.

Let  $N>1$ be as in  Lemma~\ref{G-second-density-est}, and let $q:=p/2>1$ . We choose  $\e=\e(p,\Lambda, n)>0$ be such that
\[
\e_1 \eqdef 2 (10)^{n+2}\e = \frac{1}{2 N^q}, 
\]
and let $\delta=\delta(p,\Lambda, n)$ be the corresponding constant given by Lemma~\ref{G-second-density-est}.
Assuming for a moment that $u$  satisfies 
\begin{equation}\label{g-initial-distribution-condition}
\big| \{Q_1^+: 
\M_{Q_3^+}(|\nabla u|^2)> N \}\big| \leq \e |Q_1^+|.
\end{equation}
We first consider the case  $\theta \leq\lambda$.
Then it follows from  Lemma~\ref{G-second-density-est} that
\begin{align}\label{g-initial-distribution-est}
\big|\{Q_1^+: \M_{Q_3^+}(|\nabla u|^2)> N\}\big|\leq \e_1  \left\{
\big|\{Q_1^+: \M_{Q_3^+}(|\nabla u|^2)> 1\}\big|
+ \big|\{ Q_1^+: \M_{Q_3^+}(c^2)> \delta \}\big|\right\}.
%\nonumber
\end{align}
 Let us iterate this estimate by considering
\[
u_1(x,t) = \frac{u(x,t)}{\sqrt{N}}, \quad c_1(x,t) = \frac{c(x,t)}{\sqrt{N}} \quad \mbox{and}\quad \lambda_1 = \sqrt{N}\lambda\geq \theta.
\]
It is easy to see that $u_1\in \hat\calW(Q_4^+)$ is a weak solution of
\begin{equation*}
\left \{
\begin{array}{lcll}
(u_1)_t  &=&  \nabla\cdot[(1+\lambda_1 u_1)\A\nabla u_1] + \theta^2 u_1(1-\lambda_1 u_1) - \lambda_1 \theta c_1 u_1  \quad &\text{in}\quad Q_4^+, \\
\begin{displaystyle}
\frac{\partial u_1}{\partial \bnu} 
\end{displaystyle}
 &=& 0 \quad &\text{on}\quad \tilde T_4, \\
u_1 & =& \psi / \sqrt{N}  \quad &\text{on}\quad \partial_b Q_4^+ \cup \partial_c Q_4^+.
\end{array}\right.
\end{equation*}
%Moreover, thanks to \eqref{g-initial-distribution-condition} we have
%\begin{align*}
%\big| \{Q_1^+: 
%\M_{Q_3^+}(|\nabla u_1|^2)> N \}\big| &= \big| \{Q_1^+: 
%\M_{Q_3^+}(|\nabla u|^2)> N^2 \}\big| \leq \e |Q_1^+|.
%\end{align*}
Therefore, by applying  Lemma~\ref{G-second-density-est} for $u_1$ we obtain
\begin{align*}
\big|\{Q_1^+: \M_{Q_3^+}(|\nabla u_1|^2)> N\}\big|\leq \e_1 \left( 
\big|\{Q_1^+: \M_{Q_3^+}(|\nabla u_1|^2)> 1 \}\big|
+ \big|\{ Q_1^+: \M_{Q_3^+}(|c_1|^2)> \delta \}\big|\right).
\end{align*}
We infer from this and  \eqref{g-initial-distribution-est} that
\begin{align*}\label{g-first-iteration-est}
&\big|\{Q_1^+: \M_{Q_3^+}(|\nabla u|^2)> N^2\}\big|
\leq \e_1^2 
\big|\{Q_1^+: \M_{Q_3^+}(|\nabla u|^2)> 1 \}\big|\\
&\qquad + \e_1^2\big|\{ Q_1^+: \M_{Q_3^+}(c^2)> \delta\} \big|+ \e_1\big|\{ Q_1^+: \M_{Q_3^+}(c^2)> \delta N\}\big|. \nonumber
\end{align*}
 By repeating the iteration, we then conclude that
\begin{align*}
&\big|\{Q_1^+: \M_{Q_3^+}(|\nabla u|^2)> N^k\}\big|\\
&\leq \e_1^k 
\big|\{Q_1^+: \M_{Q_3^+}(|\nabla u|^2)> 1 \}\big|+ \sum_{i=1}^k\e_1^i\big|\{ Q_1^+: \M_{Q_3^+}(c^2)> \delta N^{k-i}\} \big|\quad \mbox{for all } k=1,2,\dots
\end{align*}
As a consequence of this and by arguing as in the proof of 
Theorem~\ref{first-main-result},  we obtain 
\begin{align*}
\int_{Q_1^+}|\nabla u|^{2 q} \, dx dt\leq \int_{Q_1^+}\M_{Q_3^+}(|\nabla u|^2)^q \, dx dt \leq C\left( 1+ \int_{Q_1^+}\M_{Q_3^+}(c^2)^q \, dx dt \right)
\end{align*}
with the constant $C$ depending only on $p$, $\Lambda$ and $n$.  Therefore,
it follows from the strong type $q-q$ estimate for the maximal function and the fact $q=p/2$ that 
\begin{equation}\label{g-initial-desired-L^p-estimate}
\int_{Q_1^+}|\nabla u|^{p} \, dx dt\leq C\left( 1+ \int_{Q_3^+}|c|^{p} \, dx dt\right).
\end{equation}
The estimate  \eqref{g-initial-desired-L^p-estimate} was derived under the assumption that $\theta\leq \lambda$. But as in the proof of 
Theorem~\ref{first-main-result}, we deduce from \eqref{g-initial-desired-L^p-estimate} that
\begin{equation}\label{g-desired-L^p-estimate}
\int_{Q_1^+}|\nabla u|^p \, dx dt\leq C\left[ \Big(\frac{\theta}{\lambda}\Big)^p \vee 1+ \int_{Q_3^+}|c|^p \, dx dt\right]
\end{equation}
as long as $\lambda>0$ and $0<\theta\leq 1$.
We next remove the extra assumption  
\eqref{g-initial-distribution-condition} for $u$. Notice that for any $M>0$, by using the weak type $1-1$ estimate for the maximal function and Lemma~\ref{G-W^{1,2}-est} we get
\begin{align*}
&\big| \{Q_1^+: 
\M_{Q_3^+}(|\nabla u|^2)> N  M^2\}\big|
\leq \frac{C }{N M^2}\int_{Q_3^+} |\nabla u|^2 \, dxdt 
\leq \frac{C_n }{M^2}\int_{Q_4^+} u^2 \, dxdt.
\end{align*}
Thus, if we let
\[
\bar{u}(x,t) = \frac{u(x,t)}{M}\quad \mbox{with}\quad M^2= \frac{C_n  \|u\|_{L^2(Q_4^+)}^2}{ \e |Q_1|}
\]
then 
$\big| \{Q_1^+: 
\M_{Q_3^+}(|\nabla \bar{u}|^2)> N\}\big|
\leq \e |Q_1^+|$. Hence we can apply \eqref{g-desired-L^p-estimate} for $\bar{u}$ with $c$ and $\lambda$ being replaced by $\bar c= c/M$ and $\bar \lambda = \lambda M$. By reversing back to the  functions $u$ and $c$, we obtain \eqref{g-main-estimate}.
\end{proof}

\begin{remark}\label{rm:regu-parabolic-cube}
By inspection we see that the interior and boundary $W^{1,p}$-estimates (i.e. Theorem~\ref{first-main-result} and Theorem~\ref{Global-half-space}) also hold true if the parabolic cubes $K_\rho(y,s)=B_\rho(y)
\times (s-\rho^2, s]$ and
$K_\rho^+(y,s)=B_\rho^+(y)
\times (s-\rho^2, s]$ are used  in these statements instead of the centered parabolic cubes $Q_\rho(y,s)$ and $Q_\rho^+(y,s)$.
\end{remark}

\subsection{Global $W^{1,p}$-estimates  on Lipschitz domains}\label{globalW}

In this subsection we consider the  Neumann problem \eqref{vsys} and  derive  global  
$W^{1,p}$-estimates for the solution $u$. To this end, we will flatten the boundary of the domain $\Omega$ and then employ Theorem~\ref{Global-half-space}.
\begin{proof}[\textbf{Proof of Theorem~\ref{vreg}}]
For simplicity, we assume that $\alpha=1$.
In order to establish the estimates up to the top boundary of $\Omega_T$,  we are going to use
parabolic cubes $K_\rho$ instead of centered parabolic cubes $Q_\rho$.

Let $x_0\in \partial \Omega$ and $t_0\in [\bar t, T]$. Since $\Omega$ is $(\delta,R)$-Lipschitz, we may assume - upon relabeling and reorienting the coordinate axes if necessary - that
\[
\Omega\cap B_R(x_0) =\{(x', x_n)\in B_R(x_0):\, x_n > \gamma(x')\}
\]
for some  Lipschitz continuous function $\gamma:\R^{n-1}\to \R$ with $\text{Lip}(\gamma)\leq \delta$. By translating by a suitable vector, we can assume   that $x_0 = 0$. Let $\Phi: \R^n\longrightarrow \R^n$ be given by
\[
\Phi(x', x_n) := (x', x_n -\gamma(x'))
\] 
and let $\Psi(y', y_n) := \Phi^{-1}(y', y_n)=(y', y_n + \gamma(y'))$. We have  $\nabla \Phi = (\nabla \Psi)^{-1}$, and
\begin{align}\label{gradident-flatten-maps} 
\nabla \Phi=
\begin{bmatrix}
 1 & 0 & 0 &\cdots & 0 & 0\\
 0 & 1 & 0 & \cdots &0 & 0 \\
 . & . & . & \dots & . & .\\
 0 & 0 & 0 &\cdots & 1 & 0\\
 -\gamma_{x_1} & -\gamma_{x_2} & -\gamma_{x_3}& \cdots & -\gamma_{x_{n-1}} & 1 
\end{bmatrix}
\mbox{ and}\quad  
\nabla \Psi=
\begin{bmatrix}
 1 & 0 & 0 &\cdots & 0 & 0\\
 0 & 1 & 0 & \cdots &0 & 0 \\
 . & . & . & \dots & . & .\\
 0 & 0 & 0 &\cdots & 1 & 0\\
 \gamma_{y_1} & \gamma_{y_2} & \gamma_{y_3}& \cdots & \gamma_{y_{n-1}} & 1 
\end{bmatrix}.
\end{align}
Moreover,  $\Phi$ and $\Psi$ are measure-preserving transformations, that is,    $\det \nabla \Phi =\det \nabla \Psi=1$. As a consequence of \eqref{gradident-flatten-maps}, we obtain   
\begin{equation}\label{distortion-controlled}
\|\nabla \Phi\|_{L^\infty( B_R(x_0))}^2 \leq n + \|\nabla \gamma\|_{L^\infty}^2  \leq n+ \text{Lip}(\gamma)^2 \leq n+1\quad\mbox{and}\quad 
\|\nabla \Psi\|_{L^\infty(\Phi(B_R(x_0)))}^2\leq n+1.
\end{equation}

Let us choose $\rho\in (0, \bar t)$ small such that $\rho< 2 R/\sqrt{n+1}$ and $B_\rho^+ \subset \Phi(\Omega\cap B_R(x_0))$, and define
\[
\hat u(y,t) := u(\Psi(y), t ) \quad\mbox{for $y\in \overline{B_\rho^+}$ and $t\in [0, T]$}.
\]
Then $\hat u \in \hat\calW(K_\rho^+(0, t_0))$ is a weak solution of 
\begin{equation*}
\left \{
\begin{array}{lcll}
\hat u_t  &=&  \nabla\cdot[(1+\lambda \hat u) \hat \A\nabla \hat u] + \theta^2 \hat u(1-\lambda \hat u) - \lambda \theta \hat c \hat u  \quad &\text{in}\quad K_\rho^+(0, t_0), \\
\begin{displaystyle}
\frac{\partial \hat u}{\partial \bnu} 
\end{displaystyle}
&=& 0 \quad &\text{on}\quad  T_\rho\times (t_0-\rho^2, t_0], \\
\hat u(y,t) & =& u(\Psi(y), t)  \quad &\text{on}\quad \partial_b K_\rho^+(0, t_0)\cup \partial_c K_\rho^+(0, t_0)
\end{array}\right.
\end{equation*}
with
\[
\hat \A (y, t) =  \nabla \Phi(\Psi(y)) \cdot \A (\Psi(y), t)\cdot \nabla \Phi(\Psi(y))^T\quad\mbox{and}\quad 
\hat c(y, t) = c(\Psi(y), t).
\]
Here $\partial_b K_\rho^+(0, t_0) \eqdef B_\rho^+ \times \{t_0 -\rho^2\}$ and $ \partial_c K_\rho^+(0, t_0) \eqdef \partial_c B_\rho^+ \times (t_0 -\rho^2, t_0]$.
We would like to apply Theorem~\ref{Global-half-space} for $\hat u$ and so  we need to verify conditions in this theorem.
Since $\langle \hat \A (y, t) \cdot \xi, \xi \rangle 
=\langle A (\Psi(y), t)\cdot \big[\nabla \Phi(\Psi(y))^T \cdot \xi\big], \big[\nabla \Phi(\Psi(y))^T \cdot \xi\big] \rangle $ for $(y,t)\in K_\rho^+(0, t_0)$ and  $\xi\in \R^n$, we have
\begin{equation*}
\Lambda^{-1} |\eta |^2 \leq \langle \hat \A (y, t) \cdot \xi, \xi \rangle \leq \Lambda |\eta|^2,
\end{equation*}
where $\eta:=\nabla \Phi(\Psi(y))^T \cdot \xi$. Moreover, by using  \eqref{distortion-controlled} we get
$|\eta|^2 \leq |\nabla \Phi|^2 |\xi|^2\leq (n+1) |\xi|^2$ 
and
$|\xi|^2= |\nabla \Psi(y)^T \cdot \eta|^2 
 \leq |\nabla \Psi|^2 |\eta|^2\leq  (n+1) |\eta|^2$. 
Thus we conclude that
\begin{equation}\label{ellipticity-I}
[(n+1)\Lambda]^{-1} |\xi |^2 \leq \langle \hat \A (y, t) \cdot \xi, \xi \rangle \leq (n+1)\Lambda |\xi|^2\quad 
\mbox{for a.e. $ (y, t)\in K_\rho^+(0, t_0)$ and for all $\xi\in \R^n$}.
\end{equation}
We next show that the mean oscillation of $\hat \A$ is small. For this, let us write $\A=(a_{ij})$. A direct computation using 
\eqref{gradident-flatten-maps} gives
\begin{equation*}
\hat \A(y, t) = \A(\Psi(y), t) +
\begin{bmatrix}
 0 & 0  &\cdots & 0 & -\sum_{j=1}^{n-1}{a_{1 j} \gamma_{x_j}}\\
 0 & 0  & \cdots &0 & -\sum_{j=1}^{n-1}{a_{2 j} \gamma_{x_j}}\\
 . & .  & \dots & . & .\\
 0 & 0  &\cdots & 0 & -\sum_{j=1}^{n-1}{a_{(n-1) j} \gamma_{x_j}}\\
 -\sum_{i=1}^{n-1}{a_{ i 1} \gamma_{x_i}} & -\sum_{i=1}^{n-1}{a_{ i 2} \gamma_{x_i}} &  \cdots & -\sum_{i=1}^{n-1}{a_{ i (n-1)} \gamma_{x_i}} & -\sum_{j=1}^{n-1}{a_{ n j } \gamma_{x_j}}-\sum_{i,j=1}^{n-1}{a_{ i j} \gamma_{x_i} \gamma_{x_j}}
\end{bmatrix}.
\end{equation*}
Hence for fixed  $(z,s)\in \overline{K_{\frac{\rho}{4}}^+(0, t_0)}$ and $r\in (0, \rho/2]$, we have
\begin{align}\label{oscillation-est}
&\frac{1}{|K_r(z, s)|}\int_{K_r(z, s)\cap K_{\frac{3\rho}{4}}^+(0, t_0)}{|\hat \A(y, t) - \overline{\A}_{B_{r\sqrt{n+1}}(\Psi(z))\cap \Omega}(t)|^2\, dy dt}\\
&\leq 2 \frac{1}{|K_r(z, s)|}\int_{K_r(z, s)\cap K_{\frac{3\rho}{4}}^+(0, t_0)}{|\A(\Psi(y), t) - \overline{\A}_{B_{r\sqrt{n+1}}(\Psi(z))\cap \Omega}(t)|^2\, dy dt} + C(\Lambda, n) \, \text{Lip}(\gamma)^2.\nonumber
\end{align}
Since $B_r(z) \cap B_{\frac{3\rho}{4}}^+ \subset  \Phi(\Omega\cap B_R(x_0))$, we infer from  the second inequality in \eqref{distortion-controlled} that $\Psi\big(B_r(z) \cap B_{\frac{3\rho}{4}}^+ \big)\subset  B_{r\sqrt{n+1}}(\Psi(z))\cap  \Omega\cap B_R(x_0)$. Therefore, 
\begin{align*}
&\frac{1}{|K_r(z, s)|} \int_{K_r(z, s)\cap K_{\frac{3\rho}{4}}^+(0, t_0)}{| \A(\Psi(y), t) - \overline{\A}_{B_{r\sqrt{n+1}}(\Psi(z))\cap \Omega}(t) |^2\, dy dt}\\
&\leq \frac{(n+1)^{\frac{n}{2}}}{|K_{r\sqrt{n+1} }(\Psi(z), s)|} \int_{K_{r\sqrt{n+1} }(\Psi(z), s)\cap \Omega_T}{| \A(x, t) - \overline{\A}_{B_{r\sqrt{n+1}}(\Psi(z))\cap \Omega}(t) |^2\, dx dt}\leq (n+1)^{\frac{n}{2}} [\A]_{BMO(R, \Omega_T)}, 
\end{align*}
where we have used the fact $r\sqrt{n+1}\leq \rho\sqrt{n+1} /2 \leq R$ to achieve the  last inequality. Plug this into \eqref{oscillation-est} we arrive at
\begin{align*}
\frac{1}{|K_r(z, s)|}\int_{K_r(z, s)\cap K_{\frac{3\rho}{4}}^+(0, t_0)}{|\hat \A(y, t) - \overline{\A}_{B_{r\sqrt{n+1}}(\Psi(z))\cap \Omega}(t)|^2\, dy dt}
\leq 2(n+1)^{\frac{n}{2}} [\A]_{BMO(R, \Omega_T)} + C(\Lambda, n) \, \text{Lip}(\gamma)^2.
\end{align*}
It follows that
\begin{align*}
&\sup_{0<r\leq \frac{\rho}{2}}\sup_{(z,s)\in \overline{K_{\frac{\rho}{4}}^+(0, t_0)}} \frac{1}{|K_r(z, s)|}\int_{K_r(z, s)\cap K_{\frac{3\rho}{4}}^+(0, t_0)}{|\hat \A(y, t) - \overline{\hat\A}_{B_r(z)\cap B_{\frac{3\rho}{4}}^+}(t)|^2\, dy dt}\\
&\leq 2(n+1)^{\frac{n}{2}} [\A]_{BMO(R, \Omega_T)} + C(\Lambda, n) \, \text{Lip}(\gamma)^2
\leq C'(\Lambda, n)\delta.\nonumber
\end{align*}
Thus we can apply a rescaled version of Theorem~\ref{Global-half-space} (see  Remark~\ref{rm:regu-parabolic-cube}) to get
\begin{equation*}
\int_{K_{\frac{\rho}{4}}^+(0, t_0)}|\nabla \hat u|^p \, dy dt\leq C\left\{  \Big(\frac{\theta}{\lambda} \vee  \|\hat u\|_{L^2(K_\rho^+(0, t_0))}\Big)^p + \int_{K_\rho^+(0, t_0)}|\hat c|^p \, dy dt\right\}.
\end{equation*}
By changing variables back and using the fact 
\begin{align*}
|\nabla u(\Psi(y), t)| &
= \big|[\nabla \Phi(\Psi(y))]^T [\nabla \Psi(y)]^T \nabla u(\Psi(y), t)\big| = \big|[\nabla \Phi(\Psi(y))]^T  \nabla \hat u(y, t)\big|\\
&\leq \sqrt{n+1} |\nabla \hat u(y, t)|,
\end{align*}
we  obtain
\begin{equation*}
\int_{\Psi(B_{\frac{\rho}{4}}^+) \times (t_0 -\frac{\rho^2}{16}, t_0 ]}|\nabla  u|^p \, dx dt\leq C\, (n+1)^{\frac{p}{2}} \left\{  \Big(\frac{\theta}{\lambda} \vee  \| u\|_{L^2(\Omega_T)}\Big)^p + \int_{\Omega_T}| c|^p \, ds dt\right\}.
\end{equation*}
But as  $\Omega\cap B_{\frac{\rho}{4\sqrt{n+1}}}(x_0)\subset\Psi(B_{\frac{\rho}{4}}^+)$ by the first inequality in \eqref{distortion-controlled}, we thus conclude that 
\begin{equation}\label{piece-boundary-est}
\int_{\Omega\cap B_{\frac{\rho}{4\sqrt{n+1}}}(x_0) \times (t_0 -\frac{\rho^2}{16}, t_0 ]}|\nabla  u|^p \, dx dt\leq C \left\{  \Big(\frac{\theta}{\lambda} \vee  \| u\|_{L^2(\Omega_T)}\Big)^p + \int_{\Omega_T}| c|^p \, ds dt\right\}
\end{equation}
for any $x_0\in \partial \Omega$ and $t_0\in [\bar t, T]$.

The rest of the proof is standard. We cover the region $\Omega\times [\bar t, T]$ by a  suitable finite family  of  parabolic cubes $K_{\rho_i}(x_i, t_i)=B_{\rho_i}(x_i)\times (t_i -\rho_i^2, t_i]$ with $(x_i, t_i)\in \overline{\Omega}\times [\bar t, T]$ whose members are either interior cubes (i.e. $B_{6 \rho_i}(x_i)\subset\Omega$) or cubes centered at a point on $\partial\Omega\times [\bar t, T]$. For  each of these cubes, we can either apply a rescaled version of Theorem~\ref{first-main-result} (see also Remark~\ref{rm:regu-parabolic-cube}) or use
the estimate \eqref{piece-boundary-est}. The desired estimate \eqref{global-estimate-Lipschitz} then follows by adding up the resulting inequalities.
\end{proof}
%=============
%======================
\section{Global smooth Solutions for the SKT System}\label{sys}

We prove Theorem \ref{global-existence} in this section. Note that the equations of $u$ and $v$ can 
be written in the divergence form as 
\begin{align}\label{v-div-eqn}
u_t &= \nabla\cdot[(d_1 + 2a_{11}u + a_{12}v)\nabla u + a_{12} u \nabla v ] 
+ u(a_1 -b_1 u - c_1v),\nonumber\\
v_t &= \nabla\cdot[(d_2 + 2a_{22}v) \nabla v] + v(a_2 - c_2v) -b_2u v. 
\end{align}
For given $u_0, v_0, p_0$ as in Theorem~\ref{global-existence}, let $T =t_{max}\in (0, \infty]$ 
be the maximal time existence of the solution $u, v$ for \eqref{KST-r} as in Theorem~\ref{local-existence}. 
Let $\bar{t} \in (0, T)$ be fixed.  Then Theorem \ref{local-existence} implies that  there is a constant $C(\bar{t})>0$ such that 
\beq \label{easygrad} 
\sup_{0<t<\bar{t}} \Big[\int_{\Omega}|\nabla u(x,t)|^{p_0} dx + \int_\Omega |\nabla v(\cdot, t)|^{p_0} dx  \Big]\leq C(\bar{t}).
\eeq
We assume that $T < \infty$ and derive a contradiction to \eqref{global-cond}  by establishing  \begin{equation}\label{contra.global}
\sup_{\bar{t} < t < T} \Big[\norm{u(\cdot, t)}_{W^{1,p_0}(\Omega)} 
+ \norm{v(\cdot, t)}_{W^{1,p_0}(\Omega)}\Big] < \infty. 
\end{equation}
In the following estimates, all constants $C$ are positive, continuously 
dependent on $T$ and may change 
from line to line. They also can depend on $\bar{t}$, $\Omega$, $\norm{u_0}_{W^{1,p_0}(\Omega)}, \norm{v_0}_{W^{1,p_0}(\Omega)}$ and the coefficients in the system \eqref{KST-r}, but we may not explicitly specify such dependence. We begin with a lemma which is a consequence of the maximum principle.
\begin{lemma}\cite[Lemma 2.1]{LNW} \label{cr-mp} The solution $v$ of \eqref{KST-r} satisfies 
\[
0 \leq v (x,t) \leq M_0:= \max\Big \{\frac{a_2}{c_2}, \max_{\overline{\Omega}} v_0 \Big \}, \quad 
\forall \ (x,t) \in \Omega \times [0, T).
\]
\end{lemma}
The next lemma is an important consequence of Theorem \ref{vreg}.
\begin{lemma} \label{Lp-v} For each $p \in (2, p_0]$,   there exists 
a constant $C = C(p, T)>0$ such that
\beq \label{grad-v}
\norm{\nabla v}_{L^p(\Omega_T) } \leq C\Big[1 + \norm{u}_{L^p(\Omega_T)} \Big]. 
\eeq 
\end{lemma}
\begin{proof} Without loss of generality, we assume that $a_2 =1$. Let $\lambda = M_0^{-1}$, where 
$M_0$ is defined in Lemma \ref{cr-mp}. Then we rewrite the equation \eqref{v-div-eqn} as
\[
v_t = \nabla\cdot[(1+ \alpha \lambda)d_2 \nabla v ] + v(1-\lambda v) - c v,
\]
where
\[
c(x,t) = [c_2 -\lambda] v(x,t) + b_2 u(x,t) \geq 0, \quad \text{and} \quad \alpha = \frac{2a_{22}}{\lambda d_2}.
\]
Since  $\|c\|_{L^p(\Omega_T)}\leq C\big(1+\|u\|_{L^p(\Omega_T)}\big)$,  inequality  \eqref{grad-v} follows from estimate \eqref{global-estimate-Lipschitz} of Theorem~\ref{vreg} and \eqref{easygrad}.
\end{proof}
Our next goal is to derive an $L^l$-estimate for $u$ assuming that $\nabla v \in L^p(\Omega_T)$ 
for some $l=l(p)>p>2$.  For this, we first recall  Lemma~3.2 of \cite{Tuoc}.
%=======
\begin{lemma}\cite[Lemma 3.2]{Tuoc} \label{Tuoc-Lp} Let $p>2$ and assume 
there is a constant $M(p,T) <\infty$ such that
\[ \norm{\nabla v}_{L^p(\Omega_T)} \leq M(p,T).  \]
Then, for each $q >1$, there is a constant $C = C(q, T, M)$ such that for every $T_1 \in (0, T]$
\[
\norm{u(\cdot, t)}_{L^q(\Omega)}^q + \norm{\nabla (u^{q/2})}_{L^2(\Omega_{T_1})} ^2
+\norm{\nabla(u^{(q+1)/2})}_{L^2(\Omega_{T_1})}^2 \leq C \Big[1 + \norm{u}_{L^{\frac{p(q-1)}{p-2}}(\Omega_{T_1})}^{q-1}\Big], \quad \forall t \in (0, T_1).
\]
\end{lemma}
For each number $a \in \mathbb{R}$, we write $a_+ = \max\{ a, 0\}$. The following lemma is one of our main ingredients for the bootstrap argument.
%===========
\begin{lemma} \label{u-Lp} Let $p > 2$ and assume there is a constant $M(p,T) <\infty$ such that
\[ \norm{\nabla v}_{L^p(\Omega_T)} \leq M(p,T).  \]
Then, for every $q \in \Big (1, \frac{n(p-1)}{(n+2-p)_+} \Big]$ 
and $l \in \Big (1,  \frac{p(n+1)}{(n+2-p)_+} \Big]$ with 
$q, l \not=\infty$, there exists a positive constant $C = C(l, q,T, M)$ such that
\begin{equation} \label{abc}
\sup_{0 \leq t \leq T}\int_{\Omega} u(x,t)^q dx + \int_{\Omega_T}|\nabla [u^{\frac{q+1}{2}}]|^2 dx dt 
\leq C \quad \text{and} \quad 
 \norm{u}_{L^l(\Omega_T)} \leq C.\end{equation}
\end{lemma}
\begin{proof} We follow the approach in \cite{CLY, Tuoc}.  Let $w = u^{(q+1)/2}$. For  any number  $T_1$  in $(0, T)$, define
\[
 E(T_1)= \sup_{0 \leq t \leq T_1}\int_{\Omega} u(x,t)^q dx + \int_{\Omega_{T_1}} |\nabla [u^{\frac{q+1}{2}}]|^2 dx dt 
= \sup_{0 \leq t \leq T_1}\int_\Omega w^{\frac{2q}{q+1}} dx + \int_{\Omega_{T_1}} |\nabla w|^2 dx dt.
\]
It follows from Lemma~\ref{Tuoc-Lp} that
\begin{equation} \label{E-ineq}
E(T_1) + \norm{\nabla (u^{q/2})}_{L^2(\Omega_{T_1})}^2 \leq C\Big [1+
\norm{w}_{L^{\frac{(q-1)\bar{p}}{q+1}}(\Omega_{T_1})}^{\frac{2(q-1)}{q+1}}\Big], \quad
\text{where} \quad \bar{p} = \frac{2p}{p-2}.
\end{equation}
Now, let 
$
\bar{q} = 2 + \frac{4q}{n(q+1)}.
$
By Gagliardo--Nirenberg's inequality, there exists a constant $C>0$ depending on 
$n, \bar{q}$ and $\Omega$ such that
\begin{equation} \label{GL}
\norm{w(\cdot, t)}_{L^{\bar{q}}(\Omega)} 
\leq C\left\{\norm{\nabla w(\cdot, t)}_{L^2(\Omega)}^{\frac{2}{\bar{q}}} \norm{w(\cdot, t)}_{L^{\frac{2q}{q+1}}(\Omega)}^{\frac{4q}{n(q+1)\bar{q}}} + \norm{w(\cdot, t)}_{L^1(\Omega)} \right \}.
\end{equation}
By integrating the equation of $u$ and using Gronwall's inequality, we note that
\[ 
\sup_{0 \leq t \leq T} \int_\Omega u(x,t) dx \leq C(T). 
\]
Then the interpolation inequality yields
\[
\int_{\Omega} w(x,t) dx \leq \norm{u(\cdot, t)}_{L^1(\Omega)}^{\frac{\lambda (q+1)}{2}} 
\norm{w(\cdot, t)}_{L^{\bar{q}}(\Omega)}^{1-\lambda} 
\leq C(T) \norm{w(\cdot, t)}_{L^{\bar{q}}(\Omega)}^{1-\lambda} \quad \text{with}
 \quad \frac{ 1 -\lambda}{\bar{q}} + \frac{\lambda}{\frac{2}{q+1}} =1.
\]
This, together with \eqref{GL} and Young's inequality imply that
\begin{equation*} 
\norm{w(\cdot, t)}_{L^{\bar{q}}(\Omega)} 
\leq C\left\{\norm{\nabla w(\cdot, t)}_{L^2(\Omega)}^{\frac{2}{\bar{q}}} \norm{w(\cdot, t)}_{L^{\frac{2q}{q+1}}(\Omega)}^{\frac{4q}{n(q+1)\bar{q}}} + 1 \right \}.
\end{equation*}
Therefore,
\begin{equation} \label{GL-1}
\norm{w}_{L^{\bar{q}}(\Omega_{T_1})}
\leq C \left[\norm{\nabla w}_{L^2(\Omega_{T_1})}^{\frac{2}{\bar{q}}} 
\Big(\sup_{0<t < T_1}\norm{w(\cdot, t)}_{L^{\frac{2q}{q+1}}(\Omega)} \Big) ^{\frac{4q}{n(q+1)\bar{q}}}
+  1 \right].
\end{equation}
Also, since $ q \in \Big (1, \frac{n(p-1)}{(n+2-p)_+} \Big]$ and $q \not =\infty$, we see that
\[
(n\bar{p} - 2n -4) q \leq n(\bar{p} +2) \quad \text{and} \quad \frac{\bar{p}(q-1)}{q+1} 
\leq \bar{q}. 
\]
From this and \eqref{GL-1}, we obtain
\begin{equation} \label{w-ineq}
\norm{w}_{L^{\frac{(q-1)\bar{p}}{q+1}}(\Omega_{T_1})} 
\leq C \norm{w}_{L^{\bar{q}}(\Omega_{T_1})} \leq 
C\Big[1 + E(T_1)^{\frac{2}{n\bar{q}} + \frac{1}{\bar{q}}} \Big].
\end{equation}
Hence, it follows from this last inequality and \eqref{E-ineq} that 
\begin{equation} \label{mu-E.def}
E(T_1) \leq C[ 1 + E(T_1)^\mu], \quad \text{with} \quad
\mu =\frac{2(q-1)}{\bar{q}(q+1)} \left( \frac{2}{n}  + 1 \right) \quad \text{and} \quad C = C(T).
\end{equation}
A simple calculation shows $\mu <1$. 
Because of this and the fact $E(T_1)$ is finite, we infer from \eqref{mu-E.def} that there exists a constant $C(T)$ such that 
\begin{equation} \label{E-fin}
E(T_1) \leq C(T), \quad \forall \ T_1 \in (0, T).
\end{equation}
By passing $T_1\to T^-$, we obtain the first inequality of \eqref{abc}. The second inequality of \eqref{abc} follows directly from \eqref{w-ineq} and \eqref{E-fin}, and again, passing $T_1\to T^-$. The proof is therefore complete.
\end{proof}

To initiate our iteration process, we start with the following $L^4$-estimate for $\nabla v$.

\begin{lemma}\cite[Lemma 3.1]{Tuoc}  \label{L4-v} There exists a constant $C$ depending on $T$, 
the coefficients in the system \eqref{KST-r} and the initial data $u_0, v_0$ such that
\[ \norm{\nabla v}_{L^4(\Omega_T)} \leq C.\]
\end{lemma} 
%====

We now can combine the previous results to improve the regularity of $u$ and $v$.

\begin{lemma} \label{reg-uv-suffi} There exists a constant $C$ depending on $T$ such that 
\begin{equation} \label{reguv2}
\sup_{0 < t <T}\int_{\Omega} u(x,t)^{p_0} dx + \int_{\Omega_T}|\nabla v(x,t)|^{p_0} dx dt \leq C.
\end{equation}
Moreover, for all finite $p \in \Big (2, \frac{p_0(n+1)}{(n+2-p_0)_+}\Big]$, there exists $C = C(p, T)$ such that
\beq\label{Lpu2}
\int_{\Omega_T} u(x,t)^p dx dt \leq C(p, T).
\eeq
\end{lemma}
\begin{proof}  
Since \eqref{Lpu2} follows \eqref{reguv2} and the second inequality of \eqref{abc} in Lemma \ref{u-Lp}, it suffices to prove \eqref{reguv2} only. 
Let $l_1 = 4$. 

\textit{Step 1.} If $l_1 \geq \min\{p_0, n+2\}$, then \eqref{reguv2} can be obtained directly  from Lemmas \ref{Lp-v} and \ref{u-Lp}.  Now, we consider the case $l_1 < \min\{p_0, n + 2\}$. Note in this case that $n>2$.  We then infer from Lemmas \ref{u-Lp} and \ref{L4-v} that
\[
\int_{\Omega_T}|u(x,t)|^{l_2}dxdt \leq C(T)\quad \text{with} \quad l_2  = \frac{l_1(n+1)}{(n+2 -l_1)_+}= \frac{l_1(n+1)}{n-2}.\]

\textit{Step 2.}  If $l_2 \geq \min\{p_0, n+2\}$, we can use Lemmas \ref{Lp-v} and \ref{u-Lp} to obtain \eqref{reguv2} and we then stop.  We therefore only consider the case  $l_2 < \min\{p_0, n+2\}$. By Lemma \ref{Lp-v}, we 
see that
\[
\int_{\Omega_T} |\nabla v(x,t)|^{l_2} dxdt \leq C(T) < \infty.
\]
From this and Lemma \ref{u-Lp}, we have 
\[ \int_{\Omega_T} u^{l_3}(x,t) dx dt\leq C(T) < \infty  \quad \text{with} \quad l_3  = \frac{l_2(n+1)}{n+2 - l_2}.\]
Observe that 
\[l_3 = \frac{l_2(n+1)}{n+2 - l_2} > \frac{l_2(n+1)}{n-2} = l_1 \Big( \frac{n+1}{n-2}\Big)^2. \]

We will repeat this procedure.
For $i\ge 3$, define $\begin{displaystyle}
l_{i+1} = \frac{l_i(n+1)}{n+2-l_i}.                       
                      \end{displaystyle}
$
Then the sequence $\{l_i\}_{i=1}^\infty$ is strictly increasing and 
\[
l_{i+1} \ge  l_1\Big(\frac{n+1}{n-2}\Big)^i \quad 
\forall\ i \ge 1.
\]
Hence $\lim_{i\to\infty}l_i=\infty$.
Let $k$ be the smallest integer in $[2,\infty)$ such that $l_k \geq \min\{p_0, n+2\}$.

We repeat Step 2 above with $l_i$, for $i=3,\ldots,k-1$, to arrive at Step $(k-1)$
and obtain 
\[
 \int_{\Omega_T}|\nabla v(x,t)|^{l_i} dx dt  \leq C \quad \forall \ i = 1, \cdots, k-1, \quad \text{and} \quad 
 \int_{\Omega_T} |u(x,t)|^{l_k} dx dt \leq C.
\]
Since $l_k \geq \min \{p_0, n+2\}$, we, again,  obtain \eqref{reguv2}  from Lemmas \ref{Lp-v} and \ref{u-Lp}. The proof is complete.
\end{proof}

The next estimate for $\nabla v$ is crucial for obtaining 
the boundedness of $u$.
\begin{lemma} \label{p1-lemma}There exists $p_1> n+2$ and a constant $C(T)>0$ such that
\beq\label{gvC}
\norm{\nabla v}_{L^{p_1}(\Omega \times [\bar{t}, T])} \leq C(T).
\eeq
\end{lemma}
\begin{proof} If $n =2$, from Lemmas \ref{u-Lp} and \ref{L4-v}, we obtain
\[
\norm{u}_{L^{p}(\Omega_T)} \leq C(p,T), \quad \forall \ p \in (2,\infty).
\]
This together with Theorem \ref{vreg} imply that
\[
\norm{\nabla v}_{L^p(\Omega \times [\bar{t}, T])} \leq 
C\Big[ 1 + \norm{u}_{L^p(\Omega_T)}  \Big] \leq C(p, T), \quad \forall \ p \in [2, \infty).
\]
Hence we obtain \eqref{gvC} for $n =2$. Now, consider $n >2$. Since $p_0>n$ and $n\geq 3$, a simple calculation gives 
\[
\frac{p_0(n+1)}{(n+2 -p_0)_+} > n+2.
\]
Therefore, it follows from Lemma \ref{reg-uv-suffi} that there exists $p_1 > n+2$ such that 
\[
\norm{u}_{L^{p_1}(\Omega_T)} \leq C(T).
\]
Then applying Theorem \ref{vreg} again, we obtain
\[
\norm{\nabla v}_{L^{p_1}(\Omega \times [\bar{t}, T]} \leq 
C\Big[1 + \norm{u}_{L^{p_1}(\Omega_T)} \Big] \leq C(T),
\]
which proves \eqref{gvC}.
\end{proof}
We now show that $u$ is bounded:
\begin{lemma} \label{u-infty}There exists a constant $C(T) >0$ such that
\begin{equation} \label{eqn-Lu-in}
\norm{u}_{L^\infty(\Omega_T)} \leq C(T).
\end{equation}
\end{lemma}
\begin{proof} From Theorem \ref{local-existence}, we have $u \in C([0, \bar{t}], W^{1,p_0}(\Omega))$ with $p_0>n$. By Morrey's imbedding theorem, there exists $\bar C_0=\bar C_0(\bar{t}) >0$ such that
\begin{equation} \label{Lu-infty1}
\norm{u}_{L^\infty(\Omega \times [0, \bar{t}])} \leq\bar C_0.
\end{equation}
We thus only need to prove that $u$ is bounded in $\Omega\times [\bar{t}, T]$. For each $t_1 \in (\bar{t}, T]$, and each $ k >\bar C_0$, denote  $W_k(x,t)=\max\{u(x,t) - k, 0\}$ and
\[
 \Omega_{t_1}(k)=\{(x,t) \in \Omega \times [\bar{t}, t_1]: u(x,t) >k \}.
\]
We write the equation of $u$ as 
\begin{equation} \label{A.u-eqn}
u_t = \nabla\cdot[A(x,t)\nabla u]  + a_{12}\nabla\cdot[u \nabla v] + f(x,t) \quad \text{in }\Omega_T,
\end{equation}
with homogeneous Neumann boundary condition, where 
\begin{equation} \label{Af.def}
A(x,t) = d_1 + 2 a_{11}u(x,t) + a_{1,2}v(x,t), \quad 
f(x,t) = u(x,t)[a_1 - b_1u(x,t) -c_1v(x,t)].
\end{equation}
We note 
%in the equation \eqref{A.u-%eqn} 
that $A(x,t)$ is bounded below by $d_1>0$, and not known to be bounded above. However, we can follow the De Giorgi's iteration technique \cite[Theorem 7.1, p. 181]{La} to prove \eqref{eqn-Lu-in}. By multiplying the equation \eqref{A.u-eqn} with $W_k$ and using the integration by parts, we 
obtain
\begin{equation*} 
\begin{split}
& \frac{1}{2}\int_{\Omega} W_k^2(x,t) dx + d_1 \int_{\Omega_{t_1}(k)} |\nabla W_k|^2dxdt  
+ a_{11} \int_{\Omega_{t_1}(k)} |\nabla W_k|^2 u dx dt \\
& \leq 2 a_{12} \int_{\Omega_{t_1}(k)} |\nabla v| |\nabla W_k| u dx dt + 2 a_1 \int_{\Omega_{t_1}(k)(k)} u W_k dxdt \\
& \leq \frac{d_1}{2} \int_{\Omega_{t_1}(k)} |\nabla W_k|^2 dx dt 
+ C\int_{\Omega_{t_1}(k)}\Big [|\nabla v|^2 +1\Big] \Big[W_k^2 + k^2 \Big] dxdt, \quad 
\forall \ t \in [\bar{t},t_1].
\end{split}
\end{equation*}
Therefore, there exists a constant $C>0$ depending only on the coefficients $d_1$ and $a_1, a_{11}, a_{12}$ 
such that 
\begin{equation} \label{Giorgi-1}
\begin{split}
\norm{W_k}_{V_2}^2 &:= \sup_{0 < t< t_1}\int_\Omega W_k^2(x,t)dxdt + \int_{\Omega\times(\bar t,t_1)}|\nabla W_k(x,t)|^2 dxdt \\
&\leq C \int_{\Omega_{t_1}(k)} \Big[|\nabla v(x,t)|^2 +1 \Big]\Big[W_k^2 + k^2 \Big] dxdt.
\end{split}
\end{equation}
We estimate the right hand side of \eqref{Giorgi-1}.  Let $p_1$ be as 
in Lemma \ref{p1-lemma}. Then by H\"{o}lder's inequality and Lemma \ref{p1-lemma}, we have
\beq\label{RHS}
\begin{split}
  \int_{\Omega_{t_1}(k)}[|\nabla v|^2 +1][W_k^2 +k^2] dxdt 
  &\leq 
 \norm{\,|\nabla v| +1}_{L^{p_1}(\Omega_{t_1}(k))}^2\norm{W_k +k}_{L^{\frac{2p_1}{p_1-2}}(\Omega_{t_1}(k))}^2 \\
& \leq C(T) \left[\norm{W_k}_{L^{\frac{2p_1}{p_1-2}}(\Omega_{t_1}(k))}^2 + k^2\mu(k)^{\frac{p_1-2}{p_1}}  \right],
\end{split}
\eeq
where $\mu(k) := |\Omega_{t_1}(k)|$.
Observe that $p_1 > n+2$ implies
\[
\frac{p_1}{p_1-2}< \frac{n+2}n,\quad \text{hence}\quad
\delta:= \frac{p_1-2}{p_1} - \frac{n}{n+2} > 0.
\]
Therefore,  we can apply H\"{o}lder's inequality and then the parabolic imbedding theorem (\cite[(3.4), p. 75]{La} to infer that
\beq\label{imbed2}
 \norm{W_k}_{L^{\frac{2p_1}{p_1-2}}(\Omega_{t_1}(k))}^2 \leq \norm{W_k}_{L^{\frac{2(n+2)}{n}}(\Omega_{t_1}(k))}^2  \mu(k)^{\delta} 
\leq C(T) \norm{W_k}_{V_2}^2  \mu(k)^{\delta}.
\eeq
By \eqref{Giorgi-1}, \eqref{RHS} and \eqref{imbed2}, there exists $C_1(T) >0$ such that 
\begin{equation} \label{Giorgi-2}
\norm{W_k}_{V_2}^2 \leq C_1(T) \Big[\norm{W_k}_{V_2}^2  \mu(k)^{\delta} 
+  k^2\mu(k)^\frac{p_1-2}{p_1} \Big].
\end{equation}
Let $t_1\in (\bar t,T]$ satisfy 
%\begin{equation} \label{t1-def} 
$C_1(T) \left [|\Omega | (t_1 -\bar{t})  \right]^{\delta} \le  1/2$. 
%\end{equation}
Then the inequality \eqref{Giorgi-2} yields 
\begin{equation} \label{Giorgi-3}
\begin{split}
\norm{W_k}_{V_2}^2 \leq 2C_1(T) k^2 \mu(k) ^{\frac{p_1-2}{p_1}}.
\end{split}
\end{equation}
From this and the standard iteration technique (\cite[Theorem 6.1, p. 102]{La}), we deduce that 
\beq\label{step0}
\sup_{\Omega \times [\bar{t}, t_1]} u(x,t) \leq \bar{C}_1 .
\eeq
where $\bar{C}_1 = \bar{C}_1(T)>0$.

Now, partition the interval $[\bar t,T]$ evenly with
\[
\bar{t}=t_0 < t_1 < t_2 < \cdots < t_{N+1} = T,\quad t_j=t_0 + jh,
\]
where the time step $h>0$ is chosen to satisfy $  C_1(T) [ h \,|\Omega| ]^{\delta} \le 1/2$. Repeating  the proof of \eqref{step0} on each time interval $[t_j, t_{j+1}]$ for $j =1, 2,\cdots, N$, we obtain
\beq\label{stepj}
\sup_{\Omega \times [t_j,t_{j+1}]} u(x,t) \leq \bar{C}_{j+1}, 
\eeq
where $\bar{C}_{j+1}=\bar{C}_{j+1}(T)>0$ also depends on the preceding bound $\bar{C}_{j}$ on $[t_{j-1},t_j]$. (In the proof, we replace $\bar{C}_0$ by $\bar{C}_{j}$.)
Therefore, we obtain
\begin{equation} \label{Lu-infty2}
\sup_{\Omega \times [\bar{t}, T]} u \leq C(T).
\end{equation}
The estimate \eqref{eqn-Lu-in} then follows from \eqref{Lu-infty1} and \eqref{Lu-infty2}.
\end{proof}

We are now ready to prove Theorem~\ref{global-existence}.

\begin{proof}[\textbf{Proof of Theorem~\ref{global-existence}}] 
For all $(x,t) \in \Omega_T$, as in the proof of Lemma \ref{u-infty}, let $A(x,t)$, $f(x,t)$ be defined by \eqref{Af.def} and let 
\[
B(x,t)  = d_2 + 2a_{22}v(x,t), 
\quad
g(x,t) = v(x,t)[a_2 -b_2 u(x,t) - c_2v(x,t)].
\]
From Lemmas \ref{cr-mp} and \ref{u-infty}, there are constants $\Lambda>0$ and $C>0$ 
such that
\begin{equation} \label{para-uv}
\Lambda^{-1} \leq A(x,t) \leq \Lambda \quad\text{and}\quad \Lambda^{-1} \leq B(x,t) \leq \Lambda 
\quad \forall \ (x,t) \in \Omega_T,
\end{equation}
%and
\begin{equation} \label{f-g.bound}
\norm{f}_{L^\infty(\Omega_T)} + \norm{g}_{L^\infty(\Omega_T)} \leq C.
\end{equation}
We rewrite the equation \eqref{v-div-eqn} as
\[
v_t = \nabla \cdot[B(x,t)\nabla v] + g.
\]
From \eqref{para-uv}, \eqref{f-g.bound}, and the classical H\"{o}lder regularity theory 
(\cite[Theorem~10.1, p.~204]{La}, \cite[Theorem 1.3, Remark~1.1, p. 43]{DiB}), 
there exist $\alpha_1 \in(0,1)$ and $C(T) >0$ such that 
\begin{equation} \label{v-holder}
   \norm{v}_{C^{\alpha_1, \frac{\alpha_1}{2}}(\overline{\Omega}_T)} \leq C(T).
\end{equation}
Next, we rewrite the equation of $u$ as
\begin{equation} \label{u-v-short}
u_t = \nabla \cdot[A(x,t) \nabla u]  + \nabla\cdot \vec{F}_1 + \nabla \cdot \vec{F}_2+ f,
\end{equation}
where
\[
\vec{F}_1(x,t) = a_{12}\, \chi_{[0,\bar{t}]}(t) u(x,t) \nabla v(x,t) \quad \text{and}\quad
\vec{F}_2(x,t) = a_{12}\, \chi_{[\bar{t}, T]}(t) u(x,t) \nabla v(x,t).
\]
Here, $\chi_I$ denotes the characteristic function of the set $I\subset \mathbb R$. From 
Theorem \ref{local-existence}, Lemmas \ref{p1-lemma} and \ref{u-infty}, we see that
\[
\norm{\vec{F}_1}_{L^\infty([0, T]; L^{p_0}(\Omega))} \leq C(\bar{t}) \quad\text{and}\quad 
\norm{\vec{F}_2}_{L^{p_1}(\Omega_T)} \leq C(T),
\]
where $p_1>n+2$ is given by Lemma~\ref{p1-lemma}.
Therefore, we can again apply the classical H\"{o}lder regularity theory 
(\cite[Theorem~10.1, p.~204]{La}, \cite[Theorem 1.3, Remark 1.1, p. 43]{DiB}) to the equation 
\eqref{u-v-short} to infer that there are $\alpha_2\in(0,1)$ and $C(T)>0$ such that
\[
\norm{u}_{C^{\alpha_2, \frac{\alpha_2}{2}}(\overline{\Omega_T})} \leq C(T).
\] 
Let $w = (d_2 +a_{22}v) v$, then $w$ solves
\[
w_t = B(x,t)\Delta w + B(x,t) g(x,t) \quad \text{in }\Omega_T
\]
with homogeneous Neumann boundary condition. Since all of the 
coefficients in the equation of $w$ are H\"{o}lder continuous, we apply the Schauder estimate (\cite[Theorem 5.3, p. 320--321]{La}) to obtain 
\[
\norm{v}_{C^{2+\beta, \frac{2+\beta}{2}}(\overline{\Omega} \times [\bar{t} ,T])} \leq C(\bar{t}, T)\quad\text{fore some }\beta \in (0,1).
\]
Using this fact and, again, the Schauder estimate for the equation of 
$w_1 := (d_1 + a_{11}u + a_{12}v)u$, we find that
\[
\norm{u}_{C^{2+\mu, \frac{2+\mu}{2}}(\overline{\Omega} \times [\bar{t} ,T])} \leq C(\bar{t}, T)\quad\text{for some } \mu \in (0,1).
\]
Thus \eqref{contra.global} follows and the proof is complete.
\end{proof} 
%=========

% \begin{remark}
%  Our result can be combined with uniform Gronwall-type estimates to prove the existence of the global attractor in studying the long-time dynamics of system \eqref{KST-r} (see, e.g., \cite{LeD-Mis,LNN}).
% \end{remark}

%============

\appendix
%\myclearpage

\section{Proof of Lemma \ref{W1infty-est}}\label{apend}
\newcommand{\esssup}{\mathop{{\rm ess\_sup}}}

%\begin{proof}[Proof of Lemma \ref{W1infty-est}] 
Since $0 \leq  \bar v \leq 1$, the equation \eqref{UEQ} is uniformly parabolic. From this, the boundedness of the nonlinear term in \eqref{UEQ}, we see
 that $\bar{v}$ is locally H\"older continuous (see \cite[Theorem 6.28, p.~130]{Lib} or \cite[Theorem 1.1, p.~419]{La}). This, \cite[Theorem 3.1, p.~437]{La} and Schauder estimates for linear uniformly parabolic equations further imply that $\bar v\in C^{2,\alpha}_{loc}(Q_4)$.
Therefore, there exists some constants $\alpha\in (0,1)$ and $C>0$ depending only on $n$ such that 
\begin{equation}\label{Schauder-Ref}
\|\bar v\|_{C^{2,\alpha}(Q_{\frac72})}\leq C.
\end{equation}
Let $i=1,2,\ldots,n$ and denote $w=\bar v_{x_i}$. Taking partial derivative of \eqref{UEQ} in $x_i$, we have
\begin{equation}\label{GUED}
w_t = \nabla\cdot \big[ (1  +  \bar v) \A_0\nabla w +   w \A_0\nabla \bar v\big]+  \theta^2 (1 - 2   \bar v) w. 
\end{equation}
With \eqref{Schauder-Ref}, we can apply the well-known De Giorgi-Nash-Moser iteration technique to this 
quasilinear uniformly parabolic equation to obtain
\begin{equation}\label{AAA}
\|w\|_{L^\infty(Q_3)}\leq C_n \left(\fint_{Q_4}{|w|^2 \, dxdt}\right)^{\frac{1}{2}}.
\end{equation}
This immediately yields the inequality \eqref{regularity-reference-eq}.  
For the sake of completeness, we give the detailed proof for \eqref{AAA}.
(For further references, one can see  \cite[Theorem 8.1, p. 192, Theorem 3.1, p.~437]{La},  \cite[Theorem~2.1]{Po}, or \cite[Theorem 1.3]{T}.)  
For an $n\times n$ matrix $\A$, we denote its operator norm by $\|\A\|_{\rm op}$ and $\| \A_0\|_\infty =\esssup_{(-16,16)} \|\A_0(t)\|_{\rm op}$.
By \eqref{eigenvalues-controlled}, we have
$\|\A_0\|_\infty ,\|\A_0^{-1}\|_\infty \le \Lambda$, and $\|\A_0^{1/2}\|_\infty ,\|\A_0^{-1/2}\|_\infty\le \Lambda^{1/2}$.
Let $M_0=\| \bar v\|_{L^\infty(Q_{7/2})}$ and $M_1=\|\nabla  \bar v\|_{L^\infty(Q_{7/2})}$.
For $k\ge 0$, define 
$$w^{(k)}=\max\{w-k,0\}\quad\text{and}\quad S_k=\{(x,t) : w(x,t)>k\}.$$
Let $\zeta(x,t)$ be a cut-off function in $Q_4$. Multiplying equation \eqref{GUED} by $w^{(k)}\zeta^2$, integrating over $B_4$ and using integration by parts yield
\begin{align*}
&\frac12\frac d{dt}\int_{B_4} |w^{(k)}\zeta|^2 dx
=\int_U |w^{(k)}|^2 \zeta \zeta_t dx -\int_{B_4} [(1 +  \bar v)\A_0\nabla w+   w\A_0\nabla  \bar v]\cdot \nabla (w^{(k)}\zeta^2 )dx\\
&\quad  +\theta^2 \int_{B_4} (1 -2   \bar v) w w^{(k)}\zeta^2 dx
= \int_{B_4} |w^{(k)}|^2 \zeta \zeta_t dx  -\int_{B_4} (1 +   \bar v)\A_0\nabla w^{(k)} \cdot \nabla (w^{(k)} \zeta)^2 dx\\
&\quad\quad -\int_{B_4}   (w^{(k)}+k)\A_0 \nabla  \bar v \cdot \nabla (w^{(k)}\zeta^2) dx
 +\theta^2 \int_{B_4} (1 -2  \bar v) (w^{(k)}+k) w^{(k)} \zeta^2 dx. 
\end{align*}
We estimate
\begin{align*}
&\frac12\frac d{dt}\int_{B_4} |w^{(k)}\zeta|^2 dx
\le \int_{B_4} |w^{(k)}|^2 \zeta |\zeta_t| dx  - \int_{B_4} (1  +    \bar v) \zeta \A_0 \nabla w^{(k)}\cdot (\nabla (w^{(k)}\zeta)+w^{(k)}\nabla \zeta)dx\\
&\quad +   M_1\|\A_0\|_\infty \int_{B_4} [w^{(k)}+k] ( \zeta|\nabla (w^{(k)} \zeta)| + |w^{(k)} |\zeta |\nabla \zeta|)dx
 + \theta^2  \int_{B_4}[w^{(k)}+k] w^{(k)}\zeta^2 dx. 
\end{align*}
Note that
\begin{align*}
 \zeta \A_0 \nabla w^{(k)}\cdot (\nabla (w^{(k)}\zeta)+w^{(k)}\nabla \zeta)
&= \A_0^{1/2}(\nabla (w^{(k)}\zeta) - w^{(k)}\nabla \zeta)\cdot \A_0^{1/2}(\nabla (w^{(k)}\zeta)+w^{(k)}\nabla \zeta)\\
&= |\A_0^{1/2}\nabla (w^{(k)}\zeta)|^2 - |w^{(k)}\A_0^{1/2}\nabla \zeta|^2.
\end{align*}
Then 
\begin{align*}
&\frac12\frac d{dt}\int_{B_4} |w^{(k)}\zeta|^2 dx
\le \int_{B_4} |w^{(k)}|^2 \zeta |\zeta_t| dx  - \int_{B_4} (1  +    \bar v) (|\A_0^{1/2}\nabla (w^{(k)}\zeta)|^2 - |w^{(k)}\A_0^{1/2}\nabla \zeta|^2)dx\\
&\quad +   M_1 \|\A_0\|_\infty\int_{B_4} [w^{(k)}+k] \zeta |\A_0^{-1/2}\A_0^{1/2}\nabla (w^{(k)} \zeta)| 
+ [w^{(k)}+k] w^{(k)}  \zeta |\nabla \zeta| dx
 + \theta^2 \int_{B_4} [w^{(k)}+k] w^{(k)}\zeta^2 dx. 
\end{align*}
By Cauchy's inequality
\begin{align*}
&\frac12\frac d{dt}\int_{B_4} |w^{(k)}\zeta|^2 dx
\le \int_{B_4} |w^{(k)}|^2 \zeta |\zeta_t| dx  -  \int_{B_4} |\A_0^{1/2}\nabla (w^{(k)}\zeta)|^2 dx
+ (1  +   M_0) \int_{B_4}  |w^{(k)}\A_0^{1/2}\nabla \zeta|^2 dx\\
&\quad  + \left \{ \frac {1 }2  \int_{B_4} |\A_0^{1/2}\nabla (w^{(k)} \zeta)|^2 dx 
+\frac{   M_1^2 \|\A_0\|_\infty^2 \|\A_0^{-1/2}\|_\infty^2}{2 } \int_{B_4} \chi(S_k)  [w^{(k)}+k]^2 \zeta^2 dx \right\}\\
&\quad   +   M_1\|\A_0\|_\infty \int_{B_4} \chi(S_k) [w^{(k)}+k]^2 \zeta |\nabla \zeta|dx
 + \theta^2\int_{B_4} \chi(S_k) [w^{(k)}+k]^2\zeta^2 dx. 
\end{align*}
Therefore,
\begin{align*}
&\frac12\frac d{dt}\int_{B_4} |w^{(k)}\zeta|^2 dx +\frac {1 }2  \int_{B_4} |\A_0^{1/2}\nabla (w^{(k)} \zeta)|^2 dx \\
&\le \int_{B_4} [w^{(k)}+k]^2 \Big[ \zeta |\zeta_t| + (1  +   M_0)\Lambda  |\nabla \zeta|^2 
+\frac{ M_1^2 \Lambda^3 }{2} \zeta^2  +   M_1\Lambda  \zeta |\nabla \zeta|+ \theta^2 \zeta^2\Big] dx\\
&\le \int_{B_4} [w^{(k)}+k]^2 \Big[ \zeta |\zeta_t| + 2(1  +   M_0)\Lambda  |\nabla \zeta|^2 
+\Big(\frac{ M_1^2 \Lambda^3 }{2}  +   M_1^2\Lambda + \theta^2\Big) \zeta^2\Big] dx.  
\end{align*}
Integrating in time and taking maximum for $t\in(-16,16)$ give
\begin{align*}
&\max_{t\in [-16,16]} \int_{B_4} |w^{(k)}\zeta|^2 dx 
+  \Lambda^{-1}\int_{Q_4} |\nabla (w^{(k)} \zeta)|^2 dx dt\\
&\le  4 \int_{Q_4} \chi(S_k)  [w^{(k)}+k]^2 \Big[  \zeta |\zeta_t| + 2(1  +   M_0)\Lambda  |\nabla \zeta|^2 
+ (2 M_1^2\Lambda^3 +\theta^2) \zeta^2\Big] dxdt. 
\end{align*}
We obtain
\begin{equation}\label{localE}
\max_{t\in [-16,16]} \int_{B_4} |w^{(k)}\zeta|^2 dx 
+  \int_{Q_4} |\nabla (w^{(k)} \zeta)|^2 dx dt\\
\le   \int_{Q_4} \chi(S_k)  [w^{(k)}+k]^2 P[\zeta] dxdt,
\end{equation}
where
\begin{equation*}%\label{Pz}
P[\zeta]= 8\Lambda \Big[  \zeta |\zeta_t| + (1  +   M_0)\Lambda |\nabla \zeta|^2 
+ (M_1^2 \Lambda^3+\theta^2) \zeta^2\Big].
\end{equation*}
Let $r=2(n+2)/n$. The parabolic Sobolev embedding gives
\begin{equation}\label{parS}
 \|w^{(k)} \zeta\|_{L^r(Q_4)} \le C_0 (\max_{t\in [-16,16]} \| w^{(k)} \zeta \|_{L^2(B_4)} +\| \nabla (w^{(k)} \zeta)\|_{L^2(Q_4)} ),
\end{equation}
where $C_0>0$.
Therefore, we have from \eqref{localE} that
\begin{equation}\label{prepare}
  \|w^{(k)} \zeta\|_{L^r(Q_4)} \le 2C_0 \Big(\int_{Q_4} \chi(S_k) [w^{(k)}+k]^2 P[\zeta] dx dt\Big)^{1/2}
\end{equation}

Now, we can use standard De Giorgi's iteration.
Let $K>0$ and $k_j = K(1-2^{-j})$ for $j\ge 0$. Then $k_j\nearrow K$ as $j\nearrow\infty$.
For $j\ge 0$, let $r_j=3+2^{-j-2}$ and $t_j=r_j^2$, 
then $3<r_j<7/2$ and $r_j\searrow 3$ as $j\nearrow\infty$.
Let $\zeta_j=\phi_j(t)\varphi_j(x)$, with $0\le \phi_j,\varphi_j\le 1$,
$\phi_j=1$ on $|t|<r_j^2$, $\phi_j=0$ on $|t|>r_{j-1}^2$, $\varphi_j=1$ on $|x|<r_j$, $\varphi_j=0$ on $|x|>r_{j-1}$.
In other words, $\zeta_j=1$ on $Q_{r_j}$ and ${\rm spt}\ \zeta_j \subset \bar Q_{r_{j-1}}\subset \bar Q_{r_0}  $ for $j\ge 1$.
Also,
\begin{equation}
 |\nabla \zeta_j(x,t)|\le C_1 2^j,\quad  |\zeta_{jt}|\le C_1 4^j,\quad C_1\ge 1.
\end{equation}
Then we have on $Q_4$ that
\begin{equation}\label{Pzj}
 |P[\zeta_j](x,t)|\le 8\Lambda C_1^2 4^j [ 1+(1+M_0)\Lambda +(M_1^2\Lambda^3+\theta^3)]\le C_2^2 4^j, 
\end{equation}
where $C_2>0$ is defined by
\beq\label{C2}
C_2^2 = 8\Lambda C_1^2 \Big( 1+ (1 +  M_0+ M_1^2)\Lambda^3 +\theta^2\Big).\eeq
Let $E_{j,\ell}=\{ (x,t)\in Q_{r_\ell} : w(x,t)>k_j\}$.
For $j\ge 0$, applying \eqref{prepare} with $k=k_{j+1}$ and $\zeta=\zeta_{j+1}$ and using \eqref{Pzj} give
\begin{align*}
  \|w^{(k_{j+1})} \zeta_{j+1}\|_{L^r(Q_4)} 
&\le 2C_0 \Big(\int_{Q_4} \chi(S_{k_{j+1}})  [w^{(k_{j+1})}+k_j]^2 P[\zeta_{j+1}] dx dt\Big)^{1/2}\\
&\le 2^{j+1} C_0 C_2 \Big(\| w^{(k_{j+1})}\|_{L^2(E_{j+1,j})}+K|E_{j+1,j}|^{1/2}\Big)\\
&\le 2^{j+1} C_0 C_2 \Big(\| w^{(k_{j})}\|_{L^2(E_{j,j})}+K|E_{j+1,j}|^{1/2}\Big).
\end{align*}
On the one hand,
\begin{align*}
 \| w^{(k_{j+1})}\|_{L^2(E_{j+1,j+1})} 
&\le  \| w^{(k_{j+1})}\zeta_{j+1}\|_{L^2(E_{j+1,j+1})} 
\le \| w^{(k_{j+1})}\zeta_{j+1}\|_{L^r(E_{j+1,j+1})}  |E_{j+1,j+1}|^{1-2/r}\\
&\le \| w^{(k_{j+1})}\zeta_{j+1}\|_{L^r(Q_4)}  |E_{j+1,j}|^{1-2/r}. 
\end{align*}
On the other hand,
\begin{equation*}
 \| w^{(k_j)}\|_{L^2(E_{j,j})} \ge  \| w^{(k_j)}\|_{L^2(E_{j+1,j})} \ge (k_{j+1}-k_{j}) |E_{j+1,j}|^{1/2} 
=K 2^{-j-1} |E_{j+1,j}|^{1/2}.
\end{equation*}
Hence
$
 |E_{j+1,j}| \le  4^{j+1} K^{-2} \| w^{(k_j)}\|_{L^2(E_{j,j})}^2.  
$
Combing the above gives
\begin{align*}
 \| w^{(k_{j+1})}\|_{L^2(E_{j+1,j+1})} 
&\le \| w^{(k_{j+1})}\zeta_{j+1}\|_{L^r(Q_4)}  |E_{j+1,j}|^{1-2/r}\\
&\le  2^{j+1} C_0 C_2 (\| w^{(k_j)}\|_{L^2(E_{j,j})} + K|E_{j+1,j}|^{1/2}) |E_{j+1,j}|^{1-2/r}\\
&\le  2^{j+1} C_0 C_2 \Big[\| w^{(k_j)}\|_{L^2(E_{j,j})} + 2^{j+1} \| w^{(k_j)}\|_{L^2(E_{j,j})}\Big] \Big[ \frac{4^{j+1}}{K^2} \| w^{(k_j)}\|_{L^2(E_{j,j})}^2\Big]^{1-\frac2r}.
\end{align*}
Thus,
\begin{align*}
 \| w^{(k_{j+1})}\|_{L^2(E_{j+1,j+1})} 
&\le \frac{2\cdot 16^{j+1} C_0C_2}{K^\mu} \| w^{(k_j)}\|_{L^2(E_{j,j})}^{1+\mu},
\end{align*}
where 
$\mu=2(1-\frac2r)>0$.
Letting  $Y_j= \| w^{(k_j)}\|_{L^2(E_{j,j})}$ then we have
\begin{equation}
 Y_{j+1}\le AB^j Y_j^{1+\mu}, \quad j\ge 0,
\end{equation}
where $B=16$ and $A=32C_0C_2/K^\mu$.
The classical result on sequences  with fast geometric convergence states that if $Y_0\le A^{-1/\mu} B^{-1/\mu^2}$ then 
\begin{equation}\label{Y}
\|w^{(K)}\|_{L^2(Q_3)}=\lim_{j\to\infty} Y_j=0. 
\end{equation}
Since $k_0=0$, then $Y_0\le \|w\|_{L^2(Q_4)}$, and we choose $K$ such that
$$\|w\|_{L^2(Q_4)} \le  (32C_0C_2/K^\mu)^{-1/\mu} 16^{-1/\mu^2}.$$
Specifically, select $K= \|w\|_{L^2(Q_4)} 16^{1/\mu^2}  (32C_0C_2)^{1/\mu}$.
Then from \eqref{Y}, $w\le K$ a.e. in $Q_3$. Replacing $w$ by $-w$ we obtain $|w|\le K$ a.e. in $Q_3$. This, \eqref{C2}, definitions of $M_0$ and $M_1$, and interior estimate \eqref{Schauder-Ref} together imply \eqref{AAA}.
The proof is complete.
%\newpage
%===============

%\myclearpage

\end{document}